\documentclass[11pt]{article}
\usepackage[utf8]{inputenc}

\usepackage{hyperref}
\usepackage{mymacros3}

\usepackage{amsmath}
\usepackage{bm}
\usepackage{amssymb}
\usepackage{color}
\usepackage{mathtools}
\usepackage{yhmath}
\usepackage[all]{xy}
\usepackage{tikz}
\usetikzlibrary{cd}

\usepackage[margin=3.5cm]{geometry}
\newcommand{\Core}{\mathrm{Core}}

\usepackage[
backend=biber,
style=ext-alphabetic,
maxnames = 10,
maxalphanames = 10,
maxcitenames  = 10,
mincitenames  = 4,
minalphanames = 4,
minnames      = 1,
sorting=nyt,
giveninits,
articlein=false,
url=true,  
doi=false,
eprint=true,
isbn=false
]{biblatex}
\title{Microlocal categories over Novikov rings}
\author{Yuichi Ike \and Tatsuki Kuwagaki}
\date{\today}

\begin{document}
\maketitle

\begin{abstract}
In this paper, we define a family of categories for each Weinstein manifold, which is an enhanced version of the category first introduced by Tamarkin. 
Using our categories, for any (possibly non-exact immersed) Lagrangian brane, we develop a theory of sheaf quantization generalizing the previous researches. 
In particular, our theory involves the notion of a sheaf-theoretic bounding cochain, which is a conjectural counterpart of the theory of Fukaya--Oh--Ohta--Ono. 
We also study several structures of our categories for sufficiently Weinstein manifolds and properties known in the classical Tamarkin category; intersection points estimates, interleaving distances, energy stability with respect to Guillermou--Kashiwara--Schapira autoequivalence, and the completeness of the distance.

We conjecture that our category is equivalent to a Fukaya category defined over the Novikov ring.
\end{abstract}

\tableofcontents
\section{Introduction}
\subsection{Background: Sheaf quantization and Fukaya category}
Sheaf-theoretic study of symplectic geometry has evolved since the work of Nadler--Zaslow~\cite{NZ} (appeared on arXiv in 2006) and Tamarkin~\cite{Tamarkin} (appeared on arXiv in 2008). 
The underlying key idea of both studies is the microlocal sheaf theory due to Kashiwara and Schapira~\cite{KS90}. 
For a sheaf $\cE$ on a manifold $M$, the theory associates a subset called \emph{microsupport} of the cotangent bundle $T^*M$. 
This is the topological counterpart of the notion of \emph{characteristic variety} in the theory of $\cD$-modules. From the viewpoint of deformation quantization, a module should be viewed as a quantization of its characteristic variety. 
Hence we can view a sheaf as a ``topological" quantization of its microsupport. 
The key philosophy is that these \emph{sheaf-theoretic quantizations} form a category closely related to Fukaya categories of Lagrangian submanifolds.

Nadler--Zaslow's work suggests that the association of microsupport can be upgraded to an equivalence between the category of constructible sheaves over $M$ and a version of Fukaya category of $T^*M$. Roughly speaking, for a Lagrangian submanifold (more precisely, brane) $L$, the corresponding constructible sheaf has its microsupport in the conical limit of $L$. Nadler--Zaslow's work motivates many recent studies. One of the most notable ones is Ganatra--Pardon--Shende's work~\cite{GPS} that proves equivalences between wrapped Fukaya categories and microsheaf categories for a large class of non-compact symplectic manifolds, which enables us to compute many Fukaya categories. 

One disadvantage of Nadler--Zaslow's work is that it is designed to be Hamiltonian-invariant from the beginning, and hence the \emph{geometric shape} of a Lagrangian submanifold is somehow lost (through taking the conical limit)~in the formulation. 
This is related to the fact that microsupport is conic by definition. 
In other words, one can sheaf-quantize only conic objects. 
This makes quantitative study or geometric study impossible/difficult.

There is a known prescription for the above problem in the context of deformation quantization~(e.g., the work of Polesello--Schapira~\cite{PS}); it is to introduce one-additional parameter $\hbar$. 
Then conic objects in $2n+2$-dimension can express non-conic objects in $2n$-dimension. Tamarkin~\cite{Tamarkin} adopted this trick in sheaf theory. 
We would like to describe it more precisely.

Let $M$ be a manifold.
We consider sheaves on $M\times \bR_t$ having the 1-dimensional additional parameter, rather than considering sheaves on $M$. 
The Tamarkin category is then defined as a quotient category of the category of sheaves on $M\times \bR_t$.
An object $\cE$ of this category has its well-defined positive part of microsupport (i.e., $\tau>0$) in $T^*M\times T^*\bR_t\cong T^*M\times \bR_t\times \bR_\tau$. We consider the twisted projection (deconification map)
\begin{equation}
    \rho\colon T^*M\times \bR_t\times \{\tau>0\}\rightarrow T^*M; (p, t, \tau)\mapsto p/\tau.
\end{equation}
Then we consider a sheaf quantization of a subset $L$ of $T^*M$ as a positively-microsupported sheaf $\cE$ with $\rho(\MS(\cE)\cap \{\tau>0\})=L$. In the past several years, this notion has been proven to be quite useful: Tamarkin~\cite{Tamarkin} proved the non-displaceability theorem, Guillermou--Kashiwara--Schapira~\cite{GKS} made it precise through quantization of Hamiltonian isotopies, Guillermou~\cite{Guillermou12} and Jin--Treumann~\cite{JinTreumann} constructed sheaf quantization for a large class of exact Lagrangian branes, Asano--Ike~\cite{AI20} introduced pseudo-metric and studied its properties, etc. A partial review is given in~\cite{kuw3}. 

One disadvantage of the Tamarkin category is that it works quite well for \emph{exact Lagrangian submanifolds}, but it does not work well for non-exact Lagrangian submanifolds. 
The reason is that $\rho^{-1}(L)$ is huge and redundant, and we have to take a Lagrangian leaf of it, which is possible only for the exact case. 
On the other hand, Nadler--Zaslow--Ganatra--Pardon--Shende's approach can treat non-exact (unobstructed) Lagrangian branes if one works over the Novikov \emph{field}, as pointed out by Shende~\cite{shendeHitchin}.

In \cite{Kuw1,Kuw2}, for WKB-theoretic construction of sheaf quantization, the second-named author enhanced Tamarkin's approach to treat possibly non-exact Lagrangian submanifolds. 
The key idea is to consider the equivariant sheaf version of the Tamarkin category. Then the equivariancy cancels the redundancy of $\rho^{-1}(L)$. Moreover, the resulting category is automatically enriched over the Novikov \emph{ring}.

In this paper, based on \cite{Kuw1,Kuw2}, we would like to start a systematic sheaf-theoretic study of non-exact Lagrangian geometry, which can be considered as a step to unify the two approaches by Nadler--Zaslow--Ganatra--Pardon--Shende and Tamarkin.

\subsection{Our results}
We start with defining our working category, which is a slight generalization of that of \cite{Kuw1}. Let $M$ be a real manifold, $\bR_t$ (resp.\ $\bR_u$) be the real line with the standard coordinate $t$ (resp.\ $u$). 
For a positive real number $c$, we denote by $\bR_{u<c}$ the subset of $\bR_u$ defined by the inequality $u<c$. 
We regard $\bR_u$ as the case $c=\infty$. 
We also fix a field $\bK$ and a subgroup $\bG$ of the additive group $\bR$. 
Then we consider the derived category of sheaves of $\bK$-vector spaces over $M \times \bR_{u<c} \times \bR_t$, and denote it by $\Sh^\bG(M \times \bR_{u<c} \times \bR_t)$.
We quotient out the non-positively microsupported sheaves and obtain $\Sh_{\tau>0}^\bG(M \times \bR_{u<c} \times \bR_t)$. 
If $\bG$ is the trivial group, this category is nothing but the category introduced by Tamarkin~\cite{Tamarkin} (plus doubling parameter~\cite{Guillermou12}).

We explain the role of $t$ and $u$. Let us start with $t$. 
As explained above, $t$ is introduced to treat non-conic Lagrangians. Namely, by using this additional variable, one can lift any non-exact Lagrangian in $T^*M$ to a conic Lagrangian in $T^*M\times T^*\bR_t$. As clarified in \cite{AI20}, this variable encodes energy information. By introducing the group action, 
\cite{Kuw1} observed the category $\Sh_{\tau>0}^\bG(M \times \bR_u \times \bR_t)$ is enriched over the Novikov ring associated with $\bG$:
\begin{equation}
    \Lambda_0^\bG\coloneqq \varprojlim_{\bG\ni c\rightarrow \infty}\bK[\bG\cap \bR_{\geq 0}]/T^c\bK[\bG\cap \bR_{\geq 0}],
\end{equation}
where $T^c$ is the indeterminate corresponding to $c\in \bG\cap \bR_{\geq 0}$.
This is analogous to the fact that Fukaya categories are defined over the Novikov ring~\cite{FOOO} that encodes energy information in Floer theory.

On the other hand, the variable $u$ is first used by Guillermou~\cite{Guillermou12} for an intermediate step to produce sheaf quantizations through \emph{doubling movies}.
It is further used by Guillermou~\cite{Gu19}, Asano--Ike~\cite{AISQ}, and Nadler--Shende~\cite{nadler2020sheaf} to produce ``obstructed" sheaf quantizations. We view such sheaf quantizations as an object of the category $\Sh_{\tau>0}^{\bG}(M \times \bR_{u<c} \times \bR_t)$. 
In Floer theory, the counterpart is obstructed Lagrangian branes whose Floer cohomology is defined modulo $T^c$. 
Since objects of $\Sh_{\tau>0}^{\bG}(M \times \bR_{u<c} \times \bR_t)$ do not always have the shape of doubling movies, we restrict ourselves to the subcategory of such objects. 
We denote the subcategory by $\mu^\bG(T^*M;u<c)$. 

When $\bG$ is trivial or $\bZ$, the theory of sheaf quantizations is developed in various literature, including \cite{Guillermou12,JinTreumann,viterbo2019sheaf,AISQ,nadler2020sheaf}. Here we develop the analogous theory for non-exact Lagrangian submanifolds inside the category $\mu^\bG(T^*M;u<c)$. 

Now we state our first theorem. Let $\bm{L}$ be a Lagrangian brane, which is a tuple of a Lagrangian immersion $L\looparrowright T^*M$, a grading $\alpha$, a relative Pin structure $b$, and a local system $\cL$. 
Associated with $(L, \alpha, b)$, we have a functor assigning a local system on $L$ to an object $\cE$ with $L=\musupp(\cE)=\rho(\MS(\cE)\cap\lc\tau>0\rc)$, where $\MS$ is the microsupport\footnote{We need a mild modification of $\rho$ in the presence of $u$-variable. See the body of the paper.}. 
We say that an object $\cE$ is a \emph{sheaf quantization} of $\bm{L}$ if the associated local system is $\cL$.

We can also generalize all the stories above to the setup of Weinstein manifolds by using Shende's definition, Nadler--Shende's and Li--Nadler--Shende's studies~\cite{ShendeH-principle, nadler2020sheaf,LiNadlerShende}. 

\begin{theorem}[Existence of SQ, \cref{lem:c-SQ}]\label{thm:existence_intro}
Let $X$ be a sufficiently Weinstein manifold.
Let $\bm{L}$ be an end-conic Lagrangian brane in $X$ whose whose intersection with the Liouville domain is compact. 
Then there exists a positive number $c$ such that there exists a sheaf quantization of $\bm{L}$ in $\mu^{\bG_L}(X;u<c)$, where $\bG_L$ is the period group of $L$.
\end{theorem}
This is a sheaf-theoretic counterpart of the Gromov compactness: For any Lagrangian brane, there exists a positive number $c$ such that an unobstructed Floer complex is defined over $\Lambda_0/T^c\Lambda_0$.
In the following, we set $\bG$ to be $\bR$ and omit $\bR$ from the notation.

In general, the Floer complex over $\Lambda_0$ is not a complex (i.e., $d^2\neq 0$). To get a genuine complex, in the book~\cite{FOOO}, Fukaya--Oh--Ohta--Ono developed the notion of bounding cochain, which deforms the Floer differential to be non-curved over $\Lambda_0/T^{c'}\Lambda_0$ for a larger $c'$. Analogously, to extend our sheaf quantization to $\mu(X;u<\infty)$, we need the notion of sheaf-theoretic bounding cochain. Our theorem is the following:

\begin{theorem}[Bounding cochain, \cref{thm:bc}]
Let $X$ be a sufficiently Weinstein manifold for the precise definition). 
Let $\bm{L}$ be a Lagrangian brane in $X$. 
We fix a sufficiently fine discrete monoid $\sigma$. 
Then, there exists a curved dga $CF_{\mathrm{SQ}}(\bm{L}, \bm{L},\sigma)$ such that any Maurer--Cartan element of $CF_{\mathrm{SQ}}(\bm{L}, \bm{L},\sigma)$ ($\coloneqq$ sheaf theoretic bounding cochain) gives a sheaf quantization of $\bm{L}$.  Conversely, any sheaf quantization of $\bm{L}$ gives a sheaf-theoretic bounding cochain.
\end{theorem}
See \cref{thm:bc} for a precise statement. 

We now state our conjectures relating our category and Fukaya category.
\begin{conjecture}\label{conj:main}
    Let $X$ be a sufficiently Weinstein manifold.
    \begin{enumerate}
        \item For each Lagrangian brane $\bm L$ in $X$, we have an almost quasi-isomorphism between curved $A_\infty$-algebras $CF_{\mathrm{SQ}}(\bm{L}, \bm{L},\sigma)$ and $CF(\bm{L}, \bm{L},\sigma)$. 
        \item We denote the infinitesimally wrapped Fukaya category over $\Lambda_0$ by $\Fuk(X)$. Then there exists an almost fully faithful embedding $F\colon \Fuk(X)\hookrightarrow \mu(X, u<\infty)$ such that $F(\bm L)$ is a sheaf quantization of $\bm L$.
    \end{enumerate}
\end{conjecture}
\begin{remark}
    \begin{enumerate}
        \item Conjecture~\ref{conj:main}(1) implies that sheaf-theoretic bounding cochain is the same as Floer-theoretic bounding cochain. In particular, one can associate a sheaf quantization to a Floer-theoretic bounding cochain. An example of a version of this statement can be found in \cite{RutherfordSullivan}.
        \item Conjecture~\ref{conj:main}(2) implies that one can construct a sheaf quantization using Floer theory. In exact case, this has been proved by Viterbo~\cite{viterbo2019sheaf}.
        \item The integral version of Conjecture~\ref{conj:main}(2) is solved in \cite{KPS}.
    \end{enumerate}
\end{remark}

From \cref{sec:energy_cutoff} to the end of the paper, we give a generalization of structural properties of the Tamarkin category to our category in the sufficiently Weinstein setup. 
The first one is the intersection point estimate, which is proved by \cite{Ike19} for compact exact Lagrangian submanifolds of a cotangent bundle. 
Our generalization here is twofold: (1) We remove the compactness assumption, and (2) we identify what $\muhom$ is in the clean case. 
As a result, we obtain the following (see the body of the paper for a more precise statement):

\begin{theorem}[Intersection points estimate, \cref{cor:intersection_estimate}]
Let $X$ be sufficiently Weinstein manifold.
Let $\bm{L}_1,\bm{L}_2$ be end-conic Lagrangian branes and $\cE_1, \cE_2$ be their sheaf quantizations, respectively. 
Assume that $L_1$ cleanly intersects with $L_2$.  
Then 
\begin{equation}
   \sum_{k\in \bZ}\dim_\bK H^k\Hom_{\mu(X;u<c)}(\cE_1,\cE_2)_a\otimes \Lambda_0/\Lambda_0^+
    = 
    \sum_{k \in \bZ} \dim_\bK H^k(L_1 \cap L_2; \cM_{\bm{L}_1, \bm{L}_2}),
\end{equation}
where $\cM_{\bm{L}_1, \bm{L}_2}$ is the sheaf-theoretic Maslov index local system, which is defined only with information of the brane structures. In particular, in the case $c=\infty$,
\begin{equation}
    \sum_{k \in \bZ} \dim_\bK H^k(L_1\cap L_2; \cM_{\bm{L}_1, \bm{L}_2})
    \geq 
    \sum_{i\in \bZ} \dim_\Lambda H^i \Hom_{\mu(X;u<\infty)}(\cE_1, \cE_2)\otimes_{\Lambda_0}\Lambda.
\end{equation}
If $L_1$ transversely intersects with $L_2$ and $\cE_1, \cE_2$ are simple, then the left-hand side is $\#(L_1\cap L_2)$.
\end{theorem}

The second is about the metric aspect. 
Motivated by sheaf-theoretic study of persistence distance~\cite{KSpersistent}, Asano--Ike~\cite{AI20} introduced a pseudo-distance on the sheaf category $\Sh_{\tau>0}(M\times \bR_t)$. 
On the other hand, Guillermou--Kashiwara--Schapira~\cite{GKS} introduced sheaf quantization of Hamiltonian isotopy by sophisticating \cite{Tamarkin}. 
For the pseudo-distance on the sheaf category $\Sh_{\tau>0}(M\times \bR_t)$, \cite{AI20} proved the stability property of GKS-autoequivalence. 
Asano--Ike~\cite{AsanoIkecomplete} and Guillermou--Viterbo~\cite{GV2022singular} also proved the completeness of the metric.
In our setup, the metric is defined with the $\Lambda_0$-linear structure. Then we can prove the completeness and the Hamiltonian stability:

\begin{theorem}[Metric aspects, \cref{thm:Hamiltonianstability}] We denote by $\mu^{sm}(X;u<c)$, the subcategory of sheaf quantizations of Lagrangian branes. 
There exists a pseudo-metric $d_I$ on $\mu(X; u<c)$ such that the following holds:
\begin{enumerate}
    \item For $\cE, \cF\in \mu^{\mathrm{sm}}(X; u<c)$, if $d_I(\cE, \cF)<\infty$ , then $\cE\cong \cF$ over the Novikov field.
    \item Assume $X$ is sufficiently Weinstein. For any compactly supported Hamiltonian diffeomorphism $\varphi$, there exists an autoequivalence $\cK_\varphi$ on $\mu^{\mathrm{sm}}(X; u<c)$ that satisfies $\musupp(\cK_\varphi(\cE)) = \varphi(\musupp(\cE))$ for any $\cE \in \mu^{\mathrm{sm}}(X;u<c)$. 
    Moreover, for any $\cE \in \mu^{\mathrm{sm}}(X;u<c)$, one has 
    \begin{equation}
        d_I(\cE,\cK_\varphi(\cE)) \leq \|\varphi\|_H,
    \end{equation}
    where $\|\varphi\|_H$ is the Hofer norm of $\varphi$.
\end{enumerate}
\end{theorem}

Although $d_I$ is a degenerate metric on $\mu(X; u<c)$, it is natural to expect it to be non-degenerate on limits of sheaf quantization.
Such a kind of result is proved by Guillermou--Viterbo~\cite{GV2022singular} in the compact exact setting.
In our setup, we prove the following. 
Let $\bm{L}$ be a Lagrangian brane with a (sheaf-theoretic) bounding cochain $\bm{b}$. We assume that the corresponding sheaf quantization has rank 1 endomorphism space.
We denote by $\frakL(\bm{L}, \bm{b})$ the set of Lagrangian branes with bounding cochains that are isotopic to $(\bm{L}, \bm{b})$ by compactly-supported Hamiltonian isotopy. 
This set is equipped with the Hofer norm. 
We denote the completion by $\widehat{\frakL}(\bm{L}, \bm{b})$. 

\begin{theorem}[Limits of SQ, \cref{thm:limitSQ}]
    Assume $X$ is sufficiently Weinstein.
    There exists a canonical functor $\widehat{\frakL}(\bm{L}, \bm{b})\rightarrow \mu(X;u<\infty)$ extending the sheaf quantization. 
    More precisely, if two Cauchy sequences in $\frakL(\bm{L}, \bm{b})$ define the same limiting object, the corresponding limits of sheaf quantizations are also isomorphic.
\end{theorem}

This is the first paper of a series of papers we plan to write. In the forthcoming papers, we will address the problem (1) how to give an alternative model of our categories that is closer to Fukaya categories~\cite{KuwNov}, and (2) how to prove the conjecture comparing our categories with Fukaya categories~\cite{KPS2}.

\subsection{Organization of this paper}
Here is the organization of the paper. In \cref{section:equiv_nov_ring}, we define our category, which is a combination of the setup in \cite{Kuw1,Kuw2} and \cite{AISQ}. 
In \cref{section:microlocal_category,section:SQ}, we develop sheaf quantization formalism for possibly non-exact Lagrangian branes with the notion of sheaf-theoretic bounding cochains. 
In \cref{sec:energy_cutoff,section:metric,section:hamiltonian_auto,section:limit_SQ}, we study the structural properties of our microlocal category, which are known for the classical Tamarkin category.

\subsection*{Acknowledgment}

Y.~I.\ is supported by JSPS KAKENHI Grant Numbers JP21K13801 and JP22H05107.
T.~K.\ is supported by JSPS KAKENHI  Grant Numbers JP22K13912 and JP20H01794.
Y.~I.\ thanks Tomohiro Asano for the helpful discussions. 
T.~K.\ thanks Kaoru Ono and Hiroshi Ohta for having helpful discussions on the ideas of bounding cochain in various occasions. T.~K.\ also thanks Vivek Shende and Wenyuan Li for their helpful suggestion regarding the Liouville setup. The authors also thank the anonymous referees for their helpful comments.

\subsection*{Notation}

Throughout this paper, we let $\bK$ be a unital integral commutative ring.

Let $X$ be a manifold and let $\pi \colon T^*X \to X$ denote its cotangent bundle.
We write $T^*_XX$ for the zero-section of $T^*X$.
We set $\bK_X$ to be the constant sheaf on $X$ with stalk $\bK$.

A \emph{category} (resp.\ \emph{triangulated category}) means a \emph{dg-category} (resp.\ \emph{pre-triangulated dg-category}) unless specified.
For example, $\Sh(X)$ the derived category of the abelian category of $\bK_X$-modules, is a triangulated category in our sense.
For an object $\cE \in \Sh(X)$, we write $\MS(\cE)$ for the microsupport of $\cE$ (see \cite{KS90} for the definition), which is a closed conic subset of $T^*X$.

A category $\cC$ is said to be \emph{cocomplete} if it admits arbitrary colimits.
For a functor $F\colon \cC\to \cD$ between categories, we say $F$ is \emph{cocontinuous} if it preserves colimits.

A \emph{local system} in this paper means a finite rank locally constant $\bK_X$-module sheaf (not a complex of sheaves). We use the word \emph{derived local system} for an object in the derived category of sheaves whose cohomology sheaves are local systems. If the cohomology sheaves of a derived local system is concentrated in one degree, we call it a \emph{shifted local system}.

Let $Y$ be a Liouville manifold. We denote the contact boundary by $\partial Y$. When $Y=T^*X$, $\partial Y=(T^*X \setminus T^*_XX)/\bR_{>0}$.

\section{Equivariant sheaves and the Novikov ring}\label{section:equiv_nov_ring}

Some of the following results have already appeared in \cite{Kuw1, Kuw2}, but we recall them systematically. 
Until the end of \cref{section:equiv_nov_ring}, all the statements without proofs are proved in \cite{Kuw1} (in updated version).

\subsection{Equivariant sheaves and the Novikov action}

Let $M$ be a manifold and $\bR_t$ be the real line with the standard coordinate $t$. 
On the product space $M \times \bR_t$, we consider the action of a subgroup $\bG$ of the additive group $\bR$ via the addition on the right factor; $T_c \colon M \times \bR_t \to M \times \bR_t; (x,t) \mapsto (x,t+c)$ for $c \in \bG$. 
We consider this action as a continuous action with the discrete\footnote{If one equips $\bR$ with the usual Euclidean topology, then the resulting equivariant geometry is the same as $M$.} topology on $\bG$. 
We denote the trivial subgroup of $\bR$ by $\bO$.

Let $\Sh_\heartsuit(M \times \bR_t)$ denote the abelian category of $\bK$-module sheaves on $M \times \bR_t$.
An equivariant sheaf $\cE$ on $M \times \bR_t$ with the action of $\bG$ is given by the following data:
\begin{enumerate}
    \item A sheaf $\cE$ on $M \times \bR_t$, i.e., an object $\cE \in \Sh_\heartsuit(M \times \bR_t)$.
    \item A collection of isomorphisms $e_a \colon \cE\to T_a\cE \ (a \in \bG)$ satisfying $e_{a+b}=T_a e_b \circ e_a$.
\end{enumerate}
As usual, these objects form an abelian category.
Let $\Sh^\bG_\heartsuit(M \times \bR_t)$ be the abelian category of equivariant sheaves. 
This is a Grothendieck abelian category~\cite{Tohoku}. We denote its derived category by $\Sh^\bG(M \times \bR_t)$.
Note that $\Sh(M \times \bR_t)= \Sh^\bO(M \times \bR_t)$ for the non-equivariant category.
For an inclusion of subgroups $\bH\subset \bG\subset \bR$, we denote the functor forgetting equivariant structures by $\frakf_{\bG\bH}\colon \Sh^\bG(M \times \bR_t)\to \Sh^\bH(M \times \bR_t)$. We also use the notation $\frakf\coloneqq \frakf_{\bG\bO}$.

For an object $\cE$ of $\Sh^\bG(M \times \bR_t)$, we define its microsupport by
\begin{equation}
    \MS(\cE)\coloneqq \MS(\frakf(\cE)),
\end{equation}
where $\MS$ on the right-hand side is the usual microsupport. 
We write $\tau$ for the cotangent coordinate of $t$ and write $\{ \tau > 0\}$ for the subset of $T^*M \times T^*\bR_t$ defined by $\tau >0$.
For $\cE \in \Sh^\bG(M \times \bR_t)$, we also define 
\begin{equation}
    \MS_{\tau>0}(\cE) \coloneqq \MS(\cE) \cap \{ \tau >0\}.
\end{equation}

We write $\{\tau\leq 0\}$ for the subset of $T^*M \times T^*\bR_t$ defined by $\tau\leq 0$.
We consider the full subcategory 
\begin{equation}
    \cD^\bG_{\tau \leq 0}\coloneqq \lc \cE\in\Sh^\bG(M \times \bR_t)\relmid \MS(\cE)\subset \lc \tau\leq 0\rc \rc
\end{equation}
and set $\Sh_{\tau>0}^\bG(M \times \bR_t)\coloneqq \Sh^\bG(M \times \bR_t)/\cD^\bG_{\tau \leq 0}$. 

\begin{example}
\begin{enumerate}
    \item When $\bG=\bO$ the trivial group, we denote the category $\Sh^\bO_{\tau>0}(M \times \bR_t)$ by $\Sh_{\tau>0}(M \times \bR_t)$, which is the (usual) non-equivariant version of the Tamarkin category~\cite{Tamarkin,GS14}. 
    \item When $\bG$ is a non-trivial discrete group in $\bR$ with the Euclidean topology, it is isomorphic to $\bZ$ as a subgroup of $\bR$. Hence the category is $\Sh^\bG_{\tau>0}(M \times \bR_t)$ is equivalent to $\Sh_{\tau>0}(M \times S^1)$, which was used in \cite{AISQ}.
    \item When $\bG=\bR$, the category is the one introduced in \cite{Kuw1, Kuw2}.
\end{enumerate}
\end{example}

The functor $\frakf_{\bG\bH}$ induces a functor $\Sh^\bG_{\tau>0}(M \times \bR_t)\to \Sh^{\bH}_{\tau>0}(M \times \bR_t)$, which will also be denoted by $\frakf_{\bG\bH}$.
We note that $\MS_{\tau>0}(\cE) = \MS(\cE)\cap \{\tau>0\}$ is well-defined for objects in $\Sh_{\tau>0}^\bG(M \times \bR_t)$.

\subsection{Equivariant version of Tamarkin projector}

For $\cE, \cF\in \Sh_\heartsuit(M \times \bR_t)$, we set
\begin{equation}
    \cE\star\cF\coloneqq m_! \delta_M ^{-1}(\cE\boxtimes \cF),
\end{equation}
where $\delta_M \colon M \times \bR^2 \to M^2 \times \bR^2$ is the diagonal embedding on the first factor and $m\colon M \times \bR^2 \to M \times\bR_t$ is the addition on the second factor.
Suppose that $\cE$ is equipped with an equivariant structure. 
Then each structure morphism $e_a$ induces an isomorphism
\begin{equation}
   e_a\star \id\colon  \cE\star\cF \to (T_a\cE)\star \cF=T_a(\cE\star\cF)
\end{equation}
for $\cF\in \Sh_{\heartsuit}(M \times \bR_t)$, which gives an equivariant structure on $\cE\star \cF$. 
As a result, we get a functor $(-)\star_G \cF\colon \Sh^\bG_\heartsuit(M \times \bR_t)\to \Sh^\bG_{\heartsuit}(M \times \bR_t)$ for each $\cF\in \Sh_{\heartsuit}(M \times \bR_t)$. 
By deriving it, we get an endofunctor of $\Sh^\bG(M \times \bR_t)$. 
We denote it by the same symbol $(-)\star_G\cF$. 
The following is obvious.

\begin{lemma}
    One has $\frakf((-)\star_G\cF)=\frakf(-)\star \cF$.
\end{lemma}

In particular, consider the case when $\cF\coloneqq \bK_{t\geq 0}$, which is the constant sheaf supported on $\lc(x, t)\relmid t\geq 0 \rc$. 
By combining the above lemma with techniques developed by Tamarkin~\cite{Tamarkin} and Guillermou--Schapira~\cite{GS14}, one can deduce $\cE\star_G \bK_{\geq 0} \cong 0$ for $\cE\in \cD^\bG_{\tau \leq 0}$. Hence we get the induced functor $(-)\star_G \bK_{\geq 0}\colon\Sh_{\tau>0}^\bG(M \times \bR_t)\to \Sh^\bG(M \times \bR_t)$. 
The following is an analogue of Tamarkin's result.

\begin{lemma}[{\cite[\S 2.4]{Kuw1}}]\label{lemma:tamarkin_projector}
    The functor $(-)\star_G\bK_{\geq 0}\colon \Sh_{\tau>0}^\bG(M \times \bR_t)\to \Sh^\bG(M \times \bR_t)$ is fully faithful left adjoint of the quotient functor $\Sh^\bG(M \times \bR_t)\to \Sh^\bG_{\tau>0}(M \times \bR_t)$. The essential image coincides with the left orthogonal $^{\perp}\cD_{\tau \leq 0}^\bG$.
\end{lemma}

\subsection{Monoidal operations}

We consider equivariant operations, which is an adaptation of the materials in \cite{Kuw2}. 
Here, we briefly recall them.

We first recall some basic operations for equivariant sheaves. 
For details, we refer to \cite{BernsteinLunts, Tohoku}. 
Let $G$ be a group and $X_1, X_2$ be $G$-spaces. 
Let $f\colon X_1\to X_2$ be a $G$-map. Then we have functors
\begin{equation}
\begin{aligned}
    f_*, f_!&\colon \Sh^G(X_1)\to \Sh^G(X_2),\\
    f^{-1}, f^!&\colon \Sh^G(X_2)\to \Sh^G(X_1).
\end{aligned}
\end{equation}
For these functors, usual adjunctions hold. We can also define tensor and internal hom.

Let $\phi\colon G\to H$ be a surjective group homomorphism and $Y$ be an $H$-space. 
Then $G$ acts on $H$ through $\phi$. 
By setting $K$ to be the kernel of $\phi$, we have the invariant functor
\begin{equation}
    (-)^K\colon \Sh^G(Y)\to \Sh^H(Y)
\end{equation}
and the coinvariant functor
\begin{equation}
    (-)_K\colon \Sh^G(Y)\to \Sh^H(Y).
\end{equation}
If one has an $H$-equivariant sheaf on $Y$, it can also be considered as a $G$-equivariant sheaf, which gives a functor
\begin{equation}
    \iota^\phi\colon \Sh^H(Y)\to \Sh^G(Y).
\end{equation}

\begin{lemma}[{\cite[Lem.~6.1]{Kuw2}}]
    The functor $\iota^\phi$ is the right adjoint of $(-)_K$ and the left adjoint of $(-)^K$.
\end{lemma}

Let us go back to our particular situations.
We consider $M \times \bR^2$ on which $\bG^2$ acts by the addition on each component.
Through the addition $m \colon \bG^2 \to \bG$, the group $\bG^2$ also acts on $M \times \bR_t$. 
The kernel of the map $m \colon \bG^2 \to \bG$ is the anti-diagonal $\Delta_a\coloneqq \{(t, -t) \in \bG \times \bG\}$.

We also consider the addition map $m\colon M \times \bR^2 \to M \times \bR_t$ on $\bR_t$-factors. We then have a functor
\begin{equation}
    m_!\colon \Sh^{\bG^2}(M \times \bR^2)\to \Sh^{\bG^2}(M \times \bR_t).
\end{equation}
By combining it with the coinvariant functor
\begin{equation}
    (-)_{\Delta_a}\colon \Sh^{\bG^2}(M \times \bR_t)\to \Sh^{\bG}(M \times \bR_t),
\end{equation}
we define
\begin{equation}
    m_!^{\Delta_a}\coloneqq (-)_{\Delta_a}\circ m_!\colon \Sh^{\bG^2}(M \times \bR^2)\to\Sh^{\bG}(M \times \bR_t).
\end{equation}
We have the right adjoint of $m_!^{\Delta_a}$: 
\begin{equation}
    m^!_{\Delta_a} \coloneqq m^!\circ \iota^{m}\colon\Sh^{\bG}(M \times \bR_t) \to\Sh^{\bG^2}(M \times \bR^2).
\end{equation}

For $i=1,2$, let $p_i \colon M \times \bR^2 \to M \times \bR_t$ be the $i$-th projection. 
We also have the corresponding projection $q_i\colon \bG^2 \to \bG$. 
We then have
\begin{equation}
   p_{i*}^{\ker q_i}\coloneqq (-)^{\ker q_i}\circ p_{i*} \colon \Sh^{\bG^2}(M \times \bR^2) \to \Sh^{\bG}(M \times \bR_t)
\end{equation}
and 
\begin{equation}
   p_i^{-1}\coloneqq p_i^{-1}\circ \iota^{q_i} \colon  \Sh^{\bG}(M \times \bR_t) \to \Sh^{\bG^2}(M \times \bR^2).
\end{equation}

Combining the above statements, we deduce:

\begin{lemma}
    The functor $p_i^{-1}$ is the left adjoint of $p_{i*}^{\ker q_i}$.
\end{lemma}

For objects $\cE, \cF\in \Sh^{\bG}(M \times \bR_t)$, we set,  
\begin{equation}
    \cE \star_{\bG} \cF\coloneqq m_!^{\Delta_a}(p_1^{-1}\cE\otimes p_2^{-1}\cF).
\end{equation}
For $\cF, \cG \in \Sh^\bG(M \times \bR_t)$, we also set
\begin{equation}\label{eq:homstar}
    \cHom^{\star_\bG}(\cF,\cG) \coloneqq {p_1^{\ker q_1}}_* \cHom_{M \times \bR^2}(p_2^{-1}\cF, m^!_{\Delta_a}\cG).
\end{equation}
This is the right adjoint of $\star_\bG$.

\begin{lemma}
    The functors $\star_{\bG}$ and $\cHom^{\star_\bG}$ descend to functors on $\Sh_{>0}^\bG(M \times \bR_t)$. 
\end{lemma}
   
We also introduce the relative operation. Suppose $\bH\subset \bG$. 
For $\cE\in \Sh^\bG(M \times \bR_t)$ and $\cF\in \Sh^\bH(M \times \bR_t)$, we can consider $\cE\star_\bH\cF$. 
For the same reason as in defining $\star_G$, this object can be viewed as an object in $\Sh^\bG(M \times \bR_t)$. Hence we get a functor,
\begin{equation}
\begin{split}
   (-)\star_{\bG\bH} (-)\colon  \Sh^\bG(M \times \bR_t)\times \Sh^\bH(M \times \bR_t)\rightarrow \Sh^\bG(M \times \bR_t),\\
      (-)\star_{\bG\bH} (-)\colon  \Sh^\bH(M \times \bR_t)\times \Sh^\bG(M \times \bR_t)\rightarrow \Sh^\bG(M \times \bR_t).
\end{split}
\end{equation}
In particular, we have $\star_{\bG\bO}=\star_{G}$ and $\star_{\bG\bG}=\star_{\bG}$.

\subsection{Equivariant and non-equivariant}\label{section:equivandnonequiv}

We have introduced the forgetful functor $\frakf_{\bG\bH}\colon \Sh^\bG_{\tau>0}(M \times \bR_t)\to \Sh_{\tau>0}^\bH(M \times \bR_t)$. 
The left adjoint of this functor is given by 
\begin{equation}
   \frakf^L_{\bG\bH}\coloneqq (-)\star_{\bG\bH} \bigoplus_{c\in \bG}\bK_{t\geq c}\colon \Sh_{\tau>0}^\bH(M \times \bR_t)\to \Sh^\bG_{\tau>0}(M \times \bR_t).
\end{equation}

For any $0<a\in \bR$ and $\cE\in \Sh_{\tau>0}(M \times \bR_t)$, we have a morphism $T^a\colon \cE\to T_a\cE$. 
We have an analogous morphism $T^a\colon \cE\to T_a\cE$ for any $a>0$ and $\cE\in \Sh^\bG_{\tau>0}(M \times \bR_t)$.
For these morphisms, we have the following compatibility.

\begin{lemma}
    One has $\frakf_{\bG\bH}^L(T^a)=T^a$.
\end{lemma}

We name a particularly important object:
\begin{equation}
     1_{\mu} \coloneqq\bigoplus_{c\in \bG}\bK_{t\geq c}  =\frakf_{\bG\bH}^L(\bigoplus_{c\in \bH}\bK_{t\geq c}) \in \Sh^\bG_{\tau>0}(M \times \bR_t).
\end{equation}
With this notation, we have $\frakf_{\bG\bH}^L = (-) \star_{\bG\bH} 1_{\mu}$.

\subsection{Novikov ring action}

In this section, we see that the Novikov ring acts on $\Sh_{\tau>0}^\bG(M \times \bR_t)$. 

First, we mention the following lemma, which is easily proved by an argument in \cite[\S 2.4]{Kuw1}.

\begin{lemma}[{\cite[\S 2.4]{Kuw1}}]\label{lem:Lambda}
The endofunctor $(-)\star_{\bG} 1_{\mu}$ on $\Sh^\bG_{\tau>0}(M \times \bR_t)$ is naturally isomorphic to $\id$.
\end{lemma}

We denote the submonoid of the non-negative elements of $\bG$ by $\bG_{\geq 0}$.
We set
\begin{equation}
    \Lambda_0^\bG\coloneqq
    \varprojlim_{c \to \infty} \bK[\bG_{\geq 0}]/T^c\bK[\bG_{\geq 0}].
\end{equation}
This is the \emph{Novikov ring} $\Lambda_0$ introduced by Fukaya--Oh--Ohta--Ono~\cite{FOOO} when $\bG=\bR$. 
The ring $\Lambda_0^\bG$ for a general $\bG$ can be regarded as the counterpart of $\bG$-gappedness in \cite{FOOO}. 
The indeterminate corresponding to $a\in \bG_{\geq 0}$ will be denoted by $T^a$. Let $\Lambda_0^+$ denote the unique maximal ideal of the Novikov ring $\Lambda_0^\bG$. The quotient $\Lambda_0^\bG/\Lambda_0^+$ is isomorphic to $\bK$.

If $M$ is a singleton, we have an almost isomorphism $\End(\bigoplus_{c\in \bG}\bK_{\geq c})\cong \Lambda_0^\bG$~\cite{KuwNov}. 

\begin{definition}
The dg-algebra $\End(\bigoplus_{c\in \bG}\bK_{\geq c})$ will be denoted by $\mathbf{\Lambda}_{0}^{\bG}$, and will be called the derived Novikov ring.
\end{definition}

For general $M$, we have a morphism $\Lambda^\bG_0 \to \End((-)\star_{\bG}1_{\mu})$ (see \cite{KuwNov}). 
Hence we get the following homomorphism:
\begin{equation}
    \Lambda_0^\bG\rightarrow \End^0((-)\star_{\bG}1_{\mu})\rightarrow \Hom^0(\cE\star_{\bG}1_{\mu} , \cF\star_{\bG} 1_{\mu})=\Hom^0(\cE, \cF),
\end{equation}
where the last equality follows from \cref{lem:Lambda}.
Hence we have:
\begin{corollary}
The category $\Sh_{\tau>0}^\bG(M \times \bR_t)$ is $\Lambda_0^\bG$-linear.
\end{corollary}

\subsection{Lurie tensor product presentation}
By Lurie~\cite{HTT}, for two cocomplete categories $\cC_1, \cC_2$, there exists a cocomplete category $\cC_1\otimes \cC_2$ and a morphism $\cC_1\times \cC_2\rightarrow \cC_1\otimes \cC_2$ satisfying the following universal property: 
Let $\cC_1, \cC_2, \cC_3$ be cocomplete categories and  $f\colon \cC_1\times \cC_2\rightarrow \cC_3$ be a functor that is cocontinuous in each variable. 
Then there exists a unique functor $\widetilde{f}\colon \cC_1\otimes \cC_2\rightarrow \cC_3$ factorizing $f$.

The following K\"unneth property is known:
\begin{lemma}[{\cite[Corollary 1.3.1.8]{SAG}, see also \cite{Volpe}}]
    Let $M, N$ be topological spaces. We have
    \begin{equation}
        \Sh(M)\otimes \Sh(N)\cong \Sh(M\times N).
    \end{equation}
\end{lemma}

By using this, we prove the following:
\begin{lemma}[cf.~{\cite{KuoShendeZhang}}]
    There exists an equivalence:
    \begin{equation}
        \Sh^\bG_{\tau>0}(M \times \bR_t)\cong \Sh(M)\otimes \Sh^\bG_{\tau>0}(\bR_t).
    \end{equation}
\end{lemma}
\begin{proof}
    We would like to show that $\Sh^\bG_{\tau>0}(M \times \bR_t)$ has the desired universal property. Let us first equip $\bR_t$ be the $\gamma$-topology with $\gamma=(-\infty,0]$ and denote it by $\bR_t^\gamma$ i.e., the topology generated by $(-\infty, c)$ for any $c\in \bR$. We then take the Borel construction of the action of $\bG$ on $\bR_t^\gamma$ and denote it by $B_\bG\bR_t^\gamma$. Then, by the above lemma and \cite{KS90}, we have
    \begin{equation}
    \Sh_{\tau\geq 0}^\bG(M \times \bR_t)\cong \Sh(M \times B_\bG\bR_t^\gamma)\cong \Sh(M)\otimes \Sh(B_\bG\bR_t^\gamma) \cong \Sh(M)\otimes \Sh_{\tau\geq 0}(B_\bG\bR_t)
    \end{equation}
    where the subscript $\tau\geq 0$ means the full subcategory of the objects whose microsupports are contained in $\{\tau\geq 0\}$.
    
    Now let $f\colon \Sh(M)\times \Sh_{\tau>0}^\bG(\bR_t)\rightarrow \cD$ be a functor cocontinuous  in each variable. The composition of $f$ with the quotient functor $ q\colon \Sh(M)\times \Sh_{\tau\geq 0}^\bG(\bR_t)\rightarrow  \Sh(M)\times \Sh_{\tau>0}^\bG(\bR_t)$ is cocontiunous in each variable, hence it is factorized by a functor $\widetilde{f\circ q}\colon \Sh^\bR_{\tau\geq 0}(M\times \bR_t)\rightarrow \cD$. Since $q\circ f$ kills an object of the form $(\cE_1, \cE_2)$ with $\MS(\cE_2)\subset \{\tau=0\}$, $\widetilde{f\circ q}$ also kills the object with $\MS\subset \{\tau=0\}$. This implies that $\widetilde{f\circ q}$ is further factorized by a functor $\Sh_{\tau>0}^\bG(M \times \bR_t)\rightarrow \cD$. This gives the desired universal property.
\end{proof}

\subsection{Liouville manifolds}

In this section, we recall classes of exact symplectic manifolds. We restrict ourselves to talk only about Liouville/Weinstein manifolds.

Let us first recall the definition of Liouville manifold.
\begin{definition}[Liouville manifold]
Let $X$ be a $2n$-dimensional manifold. 
Let $\lambda$ be a 1-form. 
We say a pair $(X, \lambda)$ is a \emph{Liouville manifold} if the following holds:
\begin{enumerate}
    \item[(1)] $d\lambda$ is a symplectic form, and
    \item[(2)] there exists a compact manifold with smooth boundary $X_0 \subset X$, and the vector field defined by the relation $d\lambda(v_\lambda, -)=\lambda$ is outward-transverse on the boundary $\partial X_0$ and makes $X\bs X_0$ cylindrical. We call $X_0$ a \emph{Liouville domain}.
\end{enumerate}
\end{definition}

We also recall a related notion, Weinstein manifolds:

\begin{definition}[Weinstein manifold]
Let $(X,\lambda)$ be a Liouville manifold. 
We say a Liouville manifold is \emph{Weinstein} if there exists a Morse function $f$ satisfying the following:
\begin{enumerate}
    \item[(1)] $f$ is proper and bounded below, and
    \item[(2)] there exists a smooth positive function $\delta$ on $X$ such that $df(v_\lambda)\geq \delta(|v_\lambda|^2+|df|^2)$ for some Riemannian metric.
\end{enumerate}
\end{definition}

For our purpose, a slightly weaker notion is enough, namely \emph{sufficiently Weinstein manifold}, which is a Liouville manifold with isotropic skeleton satisfying some properties. See the \cite[Definition~9.13]{nadler2020sheaf} for the precise definition.

For a Liouville manifold $X$, we introduce the following deconification map:
\begin{equation}
    \rho\colon X\times T^*_{\tau>0}\bR_t\rightarrow X; (x, t, \tau)\mapsto \phi^{v_\lambda}_{-\log \tau}(x)
\end{equation}
where $\phi^{X_\lambda}$ is the Liouville flow. By the definition, $\rho^{-1}(A)$ is stable under the Liouville flow for any subset $A\subset X$.

\subsection{Microsheaf category}
We recall Shende's definition of microsheaf category~\cite{ShendeH-principle}.

We first start with the case of cotangent bundle. Let $M$ be a manifold. 
For an open subset of $U$ of $T^*M$, we set
\begin{equation}
    \Sh(\bK_M;U) \coloneqq \Sh(\bK_M)/\lc \cE\in \Sh(\bK_M)\relmid \MS(\cE)\subset T^*X \setminus U\rc.
\end{equation}
This forms a prestack over $T^*M$ and the stackification is called the \emph{Kashiwara--Schapira stack}.
We can consider the subsheaf spanned by the objects supported on $Y\subset T^*M$. It is a subsheaf supported on $Y$, which we denote by $\mush_Y$.

Now we recall Shende's trick~\cite{ShendeH-principle}. Let $X$ be a Liouville manifold with its Liouville form $\lambda$. Then one can embed the contactification $X\times \bR$ into $S^*\bR^N$ contactomorphically. Take a tubular neighborhood $U$ of the embedding, which is a symplectic vector bundle. Then we choose a \emph{polarization} $\frakp$, which is a Lagrangian distribution of $U$. The polarization gives a total space $X_\frakp$, which is a thickening of $X$ in $S^*\bR^N$. Then we take the subsheaf $\mush_{X_\frakp}$ whose support is contained in $X_\frakp$. This is the microsheaf category for $X$ with the choice of polarization $\frakp$. 

We can further generalize the above construction as follows, which is suggested by \cite{nadler2020sheaf}. For a Liouville manifold $X$, we take the stable Lagrangian Grassmann bundle $LGr(X\times \bR)$ over the contactification $X\times \bR$. 
Then we have the classification map $X\times \bR\rightarrow BU$ of the tangent bundle. We also have the universal Kashiwara--Schapira stack morphism $\frakL\colon U/O\rightarrow B\Pic(\Mod(\bK))$.
Composing these two morphisms, we get a morphism $X\times \bR \rightarrow B^2\Pic(\Mod(\bK))$. This gives a $B\Pic(\Mod(\bK))$-bundle $B\Pic(\Mod(\bK))_{X\times \bR}$ over $X\times \bR$ and a map $LGr(X\times \bR)\rightarrow B\Pic(\Mod(\bK))_{X\times \bR}$ between bundles.
Nadler--Shende observed that $\mush_{LGr(X\times \bR)}$ is pulled back from a sheaf over $B\Pic(\Mod(\bK))_{X\times \bR}$. This another sheaf is denoted by $\mush_{B\Pic(\Mod(\bK))_{X\times \bR}}$. 
A null homotopy of the morphism $X\times \bR\rightarrow B^2\Pic(\Mod(\bK))$ (namely, a section of $B\Pic(\Mod(\bK))_{X\times \bR}$) is called a Maslov data. For a Maslov data $\frakp$, by restricting to the image of $\frakp$, we get a sheaf on $X\times \bR$. By restricting it to $X\times 0$, we finally obtain $\mush_\frakp(X)$. 

\begin{remark}
    When working with $\mush_{\frakp}(X)$, we can work with $\mush_{LGr(X\times \bR)}$ with support and equivariantly with respect to $K \coloneqq \ker(U/O\rightarrow B\Pic(\Mod(\bK)))$. As the proof of \cite[Proposition 11.19]{nadler2020sheaf}, one can easily upgrade the results in the below from the polarization setting to the general Maslov setup. For this reason, we sometimes omit detailed explanations of $\mush_\frakp$, but all the results hold in the general Maslov setups.
\end{remark}

\section{Microlocal category}\label{section:microlocal_category}

In this section, we introduce our microlocal category over the Novikov ring.

\subsection{Doubling variable}

The doubling variable was originally used by Guillermou~\cite{Guillermou12} as an auxiliary variable to construct sheaf quantization. 
Later, Guillermou~\cite{Gu19} and Asano--Ike~\cite{AI20} used it to consider sheaf quantization of obstructed Lagrangians. 
In Nadler--Shende~\cite{nadler2020sheaf}, it is also essential in their theory. 

In what follows, let $X$ be a sufficiently Weinstein manifold.
Moreover, we let $c$ be a real positive number or $\infty$ and denote the open interval $(-\infty, c)$ with the standard coordinate $u$ by $\bR_{u<c}$. 
The cotangent coordinate of $T^*\bR_{u<c}$ will be denoted by $\upsilon$.

\begin{definition}
    Let $A$ be a subset of $X \times T^*_{\tau>0}\bR_t$ invariant under the $\bG$-translation along $\bR_t$. 
    The \emph{doubling} $AA$ of $A$ is defined as
    \begin{equation}
        AA \coloneqq AA_h\cup AA_t\subset X \times T^*\bR_{u<c} \times T^*_{\tau>0}\bR_t,
    \end{equation}
    where
    \begin{equation}
    \begin{split}
        AA_h &\coloneqq \lc (p,u,0,t,\tau) \relmid (p, t, \tau) \in A, u \ge 0 \rc \subset X \times T^*\bR_{u<c} \times T^*_{\tau >0}\bR_t,\\
        AA_t &\coloneqq \lc (p, u, \upsilon, t, \tau) \relmid (p, t, \tau) \in A, u \ge 0, \upsilon=-\tau\rc \subset X \times T^*\bR_{u<c} \times T^*_{\tau>0}\bR_t.        
    \end{split}
    \end{equation}
    We call $AA_h$ (resp.\ $AA_t$) the horizontal (resp.\ tilted) component of the doubling.
    
    A subset $B$ of $X \times T^*\bR_{u<c} \times T^*_{\tau>0}\bR_t$ is said to be a \emph{doubling} if $B=AA$ for some $A \subset X \times T^*_{\tau>0}\bR_t$ that is invariant under the $\bG$-translation. 
\end{definition}

\begin{definition}[Weak version]
    A subset $B$ of $X \times T^*\bR_{u<c} \times T^*_{\tau>0}\bR_t$ is said to be a \emph{weak doubling} if $B\subset AA$ for some $A \subset X \times T^*_{\tau>0}\bR_t$ that is invariant under the $\bG$-translation. 
\end{definition}

\subsection{Microlocal category over the Novikov ring}
For a Maslov data $\frakp$ for the sufficiently Weinstein manifold $X$, there is associated Maslov data on $X\times T^*\bR_{u<c}\times \bR_t$, which will also be denoted by $\frakp$. Then we have the category $\mush_\frakp(X\times \bR_{u<c}\times \bR_t)$.

We would like to define an analogue of Tamarkin category for this setup. We follow the trick by Li--Nadler--Shende~\cite{LiNadlerShende}, further used in \cite{KPS}. Let $\mush_{\frakp,\mathrm{Lag}}(X\times \bR_{u<c}\times \bR_t)$ be the subcategory spanned by the objects supported on Lagrangian subsets and the objects locally represented by their coproducts. Note that  $\mush_{\frakp}(X\times \bR_{u<c}\times \bR_t)$ is not known to be cocomplete, what we intend here is that $\mush_{\frakp,\mathrm{Lag}}(X\times \bR_{u<c}\times \bR_t)$ contains the coproducts if they exist.

We set 
\begin{equation}
\begin{split}
    \mush_{\frakp, \leq 0}(X\times \bR_{u<c}\times \bR_t)
    & \coloneqq \lc \cE\in\mush_{\frakp,\mathrm{Lag}}(X\times \bR_{u<c}\times \bR_t)\relmid \supp\cE\subset \lc \tau\leq 0\rc  \rc \\
    & \subset\mush_{\frakp,\mathrm{Lag}}(X\times \bR_{u<c}\times \bR_t), \\
    \mush_{\frakp, > 0}(X\times \bR_{u<c}\times \bR_t)
    & \coloneqq \mush_{\frakp, \mathrm{Lag}}(X\times \bR_{u<c})/\mush_{\frakp, \leq 0}(X\times \bR_{u<c}),
\end{split}
\end{equation}
where the second line is the Dwyer--Kan localization. It is important to note that \cite[Lemma C.1.5]{LiNadlerShende} shows that the subcategory spanned by the objects realized as microsheaves supported on $\overline{\Lambda_f}^{\prec}$ (later defined) can be computed by using $\mush_{\overline{\Lambda_f}^{\prec}}(\overline{\Lambda_f}^{\prec})$. This enables us to compute our category.

It was observed by \cite{LiNadlerShende}, this category $\mush_{\frakp, > 0}(X\times \bR_{u<c}\times \bR_t)$ has an action from $\Sh_{>0}(\bR_t)$, like for $\Sh_{>0}(M\times \bR_t)$. In particular, this has an $\bG$-action for any $\bG\subset \bR$. For $c\in \bG$, the corresponding action will again be denoted by $T_c$.

We denote the $\bG$-invariant by
\begin{equation}
    \mush_{\frakp, > 0}^\bG(X\times \bR_{u<c}\times \bR_t).
\end{equation}
Then it carries a $\Sh^\bG_{>0}(\bR_t)$-action, in particular, $\Lambda_0^\bG$-action. We denote the functor forgetting the $\bG$-equivariancy by 
\begin{equation}
    \frakf\colon \mush_{\frakp, > 0}^\bG(X\times \bR_{u<c}\times \bR_t)\to \mush_{\frakp, > 0}(X\times \bR_{u<c}\times \bR_t).
\end{equation}
Note that $\supp\cE\cap \lc \tau\geq 0\rc$ is well-defined on $\mush_{\frakp, > 0}(X\times \bR_{u<c}\times \bR_t)$.

\begin{definition}
    For an object $\cE\in  \mush_{\frakp, > 0}^\bG(X\times \bR_{u<c}\times \bR_t)$, we set $\MS_{\tau>0}(\cE)\coloneqq \supp(\frakf(\cE))\cap \{\tau>0\}$.
\end{definition}

We start with the notion of doubling movies.
\begin{definition}
    \begin{enumerate}
        \item An object $\cE$ of $ \mush_{\frakp, > 0}^\bG(X\times \bR_{u<c}\times \bR_t)$ is said to be a \emph{doubling movie} if $\cE|_{u\leq 0}=0$ and $\MS_{\tau>0}(\cE|_{u>0})$ is a doubling.
        \item An object of $ \mush_{\frakp, > 0}^\bG(X\times \bR_{u<c}\times \bR_t)$ is said to be a \emph{weak doubling movie} if $\cE|_{u\leq 0}=0$ and $\MS_{\tau>0}(\cE|_{u>0})$ is a weak doubling.
    \end{enumerate}
    
\end{definition}

Now we define our main microlocal category and non-conic microsupport of an object of the category.

\begin{definition}
    The subcategory of $\mush_{\frakp, > 0}^\bG(X\times \bR_{u<c}\times \bR_t)$ spanned by the weak doubling movies is denoted by $\mu^\bG_{\frakp}(X;u<c)$. 
    We call it the \emph{microlocal category} of $X$ with Maslov data $\frakp$.
\end{definition}

\begin{definition}[Non-conic microsupport]
    For an object $\cE$ of $\mu^\bG_\frakp(X;u<c)$, we set
    \begin{equation}
        \musupp(\cE)\coloneqq \overline{\bigcup_{u_0\in \bR_{u<c}}\rho(\MS_{\tau>0}(\cE|_{u_0}))} \subset X,
    \end{equation}
    where $\rho\colon X\times T^*_{\tau >0}\bR_t\rightarrow X;(p, t, \tau)\mapsto (\phi_{\tau}^{-1}p)$ where $\phi_\tau$ is the flow of $v_\lambda$.
\end{definition}

\begin{example}
    In the case when $X=T^*M$, the component $\overline{\rho(\MS_{\tau>0}(\cE|_{u_0}))}$ is what is called non-conic microsupport or reduced mircosupport, and widely used in the literature of Tamarkin category. See also Section~\ref{section:witout_doubling}.
\end{example}

\begin{lemma}\label{lem:musuppandSSpositive}
    Let $A'$ be a closed subset of $X$. 
    Then $\musupp(\cE)\subset A'$ if and only if $\MS_{\tau>0}(\cE|_{u>0})\subset \rho^{-1}(A')\rho^{-1}(A')$ for $\cE\in \mu^\bG_\frakp(X; u<c)$.
\end{lemma}
\begin{proof}
Suppose $\MS_{\tau>0}(\cE|_{u>0})\subset \rho^{-1}(A')\rho^{-1}(A')$. By the definition of the doubling, the inclusion of $\lc u=u_0\rc$ is non-characteristic for $ \rho^{-1}(A')\rho^{-1}(A')$. Here ``non-characteristic" means non-charactertic for a local sheaf representative (or after taking anti-microlocalization).
Hence we have $\rho(\MS_{\tau>0}(\cE|_{u_0}))\subset A'$, which implies $\musupp(\cE)\subset A'$. Conversely, suppose $\musupp(\cE)\subset A'$. By the definition of $\musupp$, we have $\rho(\MS_{\tau>0}(\cE|_{u_0}))\subset A'$. Since $\MS_{\tau>0}(\cE|_{u>0})$ is a weak doubling, we have $\MS_{\tau>0}(\cE|_{u>0})\subset \rho^{-1}(A')\rho^{-1}(A')$.
\end{proof}

Although we do not know whether $\mu^\bG_\frakp(X;u<c)$ is cocomplete or not, we still have the following.
\begin{proposition}\label{prop:pre_traingulated}
    The category $\mu^\bG_\frakp(X;u<c)$ is a pretriangulated category.
\end{proposition}
\begin{proof}
    By the definition of $\MS$, we can see that $\mu^\bG_\frakp(X;u<c)$ is a pre-triangulated category.
\end{proof}

Although we do not know whether $\mu_\frakp(X;u<c)$ is cocomplete or not, we still have the following.
\begin{lemma}\label{lem:coproduct}
    Let $\cE$ be an object in $\mu_\frakp(X;u<c)$ with Lagrangian support. Then $\bigoplus_{c\in \bG} T_c\cE$ exists in $\mu_\frakp^\bG(X;u<c)$. 
\end{lemma}
\begin{proof}
    This is proved in \cite{KPS}.\footnote{Note that a current version of \cite{KPS} contains a typo for the definition of $\mush_c$. One needs to replace $\mush_c$ there with $\mush_c$ here.}
\end{proof}

For $0<c<c'$, we have the canonical restriction functor
\begin{equation}
    r_{c'c}\colon \mu^\bG_\frakp(X;u<c')\rightarrow \mu^\bG_\frakp(X;u<c).
\end{equation}
We sometimes denote it by $r_{c}$ for simplicity.

\subsection{\texorpdfstring{$u$}{u}-translation and stupid extensions}

We would like to glue objects of $\mu(X;u<c)$ for different $c$'s. 
For this purpose, we introduce some notation.

For a positive real number $c'$, we consider the translation map
\begin{equation}
    S_{c'}\colon (-\infty, c)\rightarrow (-\infty, c+c'); a\mapsto a+c'.
\end{equation}
By the translation, we obtain the functor
\begin{equation}
S_{c'}\colon \mu(X; u<c)\rightarrow \mush^\bR_{>0}(X\times \bR_{u<c+c'} \times \bR_t). 
\end{equation}

We also define stupid extension functors. 
We consider the following map
\begin{equation}
    l_{c+c', c}\colon \bR_{u < c+c'}\rightarrow \bR_{u\leq c}; a\mapsto 
    \begin{cases}
        a & \text{if } a<c\\
        c & \text{otherwise}.
    \end{cases}
\end{equation}
Let $j_c\colon \bR_{u<c} \times \bR_t \hookrightarrow \bR_{u\leq c} \times \bR_t$ be the inclusion. We then have
\begin{equation}
     s_{c'}\colon \mu(X;u<c) \xrightarrow{j_{c*}}\mu(X;u\leq c) \xrightarrow{l_{c+c', c}^{-1}} \mu(X;u<c+c').
\end{equation}
Here $j_{c*}$ is defined through the anti-microlocalization.

\subsection{\texorpdfstring{$\sigma$}{sigma}-decomposition}
Now let us introduce a terminology: $\sigma$-decomposition. Let $\sigma\coloneqq \lc 0=c_0<c_1<\cdots <c_n<\cdots\rc $ be a set of positive real numbers less than $c$ without accumulation points. Let $\cE$ be an object $L$ in $\mu(X; u<c)$. Then we obtain 
\begin{equation}
\begin{split}
    \cE_{c_1}&\coloneqq \cE|_{u<c_1}, \\
    \cE_{c_2}&\coloneqq S_{-c_1}\Cone(s_{c_2-c_1}\cE|_{u<c_1}\rightarrow \cE|_{u<c_2}), \\
    \cE_{c_3}&\coloneqq S_{-c_2}\Cone(s_{c_3-c_2}\cE|_{u<c_2}\rightarrow \cE|_{u<c_3}), \\
    \vdots .
\end{split}   
\end{equation}
We call this sequence of sheaf quantizations \emph{$\sigma$-decomposition} of $\cE$. From the construction, one can imagine the following statement:
    If $\cE_{c_1},\dots, \cE_{c_n}, \dots$ is a $\sigma$-decomposition of $\cE$, then $\cE$ is an iterated extension of objects
    \begin{equation}
    \begin{split}
        S_{\sigma}\cE_{c_1}&\coloneqq s_{c-c_1}\cE_{c_1}, \\
        S_\sigma\cE_{c_2}&\coloneqq S_{c_1}s_{c-c_2}\cE_{c_2}, \\
        S_{\sigma}\cE_{c_3}&\coloneqq S_{c_2}s_{c-c_3}\cE_{c_3}, \\
        \vdots .
    \end{split}
    \end{equation}

To make sense of iterated extensions, we use the notion of twisted complex introduced by Bondal--Kapranov~\cite{BondalKapranov}, later generalized by Anno--Logvinenko~\cite{AnnoLogvinenko} to the case involving countably many objects.

Let $\cD$ be a dg-category. 
\begin{definition}
A \emph{twisted complex} is a tuple $(\{V_i\}_{i\in \bN},\{f_{ij}\}_{i > j})$, where
\begin{enumerate}
    \item $V_i$ is an object of $\cD$ and
    \item(Maurer--Cartan equation) $f_{ij}$ is a morphism $V_i\rightarrow V_j$ of degree $i-j+1$
\end{enumerate}
such that
\begin{equation}
    df_{ij}+\sum_{k}f_{kj}\circ f_{ik}=0
\end{equation}
holds for any $i>j$.
\end{definition}
\begin{remark}
    What we call a twisted complex is a simple one-sided complex in the original literature. For a more general version of the notion of a twisted complex, see \cite{BondalKapranov}. 
\end{remark}

The twisted complexes form a dg-category $\Tw\cD$. We denote the subcategory of $\Tw\cD$ spanned by the objects $(\{V_i\}_{i\in \bN},\{f_{ij}\}_{i > j})$ such that $\bigoplus_{i\in\bN} V_i[i]$ exists by $\Tw^{allow}(\cD)$.
The following is fundamental:
\begin{proposition}[\cite{BondalKapranov, AnnoLogvinenko}]
    Suppose $\cD$ is pretriangulated.
    \begin{enumerate}
        \item Then there exists a fully faithful functor $\Tw^{allow}\cD\rightarrow \cD$.
        \item Under the above functor, a twisted complex $(\{V_i\}_{i\in \bN},\{f_{ij}\}_{i > j})$ is mapped to the colimit of iterated extensions of $V_i$'s.
    \end{enumerate}
\end{proposition}
By this proposition, we view a twisted complex as an object of $\cD$.

Summarizing the above things, we obtain the following
\begin{lemma}
    Let $\cE$ be an object of $\mu(X; u<c)$ and $\sigma\coloneqq \lc 0=c_0<c_1<\cdots <c_n<\cdots\rc $ be a set of positive real numbers less than $c$ without accumulation points. Then $\cE$ can be represented as a twisted complex over $\lc S_{\sigma}\cE_{c_i}[i-1]\rc _{i \in \bN}$. 
\end{lemma}
\begin{proof}
    It is enough to show that the relevant direct sum exists. Since $\cE$ is supported in a doubling movie of the $\bR$-translations of a Lagrangian subset, the support of $S_{\sigma}\cE_{c_i}[i-1]$ is also controlled by the $\bR$-translations of a Lagrangian subset. By the same reasoning as \cite{KPS}, we get a local sheaf reprenentative, and get a direct sum.
\end{proof}

Going back to our situation, a $\sigma$-decomposition $\cE_i$ gives a twisted complex expressing $\cE$. In the following, we will use this perspective repeatedly.

\subsection{Without doubling variable}\label{section:witout_doubling}

It is sometimes convenient to consider the situation where $u$ is not involved. 
In particular, for many purposes (for example, the situation where one focuses on unobstructed branes), we can forget $u$. 
For this purpose, we prepare some notation for the category without $u$.

The definition is just by dropping $\bR_u$;
\begin{equation}
    \mu^\bG_\frakp (X)\coloneqq (\mush_{\frakp,c}(X\times \bR_t)/\mush_{\frakp,\leq 0}(X\times \bR_t))^\bG
\end{equation}
where the superscript $\bG$ means the $\bG$-invariant. For an object $\cE\in \mu^\bG_\frakp(X)$, we also define
\begin{equation}
    \musupp(\cE)\coloneqq \overline{\rho(\MS_{\tau>0}(\cE))} \subset X.
\end{equation}
For each $0<u\in \bR_{u<c}$, we set 
\begin{equation}
i_u\colon \bR_t \hookrightarrow  \bR_{u<c} \times \bR_t; \quad t \mapsto (u,t).
\end{equation}
Then we can define the restriction along $i_u$, called the specialization functor
\begin{equation}
    \Sp_u\coloneqq i_u^{-1}\colon \mu^\bG_\frakp(X;u<c)\rightarrow \mu^\bG_\frakp (X),
\end{equation}
which is $\Lambda_0$-linear. 
When $c<\infty$, we also consider the inclusion $j_c\colon  \bR_{u< c} \times \bR_t \rightarrow \bR_{u\leq c} \times \bR_t$. As in the above, we have $j_{c*}$, hence we can define
\begin{equation}
        \Sp_c\coloneqq i_c^{-1}j_{c*}\colon \mu^\bG(X;u<c)\rightarrow \mu^\bG(X),
\end{equation}
which is a kind of nearby cycle. 
When $c=\infty$, we consider the map $j_\infty\colon \bR_u \times \bR_t \hookrightarrow (-\infty, \infty] \times \bR_t$ associated with the one-point partial compactification $\bR_u \hookrightarrow (-\infty, \infty]$. 
With this map, we set
\begin{equation}
    \Sp_\infty\coloneqq i_\infty^{-1}j_{\infty*}\colon \mu^\bG_\frakp (X;u<\infty)\rightarrow \mu^\bG_\frakp (X),
\end{equation}
where $i_\infty \colon \bR_t \hookrightarrow (-\infty,\infty] \times \bR_t; (t) \mapsto (\infty,t)$.

\begin{lemma}
Let $A'$ be a closed subset of $X$. 
For $\cE\in \mu^\bG_\frakp (X; u<c)$ with $\MS(\cE|_{u>0})\subset \rho^{-1}(A')\rho^{-1}(A')$, we have $\musupp(\Sp_u(\cE))\subset A'$ for any $u\leq c$.
\end{lemma}
\begin{proof}
    It follows from the standard microsupport estimate (see~\cite[Lem.~3.16]{nadler2020sheaf}).
\end{proof}

\subsection{Over the Novikov field}

Since $\Lambda^\bG_0$ is an integral domain, we can take its fraction field. 
We denote it by $\Lambda^\bG$, which is called the (universal) \emph{Novikov field}.

Given a $\Lambda_0^\bG$-linear dg-category $\cD$, we can define the base-changed category $\cD\otimes_{\Lambda^\bG_0}\Lambda^\bG$ as follows:
\begin{enumerate}
    \item The set of objects $\Ob(\cD\otimes_{\Lambda_0^\bG}\Lambda^\bG)$ is $\Ob(\cD)$.
    \item For $\cE, \cF\in \Ob(\cD\otimes_{\Lambda_0^\bG}\Lambda^\bG)$, we set
    \begin{equation}
    \Hom_{\cD\otimes_{\Lambda_0^\bG}\Lambda^\bG}(\cE, \cF)\coloneqq \Hom_{\cD}(\cE, \cF)\otimes_{\Lambda_0^\bG}\Lambda^\bG.
    \end{equation}
    The differentials and compositions are naturally induced.
\end{enumerate}
We set $F^{\epsilon}\Hom_{\cD\otimes_{\Lambda_0^\bG}\Lambda^\bG}(\cE, \cF)$ to be the image of $T^{\epsilon}\Hom_{\cD}(\cE, \cF)$ under the natural morphism $\Hom_{\cD}(\cE, \cF)\rightarrow \Hom_{\cD}\otimes_{\Lambda_0^\bG}\Lambda^\bG(\cE, \cF)$.
Then the resulting category is a ``filtered" $\Lambda^\bG$-linear dg-category.

\section{Antimicrolocalization}
The case of cotangent bundles is easier than general case. The method of antimicrolocalization developed by Nadler--Shende~\cite{nadler2020sheaf} can reduce the Weinstein case to the cotangent bundle case. The purpose of this section is to adapt their method to our setting.

\subsection{Contact embeddings}
\begin{definition}
A $(2k+1)$-dimensional manifold $Z$ equipped with a $2k$-dimensional distribution $\xi$ is said to be a \emph{contact manifold} if there exists a 1-form $\alpha$ locally such that $\xi=\ker \alpha$ and $\alpha\wedge (d\alpha)^k\neq 0$.

Let $\alpha$ be a 1-form on a $(2k+1)$-dimensional manifold $Z$. The pair $(Z, \alpha)$ is said to be a \emph{co-oriented contact manifold} if $\alpha\wedge (d\alpha)^k\neq 0$. One can associate a contact manifold by setting $\xi=\ker \alpha$.
\end{definition}

\begin{definition}
Let $(Z_i, \xi_i)$ ($i=1,2$) be contact manifolds. An embedding $i\colon Z_1\hookrightarrow Z_2$ is said to be a \emph{contact embedding} if $TZ_1\cap \xi_2=\xi_1$ in $TZ_2$.

Let $(Z_i, \alpha_i)$ ($i=1,2$) be co-oriented contact manifolds. An embedding $i\colon Z_1\hookrightarrow Z_2$ is said to be a \emph{strict contact embedding} if $i^*\alpha_2=\alpha_1$.
\end{definition}

\begin{lemma}
A strict contact embedding is a contact embedding.
\end{lemma}
\begin{proof}
Let $i$ be a co-oriented contact embedding.
\begin{equation}
    \xi_1=\ker(i^*\alpha_2)=\ker(\alpha_2)\cap TZ_1=\xi_2\cap TZ_1.
\end{equation}
\end{proof}

\begin{lemma}\label{lem:cocoriented_embedding}
Let $(Z_i, \alpha_i) \ (i=1,2)$ be co-oriented contact manifolds. Let $i\colon Z_1\hookrightarrow Z_2$ be a contact embedding. Then there exists a nowhere zero function $f$ such that $f\alpha_1=i^*\alpha_2$.
\end{lemma}
\begin{proof}
Since $\ker\alpha_1=\ker i^*\alpha_2$.
\end{proof}

We also prepare a terminology for exact symplectic manifolds.
\begin{definition}
Let $(X_i, \lambda_i)$ $(i=1,2)$ be exact symplectic manifolds. An embedding $i\colon X_1\hookrightarrow X_2$ is said to be a \emph{strict exact symplectic embedding} if $i^*\lambda_2=\lambda_1$ holds.
\end{definition}

\begin{remark}
    Without the adjective ``strict", it usually means $i^*\lambda_2=\lambda_1+dH$ for some function $H$.
\end{remark}

\subsection{Strict exact embedding into \texorpdfstring{$T^*\bR^n$}{T*Rn}}
Let $(X,\lambda_X)$ be a Liouville manifold. Let $\bR_{t'}$ be the real line with the standard coordinate $t'$. The product $X\times \bR_{t'}$ with 1-form $\lambda_X+ dt'$ is a contact manifold, called the \emph{contactization}.
Indeed, the non-degeneracy condition $d(\lambda_X+dt)^n\wedge (\lambda_X+dt')=\omega^n\wedge dt'\neq 0$ is satisfied.

\begin{remark}
    The reason why we use the awkward notation $t'$ is that the microlocal-sheaf-theory's extra direction (i.e., Sato's lost dimension) $t$ differs from $t'$.
\end{remark}

Let $\lambda_{T^*\bR^n}$ be the standard Liouville form of $T^*\bR^n$. 
We equip $\bR^n$ with the standard Euclidean Riemann metric. 
Then $S^*\bR^n$ is defined by the unit cosphere bundle
\begin{equation}
    S^*\bR^n\coloneqq\lc (x, \xi)\in T^*\bR^n\relmid |\xi|^2=1 \rc.
\end{equation}
We denote the restriction of $\lambda_{T^*\bR^n}$ with respect to this embedding by $\lambda_{S^*\bR^n}$.

\begin{lemma}[h-principle]
For a sufficiently large $n$, there exists a contact embedding $X\times \bR_{t'}\hookrightarrow S^*\bR^n$. All those embeddings are stably contact isotopic.
\end{lemma}

Let us fix a contact embedding $i_1\colon X\times \bR_{t'}\hookrightarrow S^*\bR^n$. 
By \cref{lem:cocoriented_embedding}, there exists a nowhere zero function $f$ on $X\times \bR_{t'}$ such that $f(\lambda_X+dt')=i_1^*\lambda_{S^*\bR^n}$. If $f$ is negative, by composing the antipodal map on $S^*\bR^n$, we can make $i_1$ so that $f$ is positive. Hence, without loss of generality, we can assume that $f$ is positive.

Now we extend the smooth function $f$ on $X\times \bR_{t'}$ to a smooth function of $S^*\bR^n$ and denote it by the same $f$. We then consider 
\begin{equation}
    S^*_f\bR^n\coloneqq\lc (x, f(x, \xi)\xi)\relmid (x, \xi)\in S^*\bR^n\rc.
\end{equation}
By definition, $S^*\bR^n$ is diffeomorphic to $S^*_f\bR^n$. It is also standard to see that they are contact isomorphic. Hence the contact embedding $i$ can be viewed as a contact embedding
\begin{equation}
  X\times \bR \hookrightarrow S^*_f\bR^n.
\end{equation}
The advantage of this modification is that this embedding is a strict contact embedding. Hence
the further composition
\begin{equation}
    i\colon X=X\times \{0\}\hookrightarrow X\times \bR\hookrightarrow S^*_f\bR^n\hookrightarrow T^*\bR^n
\end{equation}
is a strict exact embedding in the sense that $i^*\lambda_{T^*\bR^n}=\lambda_X$.

We summarize the above as follows:
\begin{lemma}
For a sufficiently large $n$, there exists a strict exact embedding $X\hookrightarrow T^*\bR^n$.
\end{lemma}

\subsection{Antimicrolocalization}\label{subsec:antimicrolocalization}
Now let us assume $X$ is sufficiently Weinstein. 
We denote the core of $X$ by $\mathrm{Core}(X)$, which is a singular isotropic.
We fix a strict exact embedding $X\hookrightarrow T^*\bR^N$, whose existence is proved in the previous subsection. 
We also fix a polarization $\frakp$, which is a Lagrangian distribution in the symplectic normal bundle of $X$. Note that this choice is the same as the choice we made to define $\mush_\frakp(X)$.

Let $L_1,\dots, L_n$ be end-conic (see \cref{def:end_conic} below) Lagrangians in $X$. We denote the subcategory of $\mu(X;u<c)$ spanned by objects whose $\musupp$ are in $\bigcup L_i$ by $\mu_{\bigcup_iL_i}(X; u<c)$. 

On the other hand, the restriction $\frakp|_{L_i}$ defines a Lagrangian in $T^*\bR^N$. We consider $\bL \coloneqq \rho^{-1}(\bigcup_i\frakp|_{L_i}\cup \frakp|_{\mathrm{Core}(X)})$. By introducing another variable $s,v$, as in Nadler--Shende~\cite{nadler2020sheaf}, we can introduce the cusp doubling with boundary $\bL^\prec\subset T^*\bR^N\times T^*\bR_t\times T^*\bR_{u<c}\times T^*\bR_{v<a}$ for $a>0$ where $s$ is the contactification variable and $v$ is the doubling variable. We denote the subcategory of $\Sh^\bR_{\bL, \tau>0}(\bR^N\times \bR_t\times \bR_{u<c}\times\bR_s\times  \bR_{v<a})$ spanned by the objects whose restriction to $v\leq 0$ is zero by $\Sh^\bR_{\bL, \tau>0}(\bR^N\times \bR_t\times \bR_{u<c}\times\bR_s\times \bR_{v<a})_0$. 
For $0<a<a'$, we have the restriction functor
\begin{equation}
    \Sh^\bR_{\bL, \tau>0}(\bR^N\times \bR_t\times \bR_{u<c}\times \bR_s\times \bR_{v<a'})_0\rightarrow \Sh^\bR_{\bL, \tau>0}(\bR^N\times \bR_t\times \bR_{u<c}\times \bR_s \times \bR_{v<a})_0.
\end{equation}
We set
\begin{equation}
    \mu^{A\mu}_{\bigcup_iL_i}(X; u<c) \coloneqq \lim_{a\rightarrow +0}\Sh^\bR_{\bL, \tau>0}(\bR^N\times \bR_t\times \bR_u\times \bR_s\times\bR_{v<a})_0.
\end{equation}
Here $A\mu$ stands for ``anti-microlocalization".

\begin{lemma}[Antimicrolocalization]
    There exists an embedding
    \begin{equation}
        \mu_{\bigcup_iL_i}(X; u<c)\hookrightarrow \mu^{A\mu}_{\bigcup_iL_i}(X; u<c)
    \end{equation}
\end{lemma}
\begin{proof}
    We can prove this similarly as in \cite{nadler2020sheaf}.
\end{proof}

\section{Lagrangian intersection Floer theory and Fukaya category}
In the next section, we introduce the category of sheaf quantizations as a subcategory of $\mu(X; u<c)$ and describe objects in this category. The category $\mu(X; u<c)$ is conjecturally equivalent to a version of Fukaya category defined over $\Lambda_0/T^c\Lambda_0$. To understand the constructions in the next section, we briefly provide a summary of Fukaya category in the non-exact setup.

\subsection{Lagrangian branes}
In this subsection, we explain the notion of Lagrangian branes which is one of the ingredients of an object of Fukaya category.

\begin{definition}\label{def:end_conic}
Let $X$ be a Liouville manifold.
    Let $A'$ be a closed subset of $X$. 
    We say that $A'$ is \emph{end-conic} if there exists a Liouville domain $X_0$ of $X$ such that $A' \setminus X_0$ is conic, i.e., stable under the scaling action of $\bR_{\geq 1}$ induced by the flow of $v_\lambda$.
\end{definition}
For an end-conic subset $A'$, we set $\partial A'\coloneqq (\bR_{>0}\cdot (A' \setminus X_0))/\bR_{>0}\subset \partial X_0$, \emph{the boundary of $A'$}.

 Let $i_L\colon L \looparrowright T^*M$ be a Lagrangian immersion. We set $\Lim\coloneqq i_L(L)$ for the simplification. 
We make the following assumption for Lagrangian immersion hereafter in this paper.

\begin{assumption}\label{assumption:immersion}
    Any Lagrangian immersion $i_L \colon L\looparrowright X$ in this paper is assumed to satisfy the following.
    \begin{enumerate}
        \item[(1)] The image $\Lim$ is end-conic. 
        \item[(2)] The set $X_0\cap \Lim$ is relatively compact.
        \item[(3)] The image of $i_L$ has clean self-intersection of $\codim\geq 1$. 
        That is, the following hold. 
        \begin{enumerate}
            \item The fiber product $L \times_{X} L$ is a smooth submanifold of $L \times L$, and $L \times_{X} L \setminus \Delta_L$ has dimension less than $n$ where $\Delta_L$ is the diagonal of $L\times L$.
            \item For $(p_1, p_2) \in L \times_{X} L$, we have 
            \begin{equation}
                T_{(p_1,p_2)}(L \times_{X} L) = 
                \lc (V,W) \in T_{p_1}L \times T_{p_2}L \relmid 
                (d_{p_1}i_L)(V) = (d_{p_2}i_L)(W) 
                \rc.
            \end{equation}
        \end{enumerate}
    \end{enumerate}     
\end{assumption}

\begin{remark}
    The assumption of clean self-intersection of $\codim\geq 1$ is not essential.  
    If $i_L$ has codimension $0$ clean self-intersection, we can treat it as an embedded Lagrangian brane with higher rank local systems. 
    For example, let $L$ be an embedded Lagrangian and $(\alpha_i, b_i, \cL_i)$ ($i=1,2$) be brane structures on $L$ with $\bK=\bZ$, which will be defined below. 
    Then the immersed Lagrangian brane given by $L\sqcup L\rightarrow X$ with the brane structure $\bigsqcup_{i=1,2}(\alpha_i, b_i, \cL_i)$ can be regarded as an embedded Lagrangian brane $(L, \alpha_1, b_1, (\cL_1\oplus \cL_2\otimes (b_2-b_1)[\alpha_2-\alpha_1])$, where $(b_2-b_1)$ is the principal $\cO(1)$-bundle and $\alpha_2-\alpha_1\in \bZ$. 
\end{remark}

Let $i_L\colon L\looparrowright X$ be a Lagrangian immersion. Then we have the composition 
\begin{equation}
    L\xrightarrow{i_L} X\xrightarrow{TX} BU\rightarrow B(U/O)\xrightarrow{\frakL}B^2\Pic(\Mod(\bK)). 
\end{equation}
Since $L$ is a Lagrangian, we have a null homotopy of the composition $L\rightarrow B(U/O)$, which induces a null homotopy map $L\rightarrow B^2\Pic(\Mod(\bK))$. Comparing this with the null homotopy given by the given Maslov structure, we obtain a morphism $L\rightarrow B\Pic(\Mod(\bK))$. 

\begin{definition}
    A \emph{brane structure} of $L$ is a tuple $\bm{L}\coloneqq (L, \bm{\alpha}, \cL)$, where
    \begin{enumerate}
        \item $i_L\colon L\looparrowright X$ is a Lagrangian immersion,
        \item $\bm{\alpha}$ is a null-homotopy of $L\rightarrow B\Pic(\Mod(\bK))$,
        \item $\cL$ is a (shifted) local system over $L$.
    \end{enumerate}
\end{definition}

A Lagrangian brane $\bm{L}$ is said to be \emph{simple} if $\cL$ is a rank 1 local system.
When $\cL$ is a trivial rank 1 local system, we simply say $\bm{L}=(L, \bm\alpha)$ is a Lagrangian brane. 
For a Lagrangian brane $\bm{L}=(L, \bm\alpha, \cL)$, we denote by $\bm{L}^0$ the associated simple Lagrangian brane $(L, \bm\alpha)$ with the trivial rank 1 local system.

\begin{remark}\label{rmk:branecoefficient}
    Lagrangian brane data can be defined for more general coefficients $\cC$ (i.e., symmetric monoidal presentable stable $\infty$-category) instead of $\Mod(\bK)$. See \cite{Jin, nadler2020sheaf} for further information.
\end{remark}

\subsection{Curved \texorpdfstring{$A_\infty$}{Ainfty}-algebra and bounding cochain}

We recall the notion of bounding cochain by \cite{FOOO} briefly (see \Cref{sec:curved_dga} for some terminology). In this subsection (and only in this subsection), a Lagrangian brane means a Lagrangian brane in a compact symplectic manifold or a Lagrangian brane with bounded geometry in a noncompact symplectic manifold. A particular situation of the latter case is of our interest; (possibly nonexact) end-conic Lagrangian branes in sufficiently Weinstein manifolds. Strictly speaking, the book~\cite{FOOO} is only written for the compact case, but the adaption to the bounded geometry case is straght forward.

\begin{theorem}[\cite{FOOO, AkahoJoyce, FukayaLagrangianCorresp}]\label{theorem:Fukayaalgebra}
    For a Lagrangian brane $\bm{L}$, there exists a curved $A_\infty$-algebra $CF(\bm{L},\bm{L})$ defined over $\Lambda_0$ such that
    \begin{enumerate}
        \item[(1)] the cohomology of $CF(\bm{L},\bm{L})\otimes_{\Lambda_0}\Lambda_0/\Lambda_0^+$ is isomorphic to the ``cohomology ring''\footnote{If it is immersed, one has to take care of the definition. See \cite{AkahoJoyce, FukayaLagrangianCorresp} and \cref{rem:Maslov}.} of $L$, and
        \item[(2)] there exists an inclusion map $C^1(L;\bK)\otimes_\bK \Lambda_0\hookrightarrow CF(\bm{L}, \bm{L})$, where $C^1(L;\bK)$ is a model of singular chain complex of $L$ as $\Lambda_0$-modules (without differential),
        \item[(3)] there exists a small positive number $c$ such that $CF(\bm L, \bm L)\otimes_{\Lambda_0} \Lambda_0/T^c\Lambda_0$ is non-curved.
    \end{enumerate}
\end{theorem}
Of course, the statement here is far from the complete statement (and not very precise). We state this theorem in this form to compare with our results in the next section.

\begin{definition}
    A \emph{bounding cochain} of $\bm{L}$ over $\Lambda_0/T^c\Lambda_0$ is an element $\bm{b}$ of $CF^1(\bm{L}, \bm{L})\otimes_{\Lambda_0}\Lambda_0/T^c\Lambda_0$ such that it is a Maurer--Cartan element of the curved $A_\infty$-algebra.
\end{definition}
If $\bm{b}$ is a bounding cochain, one can deform $CF^1(\bm{L}, \bm{L})\otimes_{\Lambda_0}\Lambda_0/T^c\Lambda_0$ and get a genuine $A_\infty$-algebra $CF^1((\bm{L}, \bm b), (\bm{L}, \bm b))\otimes_{\Lambda_0}\Lambda_0/T^c\Lambda_0$. 

\begin{example}\label{example:bc_smallc}
    As a corollary of \cref{theorem:Fukayaalgebra}, there exists a small number $c$ such that every cocycle $b\in CF^1(\bm L,\bm L)$ is a bounding cochain of $\bm L$ over $\Lambda_0/T^c\Lambda_0$. 
\end{example}

There is also a module-version of the statement:
\begin{theorem}[\cite{FOOO, AkahoJoyce, FukayaLagrangianCorresp}]\label{theorem:Fukayamodle}
    For a cleanly intersecting pair of Lagrangian branes $\bm{L}_1, \bm{L}_2$, there exists a curved $A_\infty$-module $CF(\bm{L}_1,\bm{L}_2)$ defined over $\Lambda_0$ such that
    \begin{enumerate}
        \item[(1)] the cohomology of $CF(\bm{L}_1,\bm{L}_2)\otimes_{\Lambda_0}\Lambda_0/\Lambda_0^+$ is the cohomology of the ``intersection locus"\footnote{If it is non-transverse, one has to take care of the definition. See \cite{AkahoJoyce, FukayaLagrangianCorresp} and \cref{rem:Maslov}.} between $L_1$ and $L_2$, and
        \item[(2)] there exists an inclusion map $C^1(L_1\cap L_2;\bK)\otimes_\bK \Lambda_0\hookrightarrow CF(\bm{L}_1, \bm{L}_2)$ as $\Lambda_0$-modules (without differential), where $C^1(L_1\cap L_2;\bK)$ is a model of singular chain complex of $L_1\cap L_2$.
    \end{enumerate}
\end{theorem}
If $\bm b_1$ and $\bm b_2$ are bounding cochains of $\bm L_1$ and $\bm L_2$ over $\Lambda_0/T^c\Lambda_0$ respectively, then they deform $CF(\bm{L}_1,\bm{L}_2)\otimes_{\Lambda_0}\Lambda_0/T^c\Lambda_0$ to get a $CF((\bm{L}_1, \bm b_1),(\bm{L}_1,\bm b_1))\otimes_{\Lambda_0}\Lambda_0/T^c\Lambda_0-CF((\bm{L}_2,\bm b_2),(\bm{L}_2, \bm b_2))\otimes_{\Lambda_0}\Lambda_0/T^c\Lambda_0$-bimodule $CF((\bm{L}_1, \bm b_1),(\bm{L}_2, \bm b_2))\otimes_{\Lambda_0}\Lambda_0/T^c\Lambda_0$.

The Fukaya category of $X$ over $\Lambda_0/T^c\Lambda_0$ has the following property:
\begin{enumerate}
    \item An object is a Lagrangian brane with a bounding cochain over $\Lambda_0/T^c\Lambda_0$.
    \item A hom-space is given by $CF((\bm{L}_1, \bm b_1),(\bm{L}_2, \bm b_2))\otimes_{\Lambda_0}\Lambda_0/T^c\Lambda_0$.
\end{enumerate}
We emphasize that the above is a general sketch and we have to modify and impose several things in respective setup you want.

\begin{example}\label{ex:low-energyobject}
    Let $\bm L$ be a Lagrangian brane. By \cref{example:bc_smallc}, there exists a small $c>0$ such that $(\bm L, b)$ is an object of the Fukaya category over $\Lambda_0/T^c\Lambda_0$ for any cocycle $b\in CF^1(\bm L, \bm L)$.
\end{example}

\subsection{What is explained in the next section?}
Upon the above explanation of Fukaya category, we here would like to give a summary of the next section.

The aim of the next section is to give a category of sheaf quantizations. We expect this category is equivalent to a version of Fukaya category. As a corollary of this expectation, each sheaf quantization corresponds to a Lagrangian brane. In particular, it has
\begin{enumerate}
    \item the notion of support, which is an immersed Lagrangian, and
    \item the notion of brane structure.
\end{enumerate}
For 1, the notion of support is already defined as $\musupp$. For 2, in \S \ref{subsec:microsheaves}, we will recall the way to read brane structures from sheaves. Then, by using this, in \S \ref{subsec:SQ_brane_str}, we will define sheaf quantization as a sheaf whose support is Lagrangian and having a reasonable brane structure.

Our definition of sheaf quantization is just a definition but not a construction. In \S \ref{subsec:SQ_low_energy}, we provide a construction of sheaf quantization over $\Lambda_0/T^c\Lambda_0$ for small $c$, called standard sheaf quantizations, which conjecturally correspond to \cref{ex:low-energyobject}. 
Then we classify sheaf quantizations in \S \ref{subsec:classification_low_energy_SQ} in low-energy, where we find they are standard sheaf quantizations. Their properties are studied in \S \ref{subsec:properties_standard_SQ}.
We also describe microlocal hom-spaces between sheaf quantizations in \S \ref{subsec:maslov} using Maslov data.

Next, we would like to describe objects (conjecturally) corresponding to $(\bm L, \bm b)$, where $\bm b$ is a bounding cochain. To consider this, we take the following interpretation of $(\bm L, \bm b)$. Namely, $(\bm L, \bm b)$ is a certain iterated extension of \cref{ex:low-energyobject}. Correspondingly, we would like to describe certain twisted complexes consisting of standard sheaf quantizations. However, it seems that not every twisted complex corresponds to a bounding cochain. For example, a direct sum of ``simple" sheaf quantization is not simple.

To control this issue, in \S \ref{subsec:microlocal_extension}, we explain the notion of microlocal extension class of twisted complex. If the microlocal extension class of a twisted complex consisting of simple sheaf quantizations is standard, the resulting sheaf quantization is again simple. 

In \S \ref{subsec:sheaf_Fukaya_alg}, we package this class of twisted complexes as a solution of Maurer--Cartan element (a sheaf-theoretic bounding cochain) of a curved dga. This curved dga conjecturally quasi-isomorphic to the curved $A_\infty$-algebra in \cref{theorem:Fukayaalgebra}. The fact corresponding to \cref{theorem:Fukayaalgebra}
(1) and (2) are also explained for this curved dga. In \S \ref{subsec:bc_SQ}, we explain a correspondence between sheaf quantizations and sheaf-theoretic bounding cochains, which is a (conjectural) counterpart of the fact that an object of Fukaya category is given by a pair of Lagrangian brane and a bounding cochain.

\section{Sheaf quantizations}\label{section:SQ}

\subsection{Microsheaves along Lagrangian}\label{subsec:microsheaves}

In this subsection, we explain microsheaves along a Lagrangian immersion.

Let $(X,\lambda)$ be a sufficiently Weinstein manifold with a chosen Maslov data. Let $i_L\colon L \looparrowright X$ be a Lagrangian immersion.
We can locally take a primitive $f\colon U\cap\Lim \rightarrow \bR$ of $\lambda|_U$, where $U$ is a contractible open subset of $X$. 

We set
\begin{equation}
    U_f\coloneqq \lc (x, \xi, t)\in X \times \bR_t\relmid (x, \xi)\in \Lim, t=-f(x, \xi)\rc.
\end{equation}
This is a Legendrian in $X\times \bR_t$, one can define $\mu sh_{U_f}$ by using the higher-codimensional trick. Note that $U_f$ does not depend on the choice of $f$, since one can identify those for different $f$ by translations. Hence we can glue them over $\Lim$. We denote it by $\mush_{\Lim}$. We denote the internal hom-sheaf of $\mush_{\Lim}$ by $\muhom$.


\begin{remark}
    The definition of our $\mush_{\Lim}$ is different from that of Jin's $\muSh_L$~\cite{Jin} if $L^{\Image}$ has self-intersection. As an effect, \cref{lem:branestr} below is stated in a form slightly different from \cite{Jin}.
\end{remark}

We introduce some notation following Tamarkin. In the following, we drop the subscript for the Maslov data to lighten the notation.
For a closed subset $A'\subset X$, we set
\begin{equation}
\begin{split}
    \mu_{A'}(X;u<c) & \coloneqq \lc \cE\in \mu(X;u<c) \relmid \musupp(\cE)\subset A'\rc  \text{ and}\\
    \mu_{A'}(X)&\coloneqq \lc \cE\in \mu(X) \relmid \musupp(\cE)\subset A'\rc.
\end{split}
\end{equation}
See also \cref{lem:musuppandSSpositive}.

In \cite[\S~3.4]{Kuw1}, the author constructed the microlocalization functor 
\begin{equation}
    \mu_{\Lim} \colon  \mu^\bR_{\Lim}(T^*M)\to \mush_{\Lim}(\Lim),
\end{equation}
by mimicking the known constructions in the exact case. 
However, it contains some bugs that is not desired as explained in \cite{KuwNov}.
We explain the modification and how to construct the microlocalization functor in a general $\bG$-equivariant Weinstein case.

Now let us construct the microlocalization functor. We first consider the naive version. By the usual microlocalization we have $\mu^\bG_{\Lim}(X)\to \mush_{\rho^{-1}(\Lim)}(\rho^{-1}(\Lim))$. By restricting it to a local conic lift $U_f\subset \rho^{-1}(\Lim)$, we have $\mush_{\rho^{-1}(\Lim)}(\rho^{-1}(\Lim))\rightarrow \mush_{\rho^{-1}(\Lim)}|_{U_f}(U_f)$. 
In \cite{Kuw1} (which is in the case of $T^*M$, but we can apply it, since the argument is (micro)local), it is proved that the image of this functor is in $\mush_{U_f}(U_f)$. Hence we further have
\begin{equation}
    \mu^\bG_{\Lim}(X)\rightarrow \prod_{c\in \bR/\bG}\mush_{U_{f+c}}(U_{f+c})\rightarrow \mush_{U_{f}}(U_{f}),
\end{equation}
where the last functor is taken by the direct product using the canonical equivalences along $\mush_{U_{f+c}}(U_{f+c})\rightarrow \mush_{U_{f}}(U_{f})$. The latter product exists by the same reason as Lemma~\ref{lem:coproduct}. 
Using the equivariant structure, we get a functor $\mu_{\Lim}^{pre} \colon  \mu^\bG_{\Lim}(X) \to \mush_{L^{\Image}}(L^{\Image})$. 
This is the desired microlocalization functor. This functor does not behave well if $\bG$ is dense, so we modify as follows. 

The stalk at $p$ of the image of the restriction functor $\mu^\bG_{\Lim}(X)\to \mush_{\rho^{-1}(\Lim)}(\rho^{-1}(\Lim))$ is computed by a fomula like
\begin{equation}
    \colim_{p \in U}\Gamma_{\lc t\geq -f(x)\rc }(U, -)
\end{equation}
where $p=(x, df(x))$ for a smooth function $f$ after some higher codimensional embedding into a cotangent bundle. We replace this by killing almost zero
parts:
\begin{equation}
    \colim_{p \in U}\Gamma_{\lc t\geq -f(x)\rc }(U, -)_a
\end{equation}
See Appendix~\ref{section:almost} for the operation $(-)_a$. This replacement is glued up and gives a modified functor denoted by $\mu_{\Lim} \colon  \mu^\bG_{\Lim}(X) \to \mush_{L^{\Image}}(L^{\Image})$.


As an analogue, we can construct a functor $\mu^\bG_{\Lim} \colon \mu_{\Lim}(X;u<c) \to \mush_{\Lim}(\Lim)$ as follows: 
For some $u> 0$, we take a contractible neighborhood $U$ of $u$ in $(0, \infty)$. 
The microlocalization along $AA_h|_U$ as in loc.\ cit.\ gives an object of $\mush_{\Lim\times U}(\Lim\times U)$ along $\Lim\times U$. 
Since $U$ is contractible, it equivalently gives a functor to $\mush_{\Lim}(\Lim)$.

Let $L_1$ and $L_2$ be two Lagrangian immersions and let $\cE_i\in \mu^\bG_{\Lim_i}(X)$.
Then for $i=1,2$, we have $\mu_{\Lim_i}(\cE_i)\in \mush_{\Lim_i}(\Lim_i)$. Since $ \mush_{\Lim_i}(\Lim_i)\subset \mush_{\Lim_1\cup \Lim_2}(\Lim_1\cup \Lim_2)$,  we can consider $\mu_{\Lim_i}(\cE_i)$ as an object of the latter category. We then set
\begin{equation}\label{def:muhomR}
\muhom^\bG(\cE_1, \cE_2)\coloneqq \cHom_{\mush_{\Lim_1\cup \Lim_2}}(\mu_{\Lim_1}(\cE_1), \mu_{\Lim_2}(\cE_2)),
\end{equation}
which is a sheaf on $L_1^{\Image}\cap L_2^{\Image}$.

In the following, we mostly set $\bG$ to be $\bR$ and omit $\bR$ from the notation. 
This is only for the simplicity of notation, and almost all the results hold for other $\bG$'s. 
In \S\ref{subsec:SQ_low_energy}, we use different $\bG$'s explicitly, where we compare our sheaf quantization with the previously existing studies.

\begin{lemma}[{cf.\ \cite{Jin}}]\label{lem:branestr}
    Given a Lagrangian immersion with a brane structure $\bm\alpha$, there exists a faithful embedding $\Loc(L)\hookrightarrow \mush_{\Lim}(\Lim)$, where the left-hand side is the category of derived $\cC$-valued local systems. 
    The image of the embedding is denoted by $\muSh_L(L)$. If $L$ is an embedded Lagrangian, we have $\muSh_L(L)\cong \mush_{\Lim}(\Lim)$.
\end{lemma}
We denote the object in $\mush_{\Lim}(\Lim)$ corresponding to the rank 1 trivial local system under the above embedding by $\bm L = (L,\bm \alpha)$.

For a Lagrangian immersion $i_L\colon L\looparrowright X$, we set
\begin{equation}
    \mu_L(X; u<c) \coloneqq (\mu_{\Lim})^{-1}(\muSh_L(L)),
\end{equation}
where $\mu_{\Lim} \colon \mu_{\Lim}(X;u<c) \to \mush_{\Lim}(\Lim)$ is the microlocalization defined as above. 
If $L$ is embedded, it is compatible with the definition of $\mu_L(X; u<c)$ introduced above.

Given a simple Lagrangian brane $\bm{L}=(L,\bm \alpha)$, by composing the functor given above with $\mu_{L^\mathrm{Im}}$, we obtain a functor 
\begin{equation}
    \mu_{\bm{L}}\colon \mu_{L}(X;u<c) \to \Loc(L).
\end{equation}
If $L$ is further equipped with a rank 1 local system $\cL$ (i.e., $\bm{L}=(L,\bm\alpha,\cL)$ is a simple Lagrangian brane), we set
\begin{equation}
     \mu_{\bm{L}}\coloneqq \mu_{\bm{L}^0}\otimes \cL^{-1}\colon \mu_{L}(X;u<c) \to \Loc(L).
\end{equation}
We can similarly define $\mu_{\bm{L}}\colon \mu_{L}(X)\rightarrow \Loc(L)$.

\begin{definition}
\begin{enumerate}
\item An object of $\mu(X;u<c)$ is said to be \emph{holonomic} if $\musupp(\cE)$ is possibly singular Lagrangian and $\mu_{(L^{\mathrm{sm}}, \bm\alpha)}(\cE)$ is of finite rank for some $\bm\alpha$ on the smooth part $L^{\mathrm{sm}}$. 
    \item An object of $\mu(X;u<c)$ is said to be \emph{smooth} if $\musupp(\cE)$ is an embedded Lagrangian and $\mu_{(L, \bm\alpha)}(\cE)$ is of finite rank for some $\bm\alpha$. 
    \item The subcategory of $\mu(X;u<c)$ spanned by the holonomic objects (resp.\ smooth objects) is denoted by $\mu^{\mathrm{hol}}(X;u<c)$ (resp.\ $\mu^{\mathrm{sm}}(X;u<c)$). 
    We call it the holonomic (resp.\ smooth holonomic) microlocal category of $X$. 
\end{enumerate}
\end{definition}

\begin{conjecture}
    The inclusion $\mu^{\mathrm{sm}}(X;u<c)\hookrightarrow \mu^{\mathrm{hol}}(X;u<c)$ induces an equivalence 
    \begin{equation}
        \mu^{\mathrm{sm}}(X;u<c)\otimes_{\Lambda_0}\Lambda\xrightarrow{\cong} \mu^{\mathrm{hol}}(X;u<c)\otimes_{\Lambda_0}\Lambda
    \end{equation}
    over the Novikov field.
\end{conjecture}

\subsection{Sheaf quantization and brane structure}\label{subsec:SQ_brane_str}
Let $\bm{L}=(L,\bm{\alpha}, \cL)$ be a Lagrangian brane.

\begin{definition}
    An object $\cE\in \mu^\bG_{L}(X;u<c)$ is said to be a \emph{sheaf quantization} of a Lagrangian immersion $L$ over $\Lambda_0^\bG/T^c\Lambda_0^\bG$ if $\mu_{(L, \bm\alpha)}(\cE)$ is of finite rank for some $\bm\alpha$.
\end{definition}

More precisely, we define:
\begin{definition}\label{def:sheaf_quantization}
    An object $\cE\in \mu^{\bG}(X;u<c)$ (resp.\ $\cE\in \mu^{\bG}(X)$) is said to be a \emph{sheaf quantization} of $\bm{L}$ over $\Lambda_0^\bG/T^c\Lambda_0^\bG$ if 
    \begin{enumerate}
        \item[(1)] $\cE\in \mu^\bG_{L}(X;u<c)$ (resp.\ $\cE \in \mu^\bG_L(X)$) and
        \item[(2)] $\mu_{\bm{L}^0}(\cE) \cong \cL$. 
    \end{enumerate}
    We say that a sheaf quantization $\cE$ is \emph{simple} if $\cL$ is of rank 1.
\end{definition}
We remark that for $0<c<c'$, if $\cE\in \mu^\bG(X;u<c')$ is a sheaf quantization of $L$ (resp.\ $\bm{L}$), the restriction $r_{c'c}(\cE)$ is a sheaf quantization of $L$ (resp.\ $\bm{L}$) in $\mu^\bG(X;u<c)$.

We also give the relationship between sheaf quantizations over $\Lambda_0^\bG$ and $\Lambda_0^\bH$.
\begin{proposition}
    Let $\cE$ be a sheaf quantization of $\bm{L}$ over $\Lambda_0^\bH$. Then $\frakf_{\bG\bH}^L(\cE)$ is a sheaf quantization of $\bm{L}$ over $\Lambda_0^\bG$ for $\bH\subset \bG$.
\end{proposition}

We shall see the following relationship between sheaf quantization with/without the $u$-variable. 
Let $p_u\colon \bR_u \times\bR_t \rightarrow \bR_t$ be the projection forgetting $\bR_u$. 
We consider the sheaf $\bK_{t\geq u\geq 0}\in \Sh(\bR_u\times \bR_t)$. 
We have the map $\bK_{t\geq 0}\rightarrow \bK_{t\geq u\geq 0}$ that corresponds to $\id$ under the adjunction isomorphism 
\begin{equation}
    \Hom(\bK_{t\geq 0}, \bK_{t\geq u\geq 0})\cong\Hom(\bK_{t\geq u\geq 0}, \bK_{t\geq u\geq 0}).
\end{equation}
Then we have the induced map
\begin{equation}
    p_u^{-1}\cE=p_u^{-1}\cE\star \bK_{t\geq 0}\rightarrow p_u^{-1}\cE\star\bK_{t\geq u\geq 0}. 
\end{equation}
We denote the fiber (i.e., $[-1]$-shifted cone) of the above map by $C_u(\cE)$.

\begin{proposition}\label{prop:without2with_u}
    Let $\cE\in \mu^\bG(X)$ be a sheaf quantization of $\bm{L}$.
    Then the object $\frakf^L(C_u(\cE))\in \mu(X;u<\infty)$ is a sheaf quantization of $\bm{L}$. 
    Moreover, $\frakf^L\circ C_u$ gives a faithful embedding $\mu_L^\bG(X)\hookrightarrow \mu^\bG_{L}(X;u<\infty)$.
\end{proposition}
\begin{proof}
    The first part is easy. The composition $\Sp_\infty\circ (\frakf^L\circ C_u)$ is the identity of $\mu_L^\bG(X)$. Hence the faithfulness follows. 
\end{proof}

\begin{example}
We give several constructions of sheaf quantizations in our category.
    \begin{enumerate}
        \item For a large class of exact Lagrangians (including all the compact ones) in cotangent bundles, Guillermou~\cite{Guillermou12,Gu19} (later generalized by Jin--Treumann~\cite{JinTreumann}, a Floer-theoretic approach by Viterbo~\cite{viterbo2019sheaf}), constructed their sheaf quantizations in $\mu_L^\bO(X,; u<\infty)$.
        \item  Asano--Ike~\cite{AISQ} sheaf-theoretically constructed for strongly rational Lagrangian immersions in cotangent bundles in $\mu_L^\bZ(X,; u<c)$ for sufficiently small $c$.
        \item For rational Lagrangian immersions with bounding cochains in Weinstein manifolds, Kuwagaki--Petr--Shende~\cite{KPS} constructed their sheaf quantizations in $\mu_L^\bR(X,; u<\infty)$ Floer-theoretically.
        \item For a generic class of spectral curves in cotangent bundles of Riemann surfaces, Kuwagaki constructed their sheaf quantizations in $\mu_L^\bR(X,; u<\infty)$ by using spectral networks~\cite{Kuw1,kuwagaki2024genericexistencewkbspectral}.
        \item Beyond the classes considered in the above, we can, for example, consider the $\Lambda_0$-local systems on a manifold $M$, which are sheaf quantizations of the zero section of $T^*M$ under the correspondence~\cite{KuwNov}. In the terminology of this paper, the \v{C}ech class of such a local system gives a sheaf-theoretic bounding cochain. See the later sections.
    \end{enumerate}
\end{example}

In the next section, we will show that for sufficiently small $c$, every Lagrangian admits a sheaf quantization.

\subsection{Existence of low-energy sheaf quantization of Lagrangian immersions}\label{subsec:SQ_low_energy}

We first consider the case of cotangent bundles.
Let $\bm{L}=(L, \bm{\alpha}, \cL)$ be a Lagrangian brane in $T^*M$, where the underlying Lagrangian immersion $i_L\colon L\looparrowright X$ satisfies \cref{assumption:immersion}. 
In this subsection, we write $M \times \bR_t \times \bR_u$ instead of $M \times \bR_u \times \bR_t$ for notational convenience. 

Consider the pull-back of the Liouville form $i_L^*\lambda$. 
We can take a primitive $f$ of $i_L^*\lambda$ as a multi-valued function on $L$. 
We denote the domain of $f$ as a single-valued function by $\widetilde{L}$, which is a covering of $L$. Hence we have the associated immersion $i_{\widetilde{L}}\colon \widetilde{L}\rightarrow T^*M$. We set
\begin{equation}\label{eq:lift_Lf}
    L_f\coloneqq \lc (x, \xi, t, \tau)\in T^*M\times T_{\tau>0}^*\bR_t\relmid, \exists \widetilde{p}\in \widetilde{L} \text{ s.t.}  -f(\widetilde{p})=t \text{ and } i_{\widetilde{L}}(\widetilde{p})=(x, \xi/\tau)\rc.
\end{equation}

We denote the pull-back of $\partial L$ under the projection $T^*M\bs T^*_MM\rightarrow \partial T^*M$ by $\bR_{>0}\cdot \partial L$. We then set $\overline{L_f}\coloneqq L_f\cup \overline{\bR_{>0}\cdot \partial L}\times T^*_{\bR_t}\bR_t \subset T^*M \times T^*\bR_t$. 
The cone closure of the doubling $\overline{L_f} \ \overline{L_f}$ can be perturbed into the cusp doubling $\overline{L_f}^{\prec}$, which was introduced in \cite{nadler2020sheaf}. 
The advantage of the cusp doubling is that it is finite over the base if $L$ is so. 

We define the period group $\bG_L$ by
\begin{equation}
    \bG_L\coloneqq i_L^*\lambda(H_1(L, \bZ))\subset \bR.
\end{equation}
Note that $L_f$ carries a canonical $\bG_L$-action. 
If the closure of $\bG_L$ is not $\bR$ in the Euclidean topology, we set
\begin{equation}
    c_{L,\min}\coloneqq \inf \lc c\in \bR_{>0}\relmid c\in \bG_L\rc.
\end{equation}

We denote the sheaf of category of sheaves on $M\times \bR_t\times \bR_u$ whose microsupport included in $\overline{L_f}^{\prec}$ by $\sh_{\overline{L_f}^{\prec}}$. The substack whose objects are $0$ on $\bR_{u<0}$ is denoted by $\sh_{\overline{L_f}^{\prec};0}$. We denote the projection $T^*M\times T^*\bR_t\times T^*\bR_u\rightarrow M\times \bR_t\times\bR_u$ by $\widetilde{\pi}$. Note also that the brane structure of $\bm L$ induces an equivariant global section of $\mu_{\overline{L_f}^\prec}$.

\begin{proposition}\label{lem:c-SQcotangent}
    For every end-conic Lagrangian brane $\bm{L}$, there exists a sheaf quantization of $\cS_L\in \mu^{\bG_L}(T^*M; u<c)$ of $\bm{L}$ for a sufficiently small $c$. We can take $c$ as $c_{L, \min}$ if it is defined.
\end{proposition}
\begin{proof}
The case when $\bG_L=\{0\}$, the argument is essentially known~\cite{nadler2020sheaf}. See also the last part of the proof. Hence we assume $\bZ\subset \bG_L$ below. The category $\Sh^{\bG_L}_{\tau>0}(M\times \bR_t\times \bR_u)$ is equivalent to $\Sh^{\bG_L/\bZ}_{\tau>0}(M\times S^1_t\times \bR_u)$. We denote the projection along $M\times \bR_t\times \bR_u\rightarrow M\times S^1_t\times \bR_u$ of $\overline{L_f}$ (resp. $\overline{L_f}^\prec$) by $\overline{L_f^{S^1}}$ (resp. $\overline{L_f^{S^1}}^\prec$), which are closed by the end-conic assumption.

One can perturb the Legendrian boundary of $\overline{L_f^{S^1}}$ in $\partial T^*(M\times S^1_t)$ into a finite position. 
By the usual argument using \cite{GKS} and the microlocal cut-off result~\cite[Lem.~7.18 (v3)]{nadler2020sheaf}, we have an equivalence of sheaves of categories
\begin{equation}\label{eq:antimicro}
    \sh_{\overline{L_f^{S^1}}^{\prec};0}\rightarrow \widetilde{\pi}_*\mush_{\overline{L_f^{S^1}}^\prec}
\end{equation}
along $M\times S^1_t\times \{u=0\}$. Note also that the natural $\bG_L/\bZ$-actions on both sides are compatible under this equivalence.

A Lagrangian brane structure of $L$ gives an $\bG_L$-invariant global section of the right-hand side. Hence it gives a $\bG_L$-invariant global section $s$ of the left-hand side over $M \times S^1_t \times \{u=0\}$. 

Next, we would like to lift the section of $\sh_{\overline{L_f^{S^1}}^{\prec};0}(M \times \bR_t \times \{ u=0\})$ to a section over $M \times \bR_t \times (-c,c)$, which gives an object of $\mu(T^*M; u<c)$.
Here we can adapt an argument similar to \cite[Prop.~2.5.1]{KS90}. 

For any point $q \in \pi(\Lim) \times S^1 \times \{u=0\}$ we can take an open neighborhood $V'_q$ and a section $s'_q$ on $V'_q$ such that $s'|_{V'_q \cap M \times S^1 \times \{0\}} = s'_q|_{V'_q \cap M \times S^1 \times \{0\}}$.
Since $\pi(\Lim)$ is compact, we can take a finite subset $J$ of $\pi(\Lim) \times S^1 \times \{u=0\}$ such that $\{V'_j\}_{j \in J}$ is an open covering. 
For any $q \in \pi(\Lim) \times S^1 \times \{u=0\}$, the set $J(q)=\{j \in J \mid q \in V'_j \}$ is finite.
Hence, for any $q \in \pi(\Lim) \times S^1 \times \{u=0\}$, we can find an open neighborhood $W'_q = W''_q \times (-c_q,c_q)$ such that $s'_j|_{W'_q}=s'_k|_{W'_q}$ for any $j,k \in J(q)$. 
Then $\{ W'_q \}_{q \in \pi(\Lim) \times S^1}$ is an open covering of $\pi(\Lim) \times S^1 \times \{u=0\}$ and $s'_j|_{W'_q}=s'_k|_{W'_q}$ for any $j,k \in J(q)$. 
Again by the compactness of $\pi(\Lim)$, we can take a finite subset $J'$ of $\pi(\Lim) \times S^1 \times \{u=0\}$ such that $\{W'_j \}_{j \in J'}$ is an open covering of $\pi(\Lim) \times S^1 \times \{u=0\}$. 
If we set $c=\min_{j \in J'} c_j$ and $U'_j=V''_j \times (-c,c)$, then $\{U'_j\}_{j \in J}$ is an open covering and $s'_i|_{U'_i \cap U'_j} = s'_j|_{U'_i \cap U'_j}$. If necessary, we refine the covering (using Lebesgue covering) to get a finite Cech nerve. Then using the above identifications, we can take the colimit of $s_i$'s (and its restrictions) along the Cech nerve, hence we get a section over $[0, c)$ for some $c$. 
Incorporating the equivariant structure, this is a desired sheaf quantization.

If $c_{L, \min}>0$, we can go further. Namely, for $0<c<c_{L,\min}$, $\overline{L_f}\cup T_c\overline{L_f}$ can be considered as an image under the contact translation flow of $\overline{L_f}\cup T_{c'}\overline{L_f}$ for $0<c'<c$. 
Hence the usual argument (cf.\ \cite[\S 6.2]{nadler2020sheaf}) works and gives a sheaf quantization provided that there exists a non-trivial object in $\mush_{\overline{L_f}}(\overline{L_f})$. 
Such a non-trivial object is given by the brane structure. 
This completes the proof.
\end{proof}

Next we argue the case of a Lagrangian brane $\bm L$ in a sufficiently Weinstein manifold $X$. Similarly, we obtain $\bG_L, c_{L,\min}, L_f$, and
\begin{equation}
    \overline{L_f} \coloneqq L_f\cup \overline{\bR_{>0}\cdot \partial L}\times T^*_{\bR_t}\bR_t\subset X\times T^*\bR_t.
\end{equation}

Let us choose a strict contact embedding $X\rightarrow T^*\bR^N$ realizing our polarization~$\frakp$. We then set $\bL \coloneqq (\frakp|_{\overline{L_f}\cup \Core(X)\times T^*_{\bR_t}\bR_t})^\prec\subset T^*\bR^N\times T^*\bR_t\times T^*\bR_v$. Note that this doubling procedure is not the doubling procedure to construct sheaf quantization, but for antimicrolocalization. Then we furthermore take the doubling $\bL^\prec\subset T^*\bR^n\times T^*\bR_t\times T^*\bR_{u<c}\times T^*\bR_v$. By the main result of \cite{nadler2020sheaf}, we have a global section of $\mush_{\bL^\prec}$ associated to the brane structure $\bm L$. Now, the proof of \cref{lem:c-SQcotangent} can apply to this case, hence we obtain an object $\cS_L\in \mu^{\bG_L}(T^*M\times T^*\bR_v)$. By taking anti-antimicrolocalization, we obtain a desired object.

\begin{proposition}\label{lem:c-SQ}
    For every end-conic Lagrangian brane $\bm{L}$, there exists a sheaf quantization of $\cS_L\in \mu^{\bG_L}(X; u<c)$ of $\bm{L}$ for a sufficiently small $c$. We can take $c$ as $c_{L, \min}$ if it is defined.
\end{proposition}

\begin{example}
\begin{enumerate}
    \item Suppose that $L$ is an exact Lagrangian submanifold in a cotangent bundle. 
    Then $\bG_L=\bO$ and $c_{L, \min}=\infty$. Hence we get a sheaf quantization in $\Sh_{\tau>0}(M \times \bR_t \times \bR_u)$. Specializing at $\infty$, we get a sheaf quantization in $\Sh_{\tau>0}(M\times \bR_t)$. The resulting object coincides with the one constructed in Guillermou~\cite{Guillermou12,Gu19} and Jin--Treumann~\cite{JinTreumann}, since a sheaf quantizations of exact Lagrangian brane is unique up to shift along $\bR_t$-direction (see e.g., \cite{Gu19}).
    \item Suppose that $L$ is an strongly rational immersed Lagrangian in a cotangent bundle in the sense of \cite{AISQ}. 
    Then we find that $\bG_L \cong \bZ$, and we can take $c_{L, \min}$ as the period in the sense of \cite{AISQ}. 
    Then we get a sheaf quantization in $\Sh^\bZ(M \times \bR_t \times \bR_{u<c_{L, \min}})\cong \Sh(M \times S^1 \times \bR_{u<c_{L, \min}})$, which recovers the result in \cite{AISQ}.
\end{enumerate}
\end{example}
In the next section, we would like to give another argument to construct \emph{standard} sheaf quantization.

\subsection{Classification of low-energy sheaf quantizations}\label{subsec:classification_low_energy_SQ}
For $c>0$, we denote the set of almost isomorphism classes of sheaf quantizations of $\bm L$ in $\mu(X, c)$ by $\frakS(\bm L, c)$. For the definition of almost isomorphism, see \cite{KuwNov}. 

In this section, we prove the following theorem:
\begin{theorem}\label{thm:classification}
    Let $\bm L$ be a Lagrangian brane. There exists $c_{\bm L}>0$ with a map
    \begin{equation}
        \mathrm{cl}: H^1(L, \Lambda_0^+/T^{c_{\bm L}}\Lambda_0)\rightarrow \frakS(\bm L, c_{\bm L}).
    \end{equation}
    If $X$ is sufficiently Weinstein, this map is a bijective map.
\end{theorem}

\begin{definition}
    A sheaf quantization $S\in \mu(X, u<c_{\bm L})$ of $\bm L$ is \emph{geometric} if it is obtained as an image of $\mathrm{cl}$.
\end{definition}
\begin{conjecture}
    For any $X$ and any $\bm L$, any sheaf quantization in $\mu(X, u<c_{\bm{L}})$ is geometric. 
\end{conjecture}
The above theorem says that this conjecture holds true for sufficiently Weinstein $X$. 

Now let us go to the proof of \cref{thm:classification}.
We first start with the construction of $\mathrm{cl}$; for each element of $H^1(L, \Lambda_0^+/T^c\Lambda_0)$, we construct a sheaf quantization.

Let $\bm{L}=(L, \bm\alpha, \cL)$ again be a Lagrangian brane, where the underlying Lagrangian immersion $i_L\colon L\looparrowright X$ satisfies \cref{assumption:immersion}.
We set $\overline{\Lambda_f}^{\prec}\coloneqq \overline{L_f}^{\prec}\cup \mathrm{Core}(X)\times \bR_t\times \bR_u$ 
We denote the full substack of $\mush(X\times T^*(\bR_t\times \bR_u))$ whose microsupport included in $\overline{\Lambda_f}^{\prec}$ by $\mush_{\overline{\Lambda_f}^{\prec}}$. The substack whose objects are $0$ on $\bR_{u<0}$ is denoted by $\mush_{\overline{\Lambda_f}^{\prec};0}$. 
We denote the projection $X \times T^*\bR_t\times T^*\bR_u\rightarrow \mathrm{Core}(X)\times \bR_t\times\bR_u$ by $\widetilde{\pi}$. We now give a different argument for the first half of \cref{lem:c-SQ}. In the following proof, a specific sheaf quantization is constructed.

\begin{proof}[Another proof of the first half of \cref{lem:c-SQ}]
We will follow Nadler--Shende's nearby cycle argument~\cite{nadler2020sheaf}.

Take a point $x\in L$ and a local primitive $f_x$ of $\lambda|_L$ around $x$. Then we can lift $L$ to a Legendrian $L_{f_x}$ in $X \times \bR_\tau$. For a sufficiently small neighborhood $U$ of $\tilde x\in L_{f_x}$ where $\tilde x$ is the lift of $x$ (which is unique up to translation along $\bR_t$-direction), we can take a contact Darboux chart $\bR^{2n+1}$ such that $L_{f_x}$ and $L_{f_x+c}$ for sufficiently small $c>0$ are identified with
\begin{equation}
    \lc x_{n+1}=x_{n+2}=\cdots =x_{2n+1}=0\rc \text{ and} \lc x_{n+1}=x_{n+2}=\cdots =0, x_{2n+1}=c\rc.
\end{equation}
We also have a symplectic Darboux chart $\bR^{2n+1}\times \bR_{\leq 0}$ around $x$ in $X \times T^*\bR_t$ extending the above contact Darboux chart which is given by $\{x_{2n+2}=0\}\subset \bR^{2n+2}$. Then we get a local exact Lagrangian filling of the Legendrian $L_{f_x}\cup L_{f_x+c}$ as
\begin{equation}
     L_{\widetilde x}\coloneqq \lc x_{n+1}=\cdots =x_{2n}=0, x_{2n+2}=x_{2n+1}(x_{2n+1}-c) \relmid 0\leq x_{2n+1}\leq c\rc.
\end{equation}
Since $L_f$ is end-conic and projection-compact one can take such $c$ uniformly over $L$, and denote it by $c_{\bm L}$. If a lift $\tilde y$ of $y\in L$ is the local leaf of $\overline L_f$ on which $\widetilde x$ lives, $L_{\widetilde x}$ and $L_{\widetilde y}$ can glue up (maybe after some perturbation).
Then we take a subset $\widetilde L\coloneqq \bigcup_{c\in \bR}\bigcup_{x\in L} T_cL_{\widetilde x}\subset X\times \bR_{u<c_L}$, which does not depend on the choice of $\widetilde{x}$. The brane structure of $\bm L$ naturally gives an object of $\mush_{\widetilde L}(\widetilde L)$.

Note that as we deform $L$ under the Liouville flow, we obtain $\widetilde L\subset X\times \bR_{u<c_L}\times (0,1]$.
Now we consider the nearby cycle argument. We are would like to use the functor of \cite[Theorem 8.3]{nadler2020sheaf}. Since we cannot directly use this theorem, we would like to comment how we can modify the construction. First, we have to make an antimicrolocalization after embedded in a cosphere bundle of a high dimension.

The key diagram is \cite[(19)]{nadler2020sheaf}, which we reproduced in the below:
\begin{equation*}
\begin{tikzcd}
    \mush_{\Lambda_1}(\Lambda_1) 
    & \sh_{(\Lambda_1, \partial \Lambda_1)^\prec;0}(M \times \bR_{\leq \epsilon}) \ar[l, "\sim"'] \ar[r, hook] 
    & \sh_{(\Lambda_1, \partial \Lambda_1)_\epsilon}(M) \\
    \mush_{\widetilde{\Lambda}}(\widetilde{\Lambda}) \ar[u, "\sim"]
    & \sh_{(\widetilde{\Lambda}, \partial \widetilde{\Lambda})^\prec;0}(M \times (0,1] \times \bR_{\leq \epsilon}) \ar[l] \ar[r] \ar[u, "\sim"'] \ar[d, "\psi"]
    & \sh_{(\widetilde{\Lambda}, \partial \widetilde{\Lambda})_\epsilon}(M \times (0,1]) \ar[u, "\sim"'] \ar[d, "\psi"] \\
    \mush_{\Lambda_0}(\Lambda_0) 
    & \sh_{(\Lambda_0, \partial \Lambda_0)^\prec;0}(M \times \bR_{\leq \epsilon}) \ar[l, "\sim"'] \ar[r, hook]
    & \sh_{(\Lambda_0, \partial \Lambda_0)_\epsilon}(M).
\end{tikzcd}
\end{equation*}
Although we do not have the upper leftmost horizontal equivalence of \cite[(19)]{nadler2020sheaf},  we can lift an object of $\mush_{\widetilde \Lambda}(\widetilde \Lambda)$ to a sheaf, since $\mush_{\Lambda_1}(\Lambda_1)$ is the stalk of $\sh_{(\Lambda_1, \partial \Lambda_1)^\prec;0}(M\times \bR_{\leq \epsilon})$ in the notation of \cite[(19)]{nadler2020sheaf}.
Then the upper middle vertical arrow of \cite[(19)]{nadler2020sheaf} is an equivalence in the present setup as well. Then we can apply $\psi$ and get an object in the desired category without equivariance. Then we can modify the above argument all $\bG$-equivariant, we get a desired object.

This completes the proof.
\end{proof}
\begin{remark}
    This proof might work for any Liouville manifold, since we do not use the isotropicity of the core, although we don't define the microlocal category for Liouville manifold.
\end{remark}

\begin{definition}
    The sheaf quantization of $\bm L$ in $\mu(X; u<c)$ obtained in the above way is said to be the \emph{$c$-standard sheaf quantization}.
\end{definition}

We can also construct more sheaf quantizations by modifying a little the above construction of the standard sheaf quantization as follows: We take $c_L$ as in the proof. 
In the rest of this subsection, let $\{U_i\}_i$ be the Darboux cover. Refining it if necessary, we can assume that the intersections are contractible. 
We also fix an element of $b\in H^1(L, \Lambda_0^+/T^{c_L}\Lambda_0)$ and take a \v{C}ech representative $b_{ij}$ of $b$ with respect to $\{U_i\}_i$. Then $b_{ij}$ can be represented as $b_{ij}=\sum_{0<d<c_L}a^{ij}_dT^{d}$. Then it gives an endomorphism of the section of the Kashiwara--Schapira stack over $\bigcup_{0<d<c_L}L_{\widetilde x+d}$.
Note that $\id+b_{ij}$ is invertible, since the valuation of $b_{ij}$ is greater than $0$.
Hence we can take $\id+b_{ij}$ as gluing morphisms, we get a global section of the Kashiwara--Schapira stack (which is twisted from the above one by the \v{C}ech cocycle). Then we get a sheaf quantization associated to $b$ as in the above proof. This gives $\mathrm{cl}$. This proves the first half of \cref{thm:classification}.

\begin{lemma}\label{lem:uniquenessforsmallc}
    For any sheaf quantizations $\cE_1, \cE_2\in \mu(X;u<c)$ of a given brane structure $\bm{L}$ for some $c\in \bR_{>0}\cup \lc\infty \rc$, there exists $c'$ such that $r_{cc'}(\cE_1)\cong r_{cc'}(\cE_2)$. 
\end{lemma}
\begin{proof}
    Since they are isomorphic on $\{u=0\}$ as global sections of sheaves of \eqref{eq:antimicro}, we can lift the isomorphism for small $c'$. This completes the proof. 
\end{proof}

Let us suppose that $X$ is sufficiently Weinstein. If needed, we perturb it to have a Lagrangian skeleton $\bL$. 
We have the projection $p\colon L\rightarrow \bL$. We further perturb $L$ to satisfy the following: There exists an open covering $\{ U_i\}_i$ of $\bL$ such that each $U_i$ and each connected component of $p^{-1}(U_i)$ are contractible.
We work with this setup. A Lagrangian brane $\bm{L}$ defines a section of $\mush_{\Lim}(\Lim)$. We denote the endomorphism sheaf by $\muhom(\bm{L},\bm{L})$.

\begin{lemma}
For $0<c<c_{\bm L}$. Let $\cE\in \mu(X, u<c)$ be a $c$-standard sheaf quantization. Then $\End(\cE)$ is (almost) isomorphic to $\muhom(\bm L, \bm L)\otimes_\bK\Lambda_0/T^c\Lambda_0$.
\end{lemma}
\begin{proof}
Over each $U_i$, $L$ is exact. Even though $L\cap p^{-1}(U_i)$ is only locally closed, by using the relative doubling technique~\cite[\S 7.3]{nadler2020sheaf}, one can run the same argument to construct a $c$-standard sheaf quantization as the above other proof of \cref{lem:c-SQ}]. Note that $p^{-1}(U_i)\cap L$ is exact, so we don't need to take the union over the $\bR$-copies to run the argument. As a result, the $c$-standard sheaf quantization is of the form $\cE= \bigoplus_{c\in \bR}T_c\cE^{ex}$ over $U_i$. In this argument, $\bm{L}$ gives an object of the upper leftmost category of \cite[(19)]{nadler2020sheaf}, whereas $\cE^{ex}$ gives the corresponding object in the middle. In the exact setting, these categories are equivalent, in particular fully faithful. In other words, the endomorphism ring of $\cE^{ex}$ is isomorphic to $\muhom(\bm L, \bm L)$ over $U_i$. Then the result follows. 
\end{proof}

\begin{remark}
    Even in the case when $X$ is not sufficiently Weinstein, we have a map 
    \begin{equation}
        \muhom(\bm L, \bm L)\otimes_\bK\Lambda_0/T^c\Lambda_0\rightarrow \End(\cE).
    \end{equation}
    We conjecture that this is always an almost isomorphism.
\end{remark}

\begin{lemma}
    Let $\cE$ be a sheaf quantization of $L$. Then $\cE|_{p^{-1}(U_i)}$ is isomorphic to the (unique) standard sheaf quantization.
\end{lemma}
\begin{proof}
By \cref{lem:uniquenessforsmallc}, there exists a sequence of positive real numbers $c_1<c_2<c_3<\cdots$, such that the associated $\sigma$-decomposition consists of the standard sheaf quantizations, where $\sigma\coloneqq \lc 0=c_0<c_1<\cdots <c_n<\cdots\rc $ is a set of positive real numbers less than $c$ without accumulation points.
Since $p^{-1}(U_i)$ is contractible, the $H^1(\End(\cE))$ is almost isomorphic to zero. Hence there are no almost nonzero extensions.
\end{proof}

Let $\cE$ be a sheaf quantization of $\bm L$. Take a sufficiently small $c>0$, then there exists a standard sheaf quantization $\cE^c$. Take an open covering as above, then on each $U_i$ (more precisely, each sheet of it), we get an isomorphism $f_i\colon \cE|_{p^{-1}(U_i)}\cong \cE^c|_{p^{-1}(U_i)}$ lifting an isomorphism of the brane structures. Then, for each $i,j$, we get an isomorphism
\begin{equation}
    f_{ij}\colon \cE|_{p^{-1}(U_i\cap U_j)}\xrightarrow{f_i} \cE^c|_{p^{-1}(U_i\cap U_j)}\xrightarrow{f_j^{-1}}  \cE|_{p^{-1}(U_i\cap U_j)}.
\end{equation}
Hence this gives $f_{ij}\in \Aut(\cE|_{p^{-1}(U_i\cap U_j)})\cong 1+\Lambda_0^+/T^c\Lambda_0$. Hence, we obtain  a \v{C}ech cohomology class in $H^1(L, \Lambda_0^+/T^c\Lambda_0^+)$. Obviously, this gives a inverse of the above construction.
This completes the proof of \cref{thm:classification}.

\subsection{Some properties of standard sheaf quantizations}\label{subsec:properties_standard_SQ}

We list some properties of standard sheaf quantizations introduced above.

\begin{lemma}
    The $c$-standard sheaf quantization of $\bm{L}$ is the sheaf quantization $\mathrm{cl}(0)$.
\end{lemma}
From the construction above, the following is obvious.
\begin{lemma}
    Let $\cE$ be a $c$-standard sheaf quantization. Take $\sigma\coloneqq \{0=c_0<c_1<\cdots <c_n<\cdots\}$. Then each component of the $\sigma$-decomposition of $\cE$ is standard.
\end{lemma}
In particular, the following definition is valid.
\begin{definition}
    Let $\cE_i$ be a $c_i$-standard sheaf quantization for $i=1,2$. We set $\sigma\coloneqq \{0, c_1, c_1+c_2\}$. A \emph{standard extension between $\cE_1$ and $\cE_2$} is an extension between $S_{\sigma}\cE_1$ and $S_{\sigma}\cE_2$ which is $(c_1+c_2)$-standard.
\end{definition}
The following is again obvious.
\begin{lemma}
Let $\bm L$ be a Lagrangian brane. For sufficiently small $c_1, c_2>0$, there exists a standard extension between the $c_1$-standard sheaf quantization $\cE_1$ and the $c_2$-standard sheaf quantization $\cE_2$.
\end{lemma}

\subsection{Abstract intersection Maslov indices}\label{subsec:maslov}

In this subsection, we introduce the notion of abstract intersection Maslov indices and investigate their relationship to sheaf quantizations.

We first recall the construction of the positive definite path. Let $V$ be a symplectic vector space and $\ell$ be a Lagrangian subspace of $V$. Then the tangent space of the Lagrangian Grassmannian $LGr(V)$ at $\ell$ is canonically identified with the space of quadratic forms on $\ell$. The tangent cone spanned by the non-degenerate positive forms is denoted by $C_0(\ell)$.

\begin{definition-lemma}[Positive definite path, transverse case]\label{def-lem:pos_path_trans}
Let $\ell_1, \ell_2$ be two transverse Lagrangians in a symplectic vector space $V$. A path $c\colon [0,1]\rightarrow LGr(V)$ from $\ell_1$ to $\ell_2$ is said to be \emph{positive definite} if it satisfies the following:
\begin{enumerate}
\renewcommand{\labelenumi}{$\mathrm{(\arabic{enumi})}$}
    \item $c(0)=\ell_1$ and $c(1)=\ell_2$,
    \item $c(t)$ is transverse with $\ell_1$ for $t>0$, and
    \item $\dot{c}(0)\in C_0(\ell_1)$.
\end{enumerate} 
Such a path exists and is unique up to homotopy relative to ends.
\end{definition-lemma}

Let $\ell_1, \ell_2$ be two Lagrangians in a symplectic vector space $V$. Then $(\ell_1+\ell_2)/\ell_1\cap \ell_2$ is naturally a symplectic vector space, and a Lagrangian $\ell$ containing $\ell_1\cap \ell_2$ defines a Lagrangian in $(\ell_1+\ell_2)/\ell_1\cap \ell_2$. 

\begin{definition-lemma}[{Positive definite path, clean case, cf.\ \cite{AlstonBao}}]
Let $\ell_1, \ell_2$ be two Lagrangians in a symplectic vector space $V$. A path $c\colon [0,1]\rightarrow LGr(V)$ from $\ell_1$ to $\ell_2$ is said to be \emph{positive definite} if it satisfies the following:
\begin{enumerate}
\renewcommand{\labelenumi}{$\mathrm{(\arabic{enumi})}$}
    \item $c(0)=\ell_1$ and $c(1)=\ell_2$,
    \item $c(t)$ contains $\ell_1\cap \ell_2$ for every $t$, and
    \item $c(t)/\ell_1\cap \ell_2$ is a positive definite path in $LGr((\ell_1+\ell_2)/\ell_1\cap \ell_2)$ from $\ell_1/\ell_1\cap \ell_2$ to $\ell_2/\ell_1\cap \ell_2$ in the sense of \Cref{def-lem:pos_path_trans}.
\end{enumerate} 
Such a path exists is unique up to homotopy relative to ends.
\end{definition-lemma}

Now we recall the definition of the intersection Maslov index. 
Fix a base point $*$ in $LGr(V)$. 
Then the universal covering space $\widetilde{LGr}(V)$ is given by the based path space of $LGr(V)$. 
Suppose that $\widetilde{\ell_1}, \widetilde{\ell_2}\in \widetilde{LGr}(V)$ are points over $\ell_1, \ell_2\in LGr(V)$, respectively. 
By connecting $\widetilde{\ell}_1^{-1}$ and $\widetilde{\ell}_2$ at $*$, we get a path from $\ell_1$ to $\ell_2$. 
Moreover, by connecting it to a positive definite path from $\ell_2$ to $\ell_1$, we get a loop in $LGr(V)$. 
By evaluating it by the Maslov cycle, we get a number, called the intersection Maslov index.

Let $L_1, L_2$ be Lagrangian submanifolds intersecting cleanly in $X$ with $\cC$-brane data where $\cC$ is the coefficient category (cf.\ \cref{rmk:branecoefficient}). Consider the composition
\begin{equation}
c_i\colon L_i\hookrightarrow X\rightarrow B(U/O)\rightarrow B^2\Pic(\cC).
\end{equation}
Since $L_i$ is a Lagrangian, we have a null homotopy of $L_i\rightarrow B(U/O)$, which induces a null homotopy of $c_i$. We denote it by $h_i$. Since $X$ is equipped with a Maslov data, which induces a null homotopy of $c_i$. We denote it by $h_i'$. Hence 
\begin{equation}
    0\overset{h_i}{\sim} c_i\overset{h_i'}{\sim} 0
\end{equation}
gives a map
\begin{equation}
    c_i'\colon L_i\rightarrow B\Pic(\cC)
\end{equation}
for each $i$. The Maslov data of $L_i$ gives a null homotopy $h_i''\colon c_i'\sim 0$.

On the other hand, on the intersection locus $L_1\cap L_2$, the (family of) positive definite path connecting $TL_1$ and $TL_2$ gives a homotopy connecting $h_1$ and $h_2$. Then it induces a homotopy between $h_{12}''$ $c_1'$ and $c_2'$. Hence we get
\begin{equation}
    0\overset{h_1''}{\sim} c_1'\overset{h_{12}'}{\sim} c_2'\overset{h_2''}{\sim} 0
\end{equation}
on $L_1\cap L_2$. This gives a morphism $m_{L_1,L_2}\colon L_1\cap L_2\rightarrow \Pic(\cC)$.

\begin{definition}\label{definition:sheaf_maslov}
\begin{enumerate}
    \item The map $m_{L_1, L_2}$ is called the \emph{abstract intersection Maslov index}. 
    \item When $\cC=\Mod(\bK)$, the abstract intersection Maslov index defines a rank 1 local system over $\bK$ on $L_1\cap L_2$. We denote it by $\cM_{L_1, L_2}$, and call it the \emph{sheaf-theoretic Maslov index local system}. 
\end{enumerate}
\end{definition}

The following is a generalization of what is known in the transverse case~\cite{Ike19}:

\begin{theorem}
    Let $\bm{L}_i=(L_i, \bm{\alpha}_i, \cL_i) \ (i=1,2)$ be embedded Lagrangian branes intersecting cleanly. 
    Then each brane gives an object $\cS_{\bm{L}_i}$ of $\mush_{L_1\cup L_2}(L_1\cup L_2)$. We have the following isomorphism
    \begin{equation}
       \muhom^\bR(\cS_{\bm{L}_1}, \cS_{\bm{L}_2})\coloneqq \cHom_{\mush_{L_1\cup L_2}(L_1\cup L_2)}(\cS_{\bm{L}_1}, \cS_{\bm{L}_2})\cong \cM_{\bm{L}_1, \bm{L}_2}
    \end{equation}
    in $\Loc(L_1\cap L_2)$.
\end{theorem}
\begin{proof}
    We set $L_{12}\coloneqq L_1\cap L_2$. Let $L'_1$ be a small positive definite perturbation of $L_1$ such that $L_1\cap L_1'=L_{12}$. Then one can import the brane structure of $L_1$ to $L_1'$. Then the abstract intersection Maslov index from $L_1$ to $L_1'$ is trivial, and it is easy to see that the corresponding $\muhom$ is the rank 1 constant sheaf. Hence we have done for the case of $L_2=L_1'$. Let us consider the general case. Topologically, $L_1\cup L_1'\cong L_1\cup L_2$ around $L_{12}$. Moreover, $\mush_{L_1\cup L_1'}\cong \mush_{L_1\cup L_2}$ locally around $L_{12}$. 
    Under this equivalence, we import the microsheaf $\cS_{\bm{L}_2}\in \mush_{L_1\cup L_2}(L_1\cup L_2)$ to $\cS_{\bm{L}_2} \in \mush_{L_1'}(L_1')\subset \mush_{L_1\cup L_1'}(L_1\cup L_1')$. 
    Then we have $\mu_{\bm{L}_1'}(\cS_{\bm{L}_2})\cong \pi_{L_1'}^{-1}\cM_{L_1, L_2}$ locally around $L_{12}$, where $\pi_{L_1'}\colon L_1'\rightarrow L_1'\cap L_1$ is the projection defined locally in the tubular neighborhood of  $L_{12}$. By combining this with the first half of the proof, we get the conclusion. 
\end{proof}

We now consider the immersed case. Let $\bm{L}_1, \bm{L}_2$ be cleanly intersecting immersed Lagrangian branes. 
Each connected component of \emph{$L_1\times_{X}L_2$} can be considered as a clean intersection of embedded Lagrangian branes. 
Hence we get an object $\cM_{\bm{L}_1, \bm{L}_2}$ in $\Loc(L_1\times_{X} L_2)$. We denote the canonical map $L_1\times_{X}L_2\rightarrow X$ by $i_{12}$.

\begin{corollary}
    Let $\bm{L}_i=(L_i, \bm{\alpha}_i, \cL_i) \ (i=1,2)$ be Lagrangian branes intersecting cleanly. 
    Then each brane gives an object $\cS_{\bm{L}_i}$ of $\mush_{L^{\mathrm{Im}}_1\cup L^{\mathrm{Im}}_2}(L_1^{\mathrm{Im}}\cup L_2^{\mathrm{Im}})$. We have the following isomorphism
    \begin{equation}
       \muhom^\bR(\cS_{\bm{L}_1}, \cS_{\bm{L}_2})\coloneqq \cHom_{\mush_{L^{\mathrm{Im}}_1\cup L^{\mathrm{Im}}_2}(L_1^{\mathrm{Im}}\cup L_2^{\mathrm{Im}})}(\cS_{\bm{L}_1}, \cS_{\bm{L}_2})\cong i_{12*}\cM_{\bm{L}_1, \bm{L}_2}.
    \end{equation}
    This is an isomorphism of sheaves on $L_1^{\mathrm{Im}}\cap L_2^{\mathrm{Im}}$.
\end{corollary}

In the case $\bK=\bZ$ and $\cC=\Mod(\bZ)$, $\cC$-brane data is what we call brane data with trivial local systems, and we have a fiber sequence
\begin{equation}
    B(\bZ/2)\to \Pic(\cC)\to \bZ.
\end{equation}
The part of $m_{L_1, L_2}$ going to $\bZ$ is the usual intersection Maslov index, denoted by $\alpha_2-\alpha_1$. There exists unique $n\in \bZ$ such that the grading shift of $L_1$ by $n$ makes $m_{L_1,L_2}$ in $B(\bZ/2)$. The resulting  principal $O(1)$-bundle on $L_1\cap L_2$ is denoted by $b_2-b_1$. 

\begin{lemma}
    Let $\bm{L}_i\coloneqq (L_i, \alpha_i, b_i, \cL_i) \ (i=1,2)$ be Lagrangian branes intersecting cleanly. Then
    \begin{equation}
       \cM_{\bm{L}_1, \bm{L}_2}\cong  \cHom(\cL_1,\cL_2)\otimes_\bK(b_2-b_1)[\alpha_2-\alpha_1],
    \end{equation}
    which is a (shifted) local system on $L_1\cap L_2$. 
\end{lemma}

\begin{definition}
    For a Lagrangian brane $\bm{L}=(L, \bm{\alpha}, \cL)$, we set $L_{\mathrm{SI}}\coloneqq L\times_{X}L$ and $\cL_{\mathrm{SI}}\coloneqq \cM_{\bm{L}, \bm{L}}$.
\end{definition}

\begin{remark}\label{rem:Maslov}
    We conjecture that $\cL_{\mathrm{SI}}$ is equal to $\Theta^-$ in immersed Floer theory~\cite{FukayaLagrangianCorresp}. In the situation of \cite[Example~3.12]{FukayaLagrangianCorresp}, it is easy to see that this conjecture is true.
\end{remark}

\subsection{Microlocal extension classes}\label{subsec:microlocal_extension}

Let $\bm{L}=(L, \bm{\alpha}, \cL)$ be a Lagrangian brane and $\cE\in \mu(X; u<c)$ be a simple sheaf quantization of~$\bm{L}$ (see \cref{def:sheaf_quantization}).

\begin{lemma}\label{lemma:microlocalization_along_lag}
    For any $0<u<c$, the microlocalization of $\Sp_u(\cE)$ along $\bm{L}$ is $\cL\oplus \cL[-1]$. 
\end{lemma}
\begin{proof}
We set $A=\rho^{-1}(L)$. Perturbing $AA$ slightly by a Hamiltonian isotopy, we can make the configuration at $u=0$ be of the cusp type as in Nadler--Shende~\cite{nadler2020sheaf}. 
We put the brane structure $(\bm{\alpha}, \cL)$ on $AA_h$ and extend it to $AA_t$.
When crossing the cusp, the Maslov grading changes by 1. 
This completes the proof.
\end{proof}

For $c_1<c$, $\cE_1\coloneqq r_{cc_1}(\cE)$ and $\cE_2\coloneqq S_{-c_1}\Cone(s_{c-c_1}r_{cc_1}(\cE)\rightarrow \cE)[-1]$ are sheaf quantizations of $\bm{L}$ such that $\mu_{\bm{L}}(\cE_i)\cong \bK_L$. Then $\cE \in \mu(X;u<c)$ is an extension of $s_{c-c_1}\cE_1$ by $S_{c_1}\cE_2$ such that $\mu_{\bm{L}}(\cE)\cong \bK_L$. For $c>c_2>c_1$, we have the following exact triangle:
\begin{equation}
    \Sp_{c_2-c_1}(\cE_2)\rightarrow \Sp_{c_2}(\cE)\rightarrow \Sp_{c_1}(\cE_1)\rightarrow.
\end{equation}
Consider the microlocalization along $\bm{L}$ in $\mu(X)$. Then we have the associated exact triangle:
\begin{equation}
    \mu_{\bm{L}}(\Sp_{c_2-c_1}(\cE_2))\rightarrow \mu_{\bm{L}}(\Sp_{c_2}(\cE))\rightarrow \mu_{\bm{L}}(\Sp_{c_1}(\cE_1))\rightarrow,
\end{equation}
which is isomorphic to
\begin{equation}\label{eq:extension}
    \cL\oplus \cL[-1]\rightarrow \cL\oplus \cL[-1]\rightarrow\cL\oplus \cL[-1]\rightarrow
\end{equation}
in $\Loc(L)$ by \cref{lemma:microlocalization_along_lag}.

The space of extensions is described as 
\begin{equation}\label{eq:sp_extension}
    H^1\Hom_{\mush_{L^{\Image}}(L^{\Image})}(\mu_{\bm{L}}(\Sp_{c_1}(\cE_1)), \mu_{\bm{L}}(\Sp_{c_2-c_1}(\cE_2)))\cong  \begin{aligned}
            & \Ext^1(\cL, \cL[-1])\oplus H^1(L_{\mathrm{SI}};\cL_{\mathrm{SI}}[-1]) \\
            &\oplus \Ext^1(\cL, \cL) \oplus H^1(L_{\mathrm{SI}};\cL_{\mathrm{SI}}) \\
            & \oplus \Ext^1(\cL[-1], \cL[-1]) \oplus  H^1(L_{\mathrm{SI}};\cL_{\mathrm{SI}})\\
            &\oplus \Ext^1(\cL[-1], \cL) \oplus  H^1(L_{\mathrm{SI}};\cL_{\mathrm{SI}}[1]).
        \end{aligned}
\end{equation}
Note that the leftmost $\cL[-1]$ in \eqref{eq:extension} comes from the tilted component of the doubling. The right $\cL\oplus \cL[-1]$ and the leftmost $\cL$ comes from the horizontal component.
Hence, for the third and fourth lines of the left-hand side of \eqref{eq:sp_extension}, the microlocalization of the extension class does not have a non-trivial component there. 
Hence the object $\mu_{\bm{L}}(\Sp_{c}(\cE))$ is classified by an element of the first and second lines of the right-hand side of \eqref{eq:sp_extension}. 
We call this element the \emph{microlocal extension} class of $\cE$ at $c_1$.

\begin{lemma}\label{lemma:microlocal_extension_class}
    The microlocal extension class can be normalized into the form
    \begin{equation}\label{eq:microlocalextension}
        (x=1,y=0,z,w) \in 
 \begin{aligned}
            & \Ext^1(\cL, \cL[-1])\oplus H^1(L_{\mathrm{SI}};\cL_{\mathrm{SI}}[-1]) \\
                        &\oplus \Ext^1(\cL, \cL) \oplus H^1(L_{\mathrm{SI}};\cL_{\mathrm{SI}}).
        \end{aligned}
    \end{equation}
\end{lemma}

\begin{proof}
We view $(x,y)$ as a morphism $\cL\rightarrow \cL$. If $(x,y)$ is not an isomorphism, the corresponding extension is not $\cL\oplus \cL[-1]$, which is a contradiction. We can normalize it as $(x,y)=\id=(1,0)$.
\end{proof}

By this lemma, in the following, we regard the microlocal extension class as an element of $\Ext^1(\cL, \cL) \oplus H^1(L_{\mathrm{SI}};\cL_{\mathrm{SI}})=H^1\Hom_{\mush_{L^{\Image}}(L^{\Image})}(\mu_{\bm{L}}(\cE), \mu_{\bm{L}}(\cE))$. 

\begin{definition}
    We say that $\cE\in \mu(X; u<c)$ is the \emph{standard extension} at $c_1$ if the microlocal extension class is zero. Namely, $z=w=0$ in the expression of \cref{lemma:microlocal_extension_class}.
\end{definition}

\subsection{Sheaf-theoretic Fukaya algebra}\label{subsec:sheaf_Fukaya_alg}

Let $\bm{L}$ be a Lagrangian brane. 

Let $\sigma\coloneqq \lc 0=c_0<c_1<c_2<\cdots\rc$ be a discrete submonoid of $\bR$.
In the terminology of \cite{FOOO}, we will consider the setup gapped by $\sigma$. 
By taking $\sigma$ sufficiently fine, we obtain $(c_{i}-c_{i-1})$-standard sheaf quantizations $\cS^i_{\bm{L}}$.
Then we have a morphism
\begin{equation}
    f_{ij}^{\mathrm{st}}\colon S_{\sigma} \cS_{\bm L}^{i}[-1]\rightarrow S_{\sigma} \cS^j_{\bm L}
\end{equation}
for each $i> j$
corresponding to the standard extension. 
This defines a degree 1 endomorphism $f^{\mathrm{st}}$ of the dga $\End\lb \bigoplus_{i\geq 1}S_\sigma \cS_{\bm L}^i\rb$. 

In general, the dga $\End\lb \bigoplus_{i\geq 1}S_\sigma \cS_{\bm L}^i\rb$ is not well-behaved in the following construction. The reason is that one should rather consider the notion of  derived dga (or $E_1$-algebra), since we are working over a ring. Instead, we take a resolution of $\bigoplus_{i\geq 1}S_\sigma \cS_{\bm L}^i$ on which $\Lambda_0$ acts freely. Such a resolution can be easily found through a variant of the Godement resolution. We fix such a model.

To kill some unwanted elements, we redefine $\End\lb \bigoplus_{i\geq 1}S_\sigma \cS_{\bm L}^i\rb$ by $\End\lb \bigoplus_{i\geq 1}S_\sigma \cS_{\bm L}^i\rb_a$. 
\Cref{section:almost} for this operation. See also \Cref{sec:curved_dga} for a curved dga. 

\begin{proposition}
The dga 
\begin{equation}
    (\End\lb \bigoplus_{i\geq 1}S_\sigma \cS_{\bm L}^i\rb\otimes_{\Lambda_0/T^c\Lambda_0}\Lambda_0/\Lambda_0^+, (d+f^{\mathrm{st}})\otimes_{\Lambda_0}\Lambda_0/\Lambda_0^+)
\end{equation}
is isomorphic to the cohomology dga $\Hom_{\mush_{L^{\mathrm{Im}}}(L^{\mathrm{Im}})}(\cL,\cL)$. 
In particular, the pair $(\End\lb \bigoplus_{i\geq 1}S_\sigma \cS_{\bm L}^i\rb, d+f^{\mathrm{st}})$ defines a curved dga. 
Here $d$ denotes the differential of $\End\lb \bigoplus_{i\geq 1}S_\sigma \cS_{\bm L}^i\rb$. We denote the resulting curved dga by $CF_{\mathrm{SQ}}(\bm{L}, \bm{L},\sigma)$.
\end{proposition}

\begin{proof}
Note that \cref{cor:intersectionestimate} below gives an isomorphism $\End\lb \bigoplus_{i\geq 1}S_\sigma \cS_{\bm L}^i\rb\otimes_{\Lambda_0}\Lambda_0/\Lambda_0^+\cong
\Hom_{\mush_{L^{\mathrm{Im}}}(L^{\mathrm{Im}})}(\bigoplus_i\mu_{\bm{L}}(\cS^i), \bigoplus_{i}\mu_{\bm{L}}(\cS^i))$. Then $f^{\mathrm{st}}$ is microlocalized and gives a twisted complex in $\mush_{L^{\mathrm{Im}}}(L^{\mathrm{Im}})$. Each $f_{ij}$ is microlocalized to an isomorphism between $\cL$. Hence the resulting twisted complex is isomorphic to $\cL$. Thus, the endoalgebra is isomorphic to $\Hom_{\mush_{L^{\mathrm{Im}}}(L^{\mathrm{Im}})}(\mu_{\bm{L}}(\cS^i), \mu_{\bm{L}}(\cS^i)) \eqqcolon \Hom_{\mush_{L^{\mathrm{Im}}}(L^{\mathrm{Im}})}(\cL, \cL)$ for some (then any) $i$.
\end{proof}

Let $\sigma'$ be a refinement of $\sigma$. Then there exists a morphism induced by a quasi-isomorphism given by the refinement between $CF_{\mathrm{SQ}}(\bm{L}, \bm{L},\sigma)$ and $CF_{\mathrm{SQ}}(\bm{L}, \bm{L},\sigma')$.
Then it induces a quasi-isomorphism over $\Lambda_0/\Lambda_0^+$. Hence we obtain a quasi-isomorphism from $CF_{\mathrm{SQ}}(\bm{L}, \bm{L},\sigma)$ to $CF_{\mathrm{SQ}}(\bm{L}, \bm{L},\sigma')$.

Let $\Sigma$ be the set of discrete monoids of $\bR$, which forms a filtered poset. We take a cofinal subset $\Sigma'$ such that $CF_{\mathrm{SQ}}(\bm{L}, \bm{L},\sigma)$ is defined for any $\sigma$. We denote the colimit curved dga along $\Sigma'$ by $CF_{\mathrm{SQ}}(\bm{L}, \bm{L})$. Note that it does not depend on the choice of $\Sigma'$.

\begin{conjecture}
    There exists a quasi-isomorphism $CF_{\mathrm{SQ}}(\bm{L}, \bm{L})\simeq CF(\bm{L}, \bm{L})$ of curved $A_\infty$-algebras.
\end{conjecture}

\begin{definition}
A \emph{sheaf-theoretic bounding cochain} of $\bm{L}$ gapped by $\sigma$ is an element $\bm{b}$ of $CF^1_{\mathrm{SQ}}(\bm{L}, \bm{L}, \sigma)\otimes_{\Lambda_0}\Lambda_0^+$ such that it satisfies the Maurer--Cartan equation
    \begin{equation}
        (d+\bm{b})^2=0.
\end{equation}
\end{definition}
For $\sigma'$ refining $\sigma$, a sheaf-theoretic bounding cochain gapped by $\sigma'$ is induced by a sheaf-theoretic bounding cochain gapped by $\sigma$. Hence it also induces a Maurer--Cartan element of $CF^1_{\mathrm{SQ}}(\bm{L}, \bm{L})$.

\subsection{From bounding cochain to sheaf quantization and back}\label{subsec:bc_SQ}

In this subsection, we would like to construct a sheaf quantization from a bounding cochain.

We first fix a submonoid $\sigma=\lc 0=c_0<c_1<c_2\cdots  \rc$ of $\bR$ such that the following holds:
\begin{enumerate}
    \item The $(c_i-c_{i-1})$-standard sheaf quantization  $\cE_i$ exists for any $i$.
    \item For each $i$, there exists a standard extension morphism $\cE_{i+1}[-1]\rightarrow \cE_{i}$.
\end{enumerate}

Let $\frakB(\bm L, \sigma)$ be the set of bounding cochains of $CF_{SQ}(\bm L, \bm L,\sigma)$. We set
\begin{equation}
    \frakB_{\bm L}\coloneqq \bigsqcup_\sigma \frakB(\bm L, \sigma).
\end{equation}
Let $\frakS(\bm L, \sigma)$ is the set of sheaf quantizations of $\bm L$ such that each component of $\sigma$-decomposition is standard.
We set
\begin{equation}
    \frakS_{\bm L}\coloneqq \bigsqcup_\sigma \frakS(\bm L, \sigma).
\end{equation}
Note that any sheaf quantization of $\bm L$ is contained in $\frakS(\bm L, \sigma)$ for some $\sigma$. Hence $\frakS_{\bm L}$ is the (multi-)set of sheaf quantizations.

Our main theorem in this section is the following:
\begin{theorem}\label{thm:bc}
    There exist maps
    \begin{equation}
        \mathrm{real}\colon \frakB_{\bm{L}}\rightarrow \frakS_{\bm{L}},\quad         \mathrm{bc}\colon \frakS_L\rightarrow \frakB_{\bm{L}}
    \end{equation}
    such that $\mathrm{real}\circ \mathrm{bc} =\id$,
\end{theorem}

Let $\bm L$ be a Lagrangian brane. Let $\sigma=\lc 0=c_0<c_1<c_2<\cdots\rc$ be a discrete submonoid of $\bR_{\geq 0}$. Let $\bm b$ be a bounding cochain of $\bm b$.
The bounding cochain defines a morphism
\begin{equation}
    f^{\mathrm{st}}+\bm b\colon \bigoplus_{i\geq 1} S_\sigma\cS_{\bm L} ^i\rightarrow \bigoplus_{i\geq 1} S_\sigma\cS_{\bm L} ^i[1].
\end{equation}
On each component, we can interpret it as a degree $i-j+1$-morphism $(f^{\mathrm{st}}+\bm b)_{ij}$ from
 $S_\sigma\cS_{\bm L} ^i[i-1]$ to $S_\sigma\cS_{\bm L} ^j[j-1]$.

 \begin{lemma}
     If $i\leq j$, $(f^{\mathrm{st}}+\bm b)_{ij}=0$.
 \end{lemma}

\begin{lemma}
    The pair $\lc \{S_\sigma \cS^i_{\bm L}[i-1]\}_{i}, \{(f^{\mathrm{st}}+\bm b)_{ij}\}_{i,j}\rc$ forms a twisted complex. The resulting object $\mathrm{real}(\bm b)$ in $\mu(X,u<c)$ gives a sheaf quantization of $\bm L$.
\end{lemma}
\begin{proof}
Since $\bm b$ is a bounding cochain, it satisfies $(d+f^{\mathrm{st}}+\bm b)^2=0$. On each component, it says that 
\begin{equation}
    d(f^{\mathrm{st}}+\bm b)_{ij}+\sum_{k} (f^{\mathrm{st}}+\bm b)_{kj}\circ (f^{\mathrm{st}}+\bm b)_{ik}=0.
\end{equation}
This is nothing but the Maurer--Cartan equation of the twisted complex. This proves the first part.

Modulo $\Lambda^+_0$, $\bm b=0$. Hence, microlocally, the twisted complex is given by $f^{\mathrm{st}}$. Hence the redundant microsupports are canceled out each other, and we obtain a sheaf quantization.
\end{proof}

We now construct the map $\mathrm{bc}$. Let $\cE$ be a sheaf quantization of $\bm L$. As noted in the above, we have the following:
\begin{lemma}
There exists a sufficiently fine submonoid $\sigma=\lc 0=c_0 < c_1<\cdots\rc$ of $\bR$, one can satisfy the following:
\begin{enumerate}
    \item Each component of $\sigma$-decomposition $\cE_{c_i}$ of $\cE$ is standard.
    \item For each $i$, there exists a standard extension morphism $\cE_{i+1}[-1]\rightarrow \cE_{i}$.
\end{enumerate}
\end{lemma}

By the above lemma, we obtain a twisted complex $\lb \lc S_{\sigma}\cE_{c_i}[i-1]\rc_i ,\lc f_{ij}\rc_{i,j}\rb$. We set
\begin{equation}
    \mathrm{bc}(\cE)\coloneqq \sum_{i>j} f_{ij}-f^{\mathrm{st}}\in CF^1_{SQ}(\bm L, \bm L,\sigma). 
\end{equation}

\begin{lemma}
    The element $\bm b$ is the bounding cochain of $\bm L$.
\end{lemma}
\begin{proof}
    This follows from the proof of the above lemma.
\end{proof}

\begin{proof}[Proof of \cref{thm:bc}]
Now it is only to prove $\mathrm{real}\circ \mathrm{bc}=\id$, which is obvious from the construction.
\end{proof}

\section{Intersection point estimate}\label{sec:energy_cutoff}

In this section, we show that the hom-space between simple sheaf quantizations associated with two Lagrangians gives a lower bound of the number of intersection points. 

\begin{lemma}
    Assume that $X$ is sufficiently Weinstein and let $\cE_i\in \mu(X;u<c')$ be a standard sheaf quantizations of a Lagrangian brane $\bm L_i \  (i=1,2)$. Then 
\begin{equation}
    \Gamma(\Lim_1 \cap \Lim_2;\muhom^\bR(\cE_1, \cE_2)) \cong \Hom_{\mu(X;u<c')}(\cE_1,\cE_2)_a\otimes \Lambda_0/\Lambda_0^+.
\end{equation}
\end{lemma}
\begin{proof}
The microlocalization induces the morphism     $\Hom_{\mu(X;u<c')}(\cE_1,\cE_2) \to \Gamma(\Lim_1 \cap \Lim_2;\muhom^\bR(\cE_1, \cE_2))$. We will see that this morphism induces a desired isomorphims. Since we are working with sufficiently Weinstein manifolds and Lagrangian submanifolds, we can antimicrolocalize everything and we reduce the problem to the case of cotangent bundles. By a generic position argument, we can assume that our sheaf quantizations are of the form of direct sums of the doubling version of sheaves like $\bigoplus_{c\in \bR}\bK_{f\geq c}$. Then it is straightforward to see
\begin{equation}
    \Hom_{\mu(X;u<c')}(\cE_1,\cE_2)_a\otimes \Lambda_0/\Lambda_0^+\cong \lim_{c'\rightarrow 0}\Hom_{\mu(X;u<c')}(\cE_1,\cE_2)_a.
\end{equation}
The right hand side is nothing but our definition of $\mu hom^\bR$.
\end{proof}

\begin{corollary}\label{cor:intersectionestimate}
    Assume that $X$ is sufficiently Weinstein and let $\cE_i\in \mu(X;u<c)$ be a sheaf quantizations of a Lagrangian brane $\bm L_i \ (i=1,2)$. Then 
\begin{equation}
    \Gamma(\Lim_1 \cap \Lim_2;\muhom^\bR(\cE_1, \cE_2)) \cong \Hom_{\mu(X;u<c)}(\cE_1,\cE_2)_a\otimes \Lambda_0/\Lambda_0^+.
\end{equation}
\end{corollary}
\begin{proof}
Take sufficiently small $c'$ such that $\cE_1,\cE_2$ are both standard, then it follows from the above lemma.
\end{proof}

\begin{corollary}\label{cor:intersection_estimate}
    Assume that $X$ is sufficiently Weinstein and let $\cE_i\in \mu(X;u<c)$ be a sheaf quantizations of a Lagrangian brane $\bm L_i \ (i=1,2)$. 
    Moreover, assume that $L_1$ cleanly intersects with $L_2$. 
\begin{enumerate}
    \item One has
    \begin{equation}
       \sum_{k\in \bZ}\dim_\bK H^k\Hom_{\mu(X;u<c)}(\cE_1,\cE_2)_a\otimes \Lambda_0/\Lambda_0^+
        = 
        \sum_{k \in \bZ} \dim_\bK H^k(L_1\times_{X} L_2; \cM_{\bm{L}_1, \bm{L}_2}),
    \end{equation}
    where $\cM_{\bm{L}_1, \bm{L}_2}$ is the sheaf-theoretic Maslov index local system (see \cref{definition:sheaf_maslov}). 
    In particular, in the case $c=\infty$, 
    \begin{equation}
        \sum_{k \in \bZ} \dim_\bK H^k(L_1\times_{X} L_2; \cM_{\bm{L}_1, \bm{L}_2})
        \geq 
        \sum_{i\in \bZ} \dim_\Lambda H^i \Hom_{\mu(X;u<\infty)}(\cE_1, \cE_2)\otimes_{\Lambda_0}\Lambda.
    \end{equation}

    \item Moreover, assume that (1) $L_1$ and $L_2$ are embedded, (2) $L_1$ and $L_2$ intersect transversely, and (3) $\cE_1$ and $\cE_2$ are simple sheaf quantization.
    Then
    \begin{equation}
        \sum_{i\in \bZ}\dim_\bK H^i\Hom_{\mu(X;u<c)}(\cE_1,\cE_2)_a\otimes \Lambda_0/\Lambda_0^+ = \# (L_1\cap L_2).
    \end{equation}
    In particular, in the case $c=\infty$, 
    \begin{equation}
       \# (L_1\cap L_2) \geq \sum_{i\in \bZ} \dim_\Lambda H^i \Hom_{\mu(X;u<\infty)}(\cE_1, \cE_2)\otimes_{\Lambda_0}\Lambda.
    \end{equation}
\end{enumerate}
\end{corollary}
\begin{proof}
    The first part of the both statements follows from the $\muhom$ computation in \S\ref{subsec:maslov} and \cref{cor:intersectionestimate}. We now consider the second part of the statements concerning $\Lambda$. Note first that 
    \begin{equation}
        \Hom_{\mu(X;u<+\infty)}(\cE_1,\cE_2)_a\otimes_{\Lambda_0} \Lambda\cong \Hom_{\mu(X;u<+\infty)}(\cE_1,\cE_2)\otimes_{\Lambda_0} \Lambda.
    \end{equation}
    In \cite{KuwNov}, the hom-spaces are proved to be derived complete.
    Hence the derived Nakayama lemma~\cite[0G1U]{stacksproject} tells us that
    \begin{equation}
        \sum_{i\in \bZ} \dim_\Lambda H^i\Hom_{\mu(X;u<+\infty)}(\cE_1,\cE_2)_a\otimes_{\Lambda_0} \Lambda\leq \sum_{k\in \bZ}\dim_\bK H^k\Hom_{\mu(X;u<c)}(\cE_1,\cE_2)_a\otimes \Lambda_0/\Lambda_0^+.
    \end{equation}
    This completes the proof.
\end{proof}

\section{Metric on the objects}\label{section:metric}
We discuss the metric on the set of objects of categories following \cite{KSpersistent, AI20}. Our treatment is closer to that of \cite{FukayaGromov}.
Then the completeness with respect to the interleaving distance, which was studied in Asano--Ike~\cite{AsanoIkecomplete} and Guillermou--Viterbo~\cite{GV2022singular}, will be discussed in our setup. 

In this section, we take $\bG=\bR$ (if not, it will produce a non-interesting theory).

\subsection{\texorpdfstring{$\Lambda_0$}{Lambda0}-linear category}

Let $\cD$ be a triangulated $\Lambda_0$-linear category. 
Note that a $\Lambda_0/T^c\Lambda_0$-linear category can be viewed as a $\Lambda_0$-linear category.
The following three definitions are analogues of the notion in \cite{AI20,AsanoIkecomplete}. 

\begin{definition}
Let $\cE\in \cD$ be an object. Then $\cE$ is said to be an \emph{$a$-torsion}, if $T^a\id\in \End(\cE)$ is zero.
\end{definition}

\begin{definition}[$c$-isomorphism]
Let $\cE, \cF\in \cD$ and $c\in \bR_{\geq 0}$. 
The objects $\cE$ and  $\cF$ are said to be \emph{$c$-isomorphic} if there exist closed morphisms $\alpha\colon \cE\rightarrow \cF$ and $\beta\colon \cF\rightarrow \cE$ such that 
\begin{enumerate}
    \item[(1)] the composite $\cE\xrightarrow{\alpha}\cF\xrightarrow{\beta}\cE$ is equal to $T^{c}\id_{\cE}$ and
    \item[(2)] the composite $\cF\xrightarrow{\beta}\cE\xrightarrow{\alpha}\cF$ is equal to $T^{c}\id_\cF$.
\end{enumerate}
In this case, the pair of morphisms $(\alpha,\beta)$ is called a $c$-isomorphism. In this situation, we also say that $\alpha$ is \emph{$c$-invertible} and $\beta$ is a $c$-inverse of $\alpha$. 
\end{definition}

\begin{remark}[Weak $(a,b)$-isomorphism]\label{remark:weak_isom}
    Following \cite{AsanoIkecomplete}, one can also define weak $(a,b)$-isomorphism as follows.
    Let $\cE, \cF\in \cD$ and $a, b\in \bR_{\geq 0}$. The pair $(\cE, \cF)$ is said to be weakly $(a,b)$-isomorphic if there exist closed morphisms $\alpha, \delta\colon \cE\rightarrow \cF$ and $\beta, \gamma\colon \cF\rightarrow \cE$ such that 
    \begin{enumerate}
        \item[(1)] the composite $\cE\xrightarrow{\alpha}\cF\xrightarrow{\beta}\cE$ is equal to $T^{a+b}\id_{\cE}$,
        \item[(2)] the composite $\cF\xrightarrow{\gamma}\cE\xrightarrow{\delta}\cF$ is equal to $T^{a+b}\id_{\cF}$, and
        \item[(3)] $T^a\alpha=T^a\delta$ and $T^b\beta=T^b\gamma$.
    \end{enumerate}
    If $(\cE,\cF)$ is weakly $(a,b)$-isomorphic, then $\cE$ and $\cF$ are $2(a+b)$-isomorphic.    
\end{remark}

The proofs of the following can be easily read off from the corresponding proofs of \cite{AsanoIkecomplete}, hence omitted. 
For \cref{lem:torsion_isom}, see also \cref{remark:weak_isom}. 

\begin{lemma}[{\cite[Lem.~4.1]{AsanoIkecomplete}}]
Let $\cE\in \cD$ and $a\in \bR_{\geq 0}$, consider the exact triangle
\begin{equation}
    \cE\xrightarrow{T^a}\cE\rightarrow C\rightarrow.
\end{equation}
Then $C$ is $2a$-torsion.
\end{lemma}

\begin{lemma}[{cf.\ \cite[Lem.~3.4]{AsanoIkecomplete}}]\label{lem:torsion_isom}
    Let $\cE \to \cF \to \cG \to$ is an exact triangle in $\cD$ and assume $\cE$ is $c$-torsion. 
    Then $\cF$ and $\cG$ are $2c$-isomorphic.
\end{lemma}

\begin{lemma}[{\cite[Prop.~4.2]{AsanoIkecomplete}}]
    Assume that $(\cE, \cF)$ is $c$-isomorphic. Let $(\alpha, \beta)$ be a $c$-isomorphism on $(\cE, \cF)$. 
    Then the cone of $\alpha$ is $3c$-torsion.
\end{lemma}

\subsection{Interleaving distance}
We interpret the interleaving distance of \cite{AI20,AsanoIkecomplete} in terms of the Novikov ring.
\begin{definition}[Interleaving distance]
Let $\cD$ be a triangulated $\Lambda_0$-linear category. 
For objects $\cE, \cF$ of $\cD$, the interleaving distance $d_{I}(\cE, \cF)$ between them is the infimum of $\epsilon\geq 0$ satisfying the following: There exist closed morphisms $\alpha\in \Hom^0_{\cD}(\cE, \cF)$ and $\beta\in \Hom^0_{\cD}(\cF, \cE)$ such that
    \begin{equation}
    \begin{split}
        &\beta\circ \alpha=T^{\epsilon}\id_{\cE},\\
        &\alpha\circ \beta=T^{\epsilon}\id_{\cF}.
    \end{split}
    \end{equation}
It is easy to see that $d_{I}$ is a pseudo-distance on $\cD$.
\end{definition}

The following is also a straightforward generalization of the result in \cite{AsanoIkecomplete}.
\begin{theorem}\label{lem:complete_nonsense}
Let $\cD$ be a triangulated $\Lambda_0$-linear category with arbitrary coproduct. Then $d_I$ is a complete pseudo-metric.
\end{theorem}

\begin{definition}
    We denote the completion of $\mu^{\mathrm{sm}}(X;u<c)$ in $\mu(X;u<c)$ with respect to $d_I$ by $\mu^{\mathrm{lsm}}(X;u<c)$. Here $\mathrm{lsm}$ stands for ``limit-smooth".
\end{definition}

\begin{remark}
    Since arbitrary colimit does not necessarily exist in $\mu(X;u<c)$, the ``completion" in the above means that the inclusion of all the existing colimits of Cauchy sequences. In particular, we do not use Theorem~\ref{lem:complete_nonsense} for our microlocal categories.
\end{remark}

\begin{remark}
We can also discuss the Gromov--Hausdorff distance in the sense of Fukaya~\cite{FukayaGromov} in this setting (up to slight modification). An advantage of sheaf theory (already evident in \cite{AsanoIkecomplete} and \cite{GV2022singular}) is that the sheaf category already contains limits, and hence the limiting category is more concrete than the approach in \cite{FukayaGromov}.
\end{remark}

\subsection{Interleaving distance on the non-equivariant category}

To relate our interleaving distance with that of \cite{AI20}, we recall the formulation in loc.\ cit.
\begin{definition}[$(a,b)$-isomorphism]
Let $\cE, \cF\in \mush_{>0}(X\times \bR_{u<c}\times \bR_t)$ and $a, b \in \bR_{\geq 0}$. The pair $(\cE, \cF)$ is said to be $(a,b)$-isomorphic if there exist closed morphisms $\alpha\colon \cE\rightarrow T_a\cF$ and $\beta\colon \cF\rightarrow T_b\cE$  such that \begin{enumerate}
    \item[(1)] the composite $\cE\xrightarrow{\alpha}T_a\cF\xrightarrow{\beta}T_{a+b}\cE$ is equal to $T^{a+b}$ and 
    \item[(2)] the composite $\cF\xrightarrow{\beta}T_b\cE\xrightarrow{\alpha}T_{a+b}\cF$ is equal to $T^{a+b}$.
\end{enumerate}
We also set 
\begin{equation}
    d(\cE, \cF)\coloneqq \inf\lc a+b\relmid a,b \in \bR_{\ge 0}, \text{the pair $(\cE,\cF)$ is $(a,b)$-isomorphic}\rc.
\end{equation}
\end{definition}

\begin{lemma}
For $\cE, \cF\in  \mush_{>0}(X\times \bR_{u<c}\times \bR_t)$, we have
\begin{equation}
    d_{I}(\frakf^L(\cE), \frakf^L(\cF))\leq d(\cE, \cF).
\end{equation}
\end{lemma}

The following is also easy:
\begin{lemma}\label{lem:comparison2}
For $\cE, \cF\in \mu(X;u<c)$, we have
\begin{equation}
    d_{I}(\cE, \cF)\geq d(\frakf(\cE), \frakf(\cF)).
\end{equation}
\end{lemma}

\subsection{Fukaya's Hofer distance}
We can consider a slightly weaker version of the distance.

\begin{definition}[Hofer distance~\cite{FukayaGromov}]
Let $\cD$ be a triangulated $\Lambda_0$-linear category. For objects $c_1, c_2$ of $\cD$, the interleaving distance $d_{H}(c_1, c_2)$ between them is the infimum of $\epsilon\geq 0$ satisfying the following: There exist closed morphisms $\alpha\in F^{-\epsilon_1}\Hom^0_{\cD\otimes_{\Lambda_0}\Lambda}(c_1, c_2)$ and $\beta\in F^{-\epsilon_2}\Hom^0_{\cD\otimes_{\Lambda_0}\Lambda}(c_2, c_1)$ such that
    \begin{equation}
    \begin{split}
        &\beta\circ \alpha=\id_{c_1},\\
        &\alpha\circ \beta=\id_{c_2}.
    \end{split}
    \end{equation}
and $\epsilon_1+\epsilon_2\leq \epsilon$. It is easy to see that $d_{H}$ is a pseudo-distance on $\cD$.
\end{definition}

\begin{remark}
    The definition is motivated by the Hamiltonian invariance of the Floer cohomology over the Novikov field.
\end{remark}

The following is easy:

\begin{lemma}
    For $\cE, \cF\in \cD$, we have
    \begin{equation}
           d_H(\cE, \cF)\leq d_I(\cE, \cF).
    \end{equation}
\end{lemma}

\section{Hamiltonian automorphism}\label{section:hamiltonian_auto}

In this section, we explain that the energy stability result~\cite{AI20,AsanoIkecomplete} also holds in our microlocal category.

\begin{remark}
    The reason why we assume sufficiently Weinstein and Lagrangian support in this section is due to the absense of the theory of GKS-isotopy for the microsheaf categories. More precisely, constructing the GKS-kernel is not difficult with the technology of \cite{nadler2020sheaf}, but we do not have the full understanding of the Morita theory of microsheaves. Some partial results in this direction can be found, for example, in \cite{kuo2024duality}.
\end{remark}

\subsection{Quantized Hamiltonian transform in Weinstein context}

Let $(X,\lambda)$ be a sufficiently Weinstein manifold and $I$ be an open interval containing the closed interval $[0,1]$.
Let $H \colon X \times I \to \bR$ be a compactly supported Hamiltonian function and denote by $\Phi \colon X \times I \to X$ the associated Hamiltonian isotopy. 
We consider an action of $\Phi_s$ on our microlocal category.

We recall the setup of antimicrolocalization (see \cref{subsec:antimicrolocalization}).
We first consider the contactification $X\times \bR$. For a sufficiently large $N$, we can contactomorphically embed $X\times \bR$ into $S^*\bR^N$. We can choose the embedding of $S^*\bR^N\hookrightarrow T^*\bR^N$ in a way that the contact form on $S^*\bR^N$ is pulled back to $\lambda$ via the embedding. As a result, we obtain a symplectic embedding $X\hookrightarrow T^*\bR^N$ preserving the Liouville forms. As in \cite{ShendeH-principle, nadler2020sheaf}, we choose a polarization $\frakp$. Namely, take a tubular neighborhood $U$ of the embedding, and take a Lagrangian distribution $\frakp$ of $p\colon U\rightarrow X$.

We define $\widetilde{H} \coloneqq H \circ p \colon U \times I\to \bR$, where $p \colon U \to X$ is the projection.
Then we extend $\widetilde{H}$ to $T^*\bR^N \times I$ with some bump function. 
After conifiying $\widetilde{H}$ as $\tau \widetilde{H}(y;\eta/\tau)$, we can quantize the associated Hamiltonian isotopy $\Phi^{\widetilde{H}}$ as $\cK^{\widetilde{H}} \in \Sh_{>0}(\bR^N \times \bR^N \times \bR_t \times I)$ by the GKS theorem~\cite{GKS}.

Let $L_1,\dots, L_n$ be a finite collection of Lagrangian submanifolds in $X$. Using the polarization $\frakp$, we defined $\bL^\prec$. We then have the antimicrolocalization embedding:
\begin{equation}
    \mu_{\bigcup_iL_i}(X;u<c) \cong \mu^{A\mu}_{\bigcup_iL_i}(X;u<c).
\end{equation}
The comp-composition with the GKS kernel $\cK^{\widetilde{H}}$ gives an equivalence:
\begin{equation}
    \cK^{\widetilde{H}}_s \colon \Sh_{\bL^\prec, \tau>0}^\bR(\bR^N\times \bR_t\times\bR_{u}\times \bR_{v<a})_0\rightarrow \Sh_{\bL^\prec, \tau>0}^\bR(\bR^N\times \bR_t\times\bR_{u}\times \bR_{v<a})_0
\end{equation}
for $s \in I$ and for any $a>0$.
This induces the following functor under the antimicrolocalization embedding and the microlocalization:
\begin{equation}
    \mu(X;u<c)\supset \mu_{\bigcup_{i=1}^nL_i}(X;u<c)\xrightarrow{\cong}\mu_{\bigcup_{i=1}^n\Phi_s(L_i)}(X;u<c)\subset \mu(X;u<c).
\end{equation}
The equivalence is proved in \cite{LiNadlerShende}. Strictly speaking, \cite{LiNadlerShende} only proves in the non-equivariant case, but the adaption to the equivariant case is straightforward.
For $s \in I$, the functor $\cK^{\widetilde{H}}_s$ depends only on $H$, thus by observing the obvious compatibility, we obtain the following functor:
\begin{equation}
        \cK^H_s \colon \mu^{\mathrm{sm}}(X;u<c) \rightarrow \mu^{\mathrm{sm}}(X;u<c).
\end{equation}

\begin{lemma}[{cf.~\cite[Prop.~5.21]{AsanoIkecomplete}}]
	The functor $\cK^H_1$ depends only on the time-1 map~$\phi^H_1$. 
\end{lemma}
\begin{proof}
Let $H_1, H_2$ be different Hamiltonian functions whose time-1 Hamiltonian flow maps coincide. Then, on $U$, the time-1 maps of the pulled-back isotopies are the same. 
Hence the statement holds from the same result for the cotangent bundle~\cite{AsanoIkecomplete}.
\end{proof}

For a compactly supported Hamiltonian diffeomorphism $\varphi$ of $X$, we can define a functor $\cK_\varphi \colon \mu^{\mathrm{sm}}(X;u<c)\rightarrow \mu^{\mathrm{sm}}(X;u<c)$ by choosing a Hamiltonian function $H \colon X \times I \to \bR $ such that $\varphi=\phi^H_1$ and setting $\cK_\varphi \coloneqq \cK^H_1$.
This is well-defined by the above lemma.
The following properties are deduced from the corresponding properties of GKS kernels.

\begin{proposition}
	Let $\varphi$ and $\varphi'$ be compactly supported Hamiltonian diffeomorphisms of~$X$.
    Then the following hold:
	\begin{enumerate}
		\item $\cK_{\varphi \circ \varphi'} \cong \cK_{\varphi} \circ \cK_{\varphi'}$,
		\item $\cK_\varphi$ is an equivalence and $\cK_{\varphi^{-1}}$ is its inverse, and 
		\item $\musupp(\cK_{\varphi}(\cE))=\varphi(\musupp(\cE))$ for $\cE\in \mu^{\mathrm{sm}}(X;u<c)$. 
	\end{enumerate}
\end{proposition}

\subsection{Hamiltonian automorphism and energy stability}

We can prove the following Hamiltonian stability result also in the sufficiently Weinstein setting.

\begin{theorem}[{cf.~\cite[Thm.~5.11]{AsanoIkecomplete}}]\label{thm:Hamiltonianstability}
	For a compactly supported Hamiltonian diffeomorphism $\varphi$ of $X$ and $\cE \in \mu^{\mathrm{sm}}(X;u<c)$, one has 
	\begin{equation}
		d_I(\cE,\cK_\varphi(\cE)) \leq \|\varphi\|_H,
	\end{equation}
	where the right-hand side is the Hofer norm of $\varphi$ defined by 
	\begin{equation}\label{eq:Hofer_norm}
		\|\varphi\|_H
		\coloneqq 
		\inf \lc \int_0^1 \lb \max_{p \in X} H_s(p) -\min_{p \in X} H_s(p) \rb ds
		\relmid 
		\begin{aligned}
			& H \colon X \times I \to X \\ 
			& \text{is compactly supported} \\
			& \text{and $\varphi = \phi^H_1$}            
		\end{aligned}
		\rc.
	\end{equation}
\end{theorem}
\begin{proof}
	Let $H \colon X \times I \to \bR$ be a compactly supported Hamiltonian function. 
	It is enough to prove that 
	\begin{equation}
		d_I(\cE,\cK^H_1(\cE)) \leq  \int_0^1 \lb \max_{p \in X} H_s(p) -\min_{p \in X} H_s(p) \rb ds.
	\end{equation} 
	Recall that the functor $\cK^H_1$ is defined as the composition with $\cK^{\widetilde{H}}_1$, where $\widetilde{H}$ is an extension $\widetilde{H}$ of $H$ by some bump function. In particular, $\min_{p \in X} H_s(p)=\min_{p \in T^*\bR^N} \widetilde H_s(p)$ and $\max_{p \in X} H_s(p)=\max_{p \in T^*\bR^N} \widetilde H_s(p)$, which implies 
    \begin{equation}
		\int_0^1 \lb \max_{p \in T^*\bR^N} \widetilde{H}_s(p) -\min_{p \in T^*\bR^N} \widetilde{H}_s(p) \rb ds 
		= 
		\int_0^1 \lb \max_{p \in X} H_s(p) -\min_{p \in X} H_s(p) \rb ds.
	\end{equation}
	By \cite{AsanoIkecomplete}, we have 
	\begin{equation}
		d(\cK^0_1,\cK^{\widetilde{H}}_1) \leq \int_0^1 \lb \max_{p \in T^*\bR^N} \widetilde{H}_s(p) -\min_{p \in T^*\bR^N} \widetilde{H}_s(p) \rb ds.
	\end{equation}	
	Since 
	\begin{equation}
		d_{I}(\cK_1^0(\cE),\cK^{\widetilde{H}}_1 (\cE)) \leq d(\cK^0_1,\cK^{\widetilde{H}}_1)
	\end{equation}
	and $\cK^0_1$ is the identity functor, we obtain the result.
\end{proof}

\begin{corollary}\label{cor:8.4}
For any object $\cE\in \mu^{\mathrm{sm}}(X;u<c)$ and any compactly supported Hamiltonian diffeomorphism $\varphi$, we have $\cE\cong \cK_\varphi(\cE)$ in $\mu(X;u<c)\otimes_{\Lambda_0} \Lambda$.
\end{corollary}

Note that this corollary is only meaningful when $c=\infty$. Otherwise, the claimed equality means $0\cong 0$.

\begin{proof}[Proof of \cref{cor:8.4}]
By \cref{thm:Hamiltonianstability}, there exists $\alpha \colon \cE \rightarrow \cK_\varphi(\cE)$ and $\beta\colon \cK_\varphi( \cE)\rightarrow \cE$ such that $\alpha\circ \beta=T^{\epsilon}, \beta \circ \alpha=T^{\epsilon}$ for some $\epsilon\geq 0$. After the base change to $\Lambda$, we can rewrite the equalities as $(T^{\epsilon}\alpha)\circ \beta=\id, \beta\circ (T^{-\epsilon}\alpha)=\id$. This completes the proof.
\end{proof}

\subsection{Hamiltonian automorphisms on the completed category}

In the following, we consider an action of the group of hameomorphisms on $\mu^{\mathrm{lsm}}(X;u<\infty)$. 

An isotopy of homeomorphisms $\phi=(\phi_s)_s \colon X \times I \to X$ of $X$ is said to be a \emph{continuous Hamiltonian isotopy} if there exist a compact subset $C \subset X$ and a sequence of smooth Hamiltonian functions $(H_n \colon X \times I \to \bR)_{n \in \bN}$ supported in $C$ satisfying the following two conditions.
\begin{itemize}
    \item[(1)] The sequence of flows $(\phi^{H_n})_{n \in \bN}$ $C^0$-converges to $\phi$, uniformly in $s \in I$.
    \item[(2)] The sequence $(H_n)_{n \in \bN}$ converges uniformly to a continuous function $H \colon X \times I \to \bR$. 
    That is, $\|H_n-H\|_\infty \coloneqq \sup_{p \in X}|H_n(p)-H(p)| \to 0$.
\end{itemize}
A homeomorphism of $X$ is called a \emph{hameomorphism} if it is the time-1 map of a continuous Hamiltonian isotopy.
The group of hameomorphisms is denoted by $\Hameo(X)$.

\begin{proposition}
    The group $\Hameo(X)$ acts on the category $\mu^{\mathrm{lsm}}(X;u<\infty)$.
\end{proposition}
\begin{proof}
	For a hameomorphism $\varphi_\infty$, take a sequence $(H_n)_n$ satisfying the above conditions. 
	Consider the sequence of extensions $(\widetilde{H}_n)_n$ on $T^*\bR^N \times I$.
	Then one can obtain a kernel $\cK$ as a distance limit of $\cK^{\widetilde{H}_n}_1$. Since $\cK_1^{\widetilde H_n}$ restricted to $U\times U\times I$ does not depend on the choices of extensions, the resulting $\cK$ also does not depend on those choices.
    Hence the action of $\cK$ on $\mu^{\mathrm{lsm}}(X;u<\infty)$ depends only on $\varphi_\infty$. 
    The functor $\cK$ preserves $\mu^{\mathrm{lsm}}(X;u<\infty)$.
\end{proof}

\section{Sheaf quantization of limits of Lagrangian branes}\label{section:limit_SQ}

In this section, we discuss sheaf quantization of limits of Lagrangian branes, which is previously discussed in \cite{GV2022singular, AsanoIkecomplete}. 
We assume that $X$ is sufficiently Weinstein throughout this section.

To simplify the following discussion, we introduce some terminologies.

\begin{definition}
A morphism $f\colon \cE\rightarrow \cF$ in $\mu(X;u<\infty)$ is said to be an \emph{isomorphism over $\Lambda$} if it induces an isomorphism in $\mu(X;u<\infty)\otimes_{\Lambda_0}\Lambda$.
\end{definition}
Note that a $c$-invertible morphism for some $c\geq 0$ is an isomorphism over $\Lambda$.

\begin{definition}
Let $\bm{L}$ (resp.\ $\bm{L}_i$) be a Lagrangian brane with a bounding cochain $\bm{b}$ (resp.\ $\bm{b}_i$). 
Assume that the corresponding sheaf quantization satisfies that $\End^0(\cS_{\bm{L}})$ is almost isomorphic to $\Lambda_0\id$.
\begin{enumerate}
    \item If $(\bm{L}_1, \bm{b}_1)$ and $(\bm{L}_2, \bm{b}_2)$ are Hamiltonian isotopic by a compactly supported Hamiltonian diffeomorphism $\varphi$, one writes $(\bm{L}_1, \bm{b}_1) \sim_\varphi (\bm{L}_2, \bm{b}_2)$.
    \item One denotes by $\frakL(\bm{L}, \bm{b})$ the set of Lagrangian branes with bounding cochains which are isotopic to $(\bm{L}, \bm{b})$ by a compactly supported Hamiltonian diffeomorphism.
    \item One sets
    \begin{equation}
        d((\bm{L}_1, \bm{b}_1),(\bm{L}_2, \bm{b}_2))\coloneqq \inf\lc \|\varphi\|_H \relmid (\bm{L}_1, \bm{b}_1) \sim_\varphi (\bm{L}_2, \bm{b}_2)\rc,
    \end{equation}
    where $\|\varphi\|_H$ is the Hofer norm defined as in \eqref{eq:Hofer_norm}. 
    The metric completion of $\frakL(\bm{L}, \bm{b})$ is denoted by $\widehat{\frakL}(\bm{L}, \bm{b})$ .
\end{enumerate}
\end{definition}

Now we can state our theorem in this section:

\begin{theorem}\label{thm:limitSQ}
    There exists a canonical functor $\widehat{\frakL}(\bm{L}, \bm{b})\rightarrow \mu^{\mathrm{lsm}}(X)$ extending the sheaf quantization. More precisely, if two Cauchy sequences in $\frakL(\bm{L}, \bm{b})$ define the same limit object, the corresponding limits of sheaf quantizations are also isomorphic.
\end{theorem}

\begin{proof}
Note that all the Cauchy colimits along Hamiltonian isotopies of smooth sheaf quantizations exist in $\mu^{\mathrm{lsm}}(X)$. The explanation is as follows: Let $\lb L_n\rb_{n\in \bZ_{\geq 0}}$ be such a sequence. For $N$ large, $\lb L_n\rb_{n\geq N}$ moves in a $C^0$-small way when $n$ moves. Hence we can take a cover $\lc U_i\rc_i$ of $X$, where $L_n$ has a representative in each cover $U_i$ for any $n$. Hence the colimit exists. By using the argument of \cite{GV2022singular, AsanoIkecomplete}, we can deduce that the Cauchy colimit exists.

Let $((\bm{L}_i, \bm{b}_i))_i$ and $((\bm{L}'_i, \bm{b}'_i))_i$ be Cauchy sequences in $\frakL(\bm{L}, \bm{b})$ that define the same limit. 
We let $\cE$ and $\cE'$ be the corresponding sheaf quantizations, respectively. 
Then it follows that $d_I(\cE, \cE')=0$. 

We clarify the situation a little more. 
For each $i$, there exist compactly supported Hamiltonian diffeomorphisms $\varphi_i$ and $\varphi'_i$ such that
\begin{enumerate}
    \item[(1)] $(\bm{L}_i, \bm{b}_i) \sim_{\varphi_i} (\bm{L}_{i+1}, \bm{b}_{i+1})$, $(\bm{L}'_i, \bm{b}'_i) \sim_{\varphi'_i} (\bm{L}'_{i+1}, \bm{b}'_{i+1})$ and
    \item[(2)] $\|\varphi_i\|_H, \|\varphi_i'\|_H \rightarrow 0$ as $i\rightarrow \infty$.
\end{enumerate}
We fix such choices of Hamiltonians. 
Associated with the choice, we obtain
    \begin{enumerate}
    \renewcommand{\labelenumi}{$\mathrm{(\roman{enumi})}$}
    \item a sequence of sheaf quantizations $\cE_i$ (resp.\ $\cE'_i$) of $(\bm{L}_i, \bm{b}_i)$ (resp.\ $(\bm{L}'_i, \bm{b}'_i)$), 
        \item sequences of closed morphisms $f_{i, i+1}\in \Hom^0(\cE_i, \cE_{i+1})$ and $g_{i, i+1}\in\Hom^0(\cE_{i+1}, \cE_{i})$ (resp.\ $f'_{i, i+1}\in \Hom^0(\cE'_i, \cE'_{i+1})$ and $g'_{i, i+1}\in\Hom^0(\cE'_{i+1}, \cE'_{i})$) such that $f_{i, i+1}$ and $g_{i, i+1}$ (resp.\ $f'_{i, i+1}$ and $g'_{i, i+1}$) are mutually $\epsilon_i$-inverse (resp.\ $\epsilon'_i$-inverse) with $\epsilon_i, \epsilon'_i\rightarrow 0$, and 
       \item the colimit of $(\cE_i, f_{i, i+1})$ (resp.\ $(\cE'_i, f_{i, i+1}')$) is isomorphic to $\cE$ (resp.\ $\cE'$).
    \end{enumerate}
We denote the associated morphism $\cE_i\rightarrow \cE$ (resp.\ $\cE'_i \rightarrow \cE'$) by $f_i$ (resp.\ $f_i'$). 
By taking a subsequence if necessary, we can further assume:
\begin{enumerate}
    \item[(iv)] there exists a closed morphism $g_i$ (resp.\ $g_i'$) which is $\eta_i$-inverse (resp.\ $\eta_i'$-inverse) of $f_i$ (resp.\ $f_i'$) with $\eta_i, \eta_i'\rightarrow 0$. 
\end{enumerate}

Let us now construct closed morphisms between $\cE$ and $\cE'$. 
Since $d_I(\cE,\cE')=0$, we have $\alpha_1\colon \cE\rightarrow \cE'$ and $\beta_1\colon \cE'\rightarrow \cE$ that are mutually $\rho_1$-inverse. 
Take a positive number $\rho_2$ less than $\rho_1/2$. 
We also have $\alpha_2\colon \cE\rightarrow \cE'$ and $\beta_2\colon \cE'\rightarrow \cE$ that are mutually $\rho_2$-inverse as well. We also note that $\alpha_2$ is a scalar multiple of $\alpha_1$ by the construction over the Novikov field, since $\Hom^0_{\Lambda}(\cE, \cE')\cong \Lambda$ by the assumption of the almost isomorphicity between  $\End^0(\cE_i)$ and $\Lambda_0\id$.

Repeating the arguments, we find that there exists a sequence of closed morphisms $\alpha_j \colon \cE\rightarrow \cE'$ and $\beta_j \colon \cE'\rightarrow \cE$ in $\mu(X, u<\infty)$ such that $\beta_j \circ \alpha_j =T^{\rho_j}, \alpha_j \circ \beta_j = T^{\rho_j}$ with $\rho_j \rightarrow 0$ as $j \rightarrow \infty$ and the subspace of $\Hom_{\Lambda}(\cE, \cE')$ spanned by $\alpha_j$ does not depend on $j$. 

For each $j$, we set
\begin{equation}
    \begin{split}
        \theta_\alpha^j &\coloneqq \max\lc\theta\relmid T^{-\theta}\alpha_j \in F^0\Hom(\cE, \cE')\rc \\
        \theta_\beta^j &\coloneqq \max\lc \theta\relmid T^{-\theta}\beta_j \in F^0\Hom(\cE', \cE)\rc.
\end{split}
\end{equation}
Then by the the one-dimensionality of the hom-spaces, we have 
\begin{equation}
    \theta_\alpha^j+\theta_\beta^j=\rho_j.
\end{equation}

If necessary, we take a subsequence of $\alpha_j$ so that $\theta_\alpha^j\rightarrow 0$ as $j\rightarrow \infty$. 
We again denote the subsequence by $(\alpha_j)_j$. 
We similarly take such a subsequence for $\beta_j$. 
Note that $\alpha_j=T^{\theta_\alpha^j-\theta_\alpha^{j'}}\alpha_{j'}$ if $j'>j$.

We set $\alpha\coloneqq T^{-\theta_\alpha^j}\alpha_j \in F^0\Hom_{\Lambda}(\cE, \cE')$ and $\beta\coloneqq T^{-\theta_\beta^j} \beta_j \in F^0\Hom_{\Lambda}(\cE', \cE)$ (which does not depend on $j$).
Then we find that $\alpha\circ \beta=\id$ and $\beta\circ \alpha=\id$. 
By Lemma~\ref{lem:existence_lift_alpha_beta} below, we can take $\widetilde{\alpha}\in \Hom_{\Lambda_0}(\cE, \cE')$ and $\widetilde{\beta}\in \Hom_{\Lambda_0}(\cE', \cE)$ that lift $\alpha$ and $\beta$, respectively, and satisfy $T^{\theta_\alpha^j}\widetilde{\alpha}=\alpha_j, T^{\theta_\beta^j}\widetilde{\beta}=\beta_j$ for any $j$.
Then, we have $T^{\epsilon}\widetilde{\alpha}\circ \widetilde{\beta}=T^{\epsilon}\id, T^{\epsilon}\widetilde{\beta}\circ \widetilde{\alpha}=T^{\epsilon}\id$ for any $\epsilon>0$. By running the argument of \cite[Lem.~B.7]{GV2022singular},  we complete the proof of \cref{thm:limitSQ}.
\end{proof}

\begin{lemma}\label{lem:existence_lift_alpha_beta}
    In the situation of the proof of \cref{thm:limitSQ}, there exist $\widetilde{\alpha}\in \Hom_{\Lambda_0}(\cE, \cE')$ and $\widetilde{\beta}\in \Hom_{\Lambda_0}(\cE', \cE)$ that lift $\alpha$ and $\beta$, respectively, and satisfy $T^{\theta_\alpha^j}\widetilde{\alpha}=\alpha_j, T^{\theta_\beta^j}\widetilde{\beta}=\beta_j$ for any $j$.
\end{lemma}
\begin{proof}
For $\cE, \cE'\in \mu(X,u<c)$, there exists an object $\cF_{\cE, \cE'} \in \Sh^\bR_{\tau>0}(\bR_t)$ such that 
    \begin{equation}
        \Hom_{\Lambda_0}(\cE, \cE')\cong \Hom_{\Sh(\bR_t)}(\bK_{[a,\infty)},\cF_{\cE, \cE'})
    \end{equation}
    for any $a\in \bR$. 
    For example, this follows from the main result of \cite{KuwNov}.
    Then the action of $T^a$ on $\Hom_{\Lambda_0}(\cE, \cE')$ is given by 
    \begin{equation}
        \Hom_{\Sh(\bR_t)}(\bK_{[a,\infty)}, \cF_{\cE, \cE'})\rightarrow \Hom_{\Sh(\bR_t)}(\bK_{[0,\infty)}, \cF_{\cE, \cE'})
    \end{equation}
    induced by $\bK_{[0, \infty)}\rightarrow \bK_{[a, \infty)}$. 

    We regard $\alpha_{j}$ as 
    \begin{equation}
        \alpha_{j} \in \Hom_{\Sh(\bR_t)}(\bK_{[-\theta_{\alpha}^{j},\infty)}, \cF_{\cE, \cE'}).
    \end{equation}
    For $j'>j$, we have a morphism
    \begin{equation}
       \Hom_{\Sh(\bR_t)}(\bK_{[-\theta_{\alpha}^{j'},\infty)}, \cF_{\cE, \cE'}) \rightarrow \Hom_{\Sh(\bR_t)}(\bK_{[-\theta_{\alpha}^{j},\infty)}, \cF_{\cE, \cE'}),
    \end{equation}
    which maps $\alpha_{j'}$ to $\alpha_{j}$. Hence $(\alpha_{j})_{j}$ forms a direct system of elements in the direct system $(\Hom_{\Sh(\bR_t)}(\bK_{[-\theta_{\alpha}^{j},\infty)}, \cF_{\cE, \cE'}))_{j}$. Hence it defines an element
    \begin{equation}
    \begin{split}
        \widetilde{\alpha}=\varprojlim_{j \rightarrow \infty}\alpha_{j} &\in \varprojlim_{j \rightarrow \infty} \Hom_{\Sh(\bR_t)}(\bK_{[-\theta_{\alpha}^{j},\infty)}, \cF_{\cE, \cE'})\\
        &\cong \Hom_{\Sh(\bR_t)}(\varinjlim_{j \rightarrow \infty}\bK_{[-\theta_{\alpha}^{j},\infty)}, \cHom^{\star\bR}(\cE, \cE'))\\ 
        &\cong \Hom_{\Sh(\bR_t)}(\bK_{[0,\infty)}, \cF_{\cE, \cE'}),
    \end{split}
    \end{equation}
    which is our desired lift. 
    Indeed, we have 
    \begin{equation}
    \begin{aligned}
        \Hom_{\Sh(\bR_t)}(\bK_{[0,\infty)}, \cF_{\cE, \cE'}) & \rightarrow \Hom_{\Sh(\bR_t)}(\bK_{[-\theta_{\alpha}^j,\infty)}, \cF_{\cE, \cE'}); \\
        \widetilde{\alpha} & \mapsto T^{\theta_\alpha^j}\widetilde{\alpha}=\alpha_j
    \end{aligned}
    \end{equation}
    by the definition of the limit. We can prove similarly for $\beta$.
\end{proof}

\begin{remark}
One can carry out the argument of \cite{AGHIV} in this setup, by replacing the $\gamma$-metric with the Hofer metric. 
Then one can conclude that the Hofer-version of the $\gamma$-support of an element of $\widehat{\cL}(\bm{L}, \bm{b})$ is the closure of $\musupp$ of the corresponding sheaf quantization.
\end{remark}

\appendix

\section{Curved dga and \texorpdfstring{$A_\infty$}{Ainfty}-algebra}\label{sec:curved_dga}

We briefly recall the notion of curved dga and $A_\infty$-algebra.

\begin{definition}
    A \emph{curved dga} is a pair $(A, d)$, where  
    \begin{enumerate}
        \item $A$ is a graded $\Lambda_0$-algebra and
        \item $d$ is $\Lambda_0$-linear morphism of degree 1
    \end{enumerate}
    such that $A$ is free as $\Lambda_0$-module and $d^2\otimes_{\Lambda_0}\Lambda_0/\Lambda_0^+=0$.
\end{definition}
A curved $A_\infty$-algebra (=filtered $A_\infty$-algebra in \cite{FOOO}) is a generalization of curved dga. 

\begin{definition}
    A morphism of curved dga is a morphism of graded $\Lambda_0$-algebras compatible with $d$. 
    A morphism of curved dga is said to be a \emph{quasi-isomorphism} if it induces a quasi-isomorphism over $\Lambda_0/\Lambda_0^+$.
\end{definition}

\begin{theorem}[{$A_\infty$-Whitehead theorem~\cite{FOOO}}]
    Suppose $A, B$ are countably generated curved $A_\infty$-algebras. A quasi-isomorphism has a homotopy inverse. 
\end{theorem}

\begin{definition}
    Let $A$ be a curved dga/$A_\infty$-algebra. 
    A Maurer--Cartan element is an element $\bm{b}$ of $A$ such that it satisfies the Maurer--Cartan equation
    \begin{equation}
        \sum_{i\geq 0}m_i(\bm{b}^{\otimes i})=0,
    \end{equation}
    where $m_i$ is the $i$-th operation of the curved $A_\infty$-structure of $A$. 
\end{definition}

Let $A_1, A_2$ be a curved dga and $f\colon A_1\rightarrow A_2$ be a morphism between them. Let $\bm{b}$ a Maurer--Cartan element of $A$. Then $f(\bm{b})$ is a Maurer--Cartan element of $\bm{b}$.

Let $I$ be a filtered poset and $A_\bullet$ be a curved dga parameterized by $I$; for each $i\in I$, we have a curved dga $A_i$, and a morphism $\rho_{ij}\colon A_i\rightarrow A_j$ for any $i<j$. Then one can define the colimit curved dga $A$ as the quotient of the direct sum $\bigoplus_{i\in I} A_i$ by the relations induced by $\rho_{ij}$'s. Then a Maurer--Cartan element of $A_i$ for some $i\in I$ defines a Maurer--Cartan element of~$A$.

\section{Almost terminology}\label{section:almost}
Here we recall some ``almost" terminology used in the body of the paper.

\subsection{Almost zero module}
As observed in \cite{KuwNov}, there are certain bugs in the sheaf side in the higher cohomology of equivariant sheaves. To ignore such bugs, we use almost mathematics.
In this section, we set $\bG=\bR$, which does not loss any generality and we set $\Lambda_0\coloneqq \Lambda_0^\bR$.

\begin{definition}
\begin{enumerate}
    \item An object $M\in \Mod(\Lambda_0)$ is \emph{almost zero} if $M\otimes_{\Lambda_0}\Lambda_0^+\cong 0$. We denote the full subcategory of the derived category of $\Lambda_0$-modules $\Mod(\Lambda_0)$ consisting of almost zero modules by $\Sigma$. We denote by $\Sigma$ the subcategory of the almost zero modules.
    \item A morphism $f\colon M\rightarrow N\in \Mod(\Lambda_0)$ is an \emph{almost isomorphism} if $\Cone(f)\in \Sigma$. 
\end{enumerate}
\end{definition}
\begin{definition}
    An element $f\in M\in \Mod(\Lambda_0)$ defines a morphism $e_f \colon \Lambda_0\rightarrow M$. We say $f, g\in M$ are almost the same if $e_f, e_g$ are the same in $\Mod(\Lambda_0)/\Sigma$.
\end{definition}

One way to ignore almost zero modules is by considering the quotient functor: $\fraka\colon \Mod(\Lambda_0)\rightarrow \Mod(\Lambda_0)/\Sigma$, which is called almostization. If $M, N\in \Mod(\Lambda_0)$ are isomorphic after applying $\fraka$, we say they are almost isomorphic. It is equivalent to have a chain of almost isomrophisms $M\leftarrow M_1\rightarrow M_2\leftarrow M_3\rightarrow\cdots \rightarrow N$. 

Let $\cC$ be a category enriched over $\Lambda_0$. We have the Yoneda embedding and the almostization:
\begin{equation}
    \cC\xrightarrow{\cY} \Fun(\cC, \Lambda_0)\xrightarrow{\fraka} \Fun(\cC, \Mod(\Lambda_0)/\Sigma)
\end{equation}
where $\Fun$ means the functor category.

\begin{definition}
    We say $\cE, \cF\in \cC$ are \emph{almost isomorphic} if $\fraka(\cY(\cE))\cong \fraka(\cY(\cF))$. 
\end{definition}

We use the following lemma in the body of the paper.
\begin{lemma}[\cite{KuwNov}]
    Suppose $f,g\colon \cE\rightarrow \cF\in \cC$ are almost the same elements in $\Hom(\cE, \cF)$. Then $\Cone(f)$ and $\Cone(g)$ are almost isomorphic.
\end{lemma}

We also have the following notion, which we use to formulate our conjectures stated in Introduction.
\begin{definition}
    Let $F\colon \cC\rightarrow \cD$ be $\Lambda_0$-linear functor between categories enriched over $\Lambda_0$.
    \begin{enumerate}
        \item We say $F$ is almost fully faithful (or, almost embedding) if the morphisms between hom-spaces are almost isomorphisms.
        \item We say $F$ is almost essentially surjective if, for any object $\cE\in \cD$, there exists $\cF\in \cC$ such that $F(\cF)$ is almost isomorphic to $\cE$.
        \item We say $F$ is almost equivalence if it is almost fully faithful and almost essentially surjective.
    \end{enumerate}
\end{definition}

\subsection{Almost nonzero part}
There is another way to ignore almost zero part, in the finite situation. To introduce it, we recall some ring-theoretic properties of the Novikov rings, which are easy.

\begin{proposition}
    \begin{enumerate}
        \item $\Lambda_0^\bG$ is an integral domain.
        \item $\Lambda_0^\bG$ is a valuation ring. In particular, a local ring.
        \item If $\bK$ is algebraically closed and $\bG$ is divisible, $\Lambda_0^\bG$ is also algebraically closed.
        \item $\Lambda_0^\bG$ is not Noether for a non-discrete $\bG$. In particular, $\Lambda_0^\bG$ is not principal ideal domain (PID) for a non-discrete $\bG$.
    \end{enumerate}
\end{proposition}

Although, $\Lambda_0^\bG$ is not PID, we can develop a kind of elementary divisor theory.
\begin{proposition}
    \begin{enumerate}
        \item Any ideal of $\Lambda_0^\bG$ is isomorphic to $\Lambda_0$ or $\Lambda_0^+$ as a $\Lambda_0^\bG$-module.
        \item Any finitely generated module over $\Lambda_0^\bG$ is isomorphic to a module of the form
        \begin{equation}
           \bigoplus_{i=1}^j(\Lambda_0/T^{c_i}\Lambda_0)^{k_i}\oplus (\Lambda_0/\Lambda_0^+)^{k_{j+1}}
        \end{equation}
        where $c_i\in \bR_{> 0}\cup\{+\infty\}$.
    \end{enumerate}
\end{proposition}
\begin{proof}
    1. Easy. 2. One can transplant the usual proof of the elementary divisor theory in this setup using 1.
\end{proof}
Let $M$ be a finite generated $\Lambda_0$-module. Note that the second direct summand (the almost zero part) in the above proposition can be written as 
\begin{equation}
    M_0 \coloneqq \lc m\in M\relmid \Lambda_0^+m=0\rc.
\end{equation}
We set
\begin{equation}
    M_a \coloneqq M/M_0
\end{equation}
for any $\Lambda_0$-module. Of course, $(M_a)_0=0$.

\printbibliography

\noindent Yuichi Ike:
Graduate School of Mathematical Sciences, The University of Tokyo, 3-8-1 Komaba Meguro-ku Tokyo 153-8914, Japan.

\noindent
\textit{E-mail address}: \texttt{ike[at]ms.u-tokyo.ac.jp}, \texttt{yuichi.ike.1990[at]gmail.com}
\medskip

\noindent Tatsuki Kuwagaki: 
Department of Mathematics, Kyoto University, Kitashirakawa Oiwake-cho, Sakyo-ku, Kyoto 606-8502, Japan.

\noindent 
\textit{E-mail address}: \texttt{tatsuki.kuwagaki.a.gmail.com}

\end{document}